\documentclass[smallextended,final]{svjour3}     
\usepackage[backref]{hyperref}
\usepackage[letterpaper,top=2in, bottom=1.5in, left=1in, right=1in]{geometry}

\smartqed  

\usepackage{tikz}
\usetikzlibrary{fit}\usetikzlibrary{arrows,shapes,positioning}
\usetikzlibrary{decorations.markings}
\tikzstyle arrowstyle=[scale=1]
\tikzstyle directed=[postaction={decorate,decoration={markings,
    mark=at position .5 with {\arrow[arrowstyle]{stealth}}}}]
\tikzstyle reverse directed=[postaction={decorate,decoration={markings,
    mark=at position .5 with {\arrowreversed[arrowstyle]{stealth};}}}]

\spnewtheorem{thm}{Theorem}[section]{\bf}{\it}
\spnewtheorem{cor}[thm]{Corollary}{\bf}{\it}
\spnewtheorem{lem}[thm]{Lemma}{\bf}{\it}
\spnewtheorem{prop}[proposition]{Proposition}{\bf}{\it}
\spnewtheorem{defi}[thm]{Definition}{\bf}{\it}
\spnewtheorem{que}[thm]{Question}{\bf}{\it}
\spnewtheorem{rem}[thm]{Remark}{\bf}{\it}
\spnewtheorem{prob}[thm]{Problem}{\bf}{\it}

\usepackage{mathtools}
\usepackage[normalem]{ulem} 

\newcommand{\vx}{{\mathbf{x}}}
\newcommand{\vy}{{\mathbf{y}}}

\newcommand{\vN}{{\mathbf{N}}}

\newcommand{\vR}{{\mathbf{R}}}

\newcommand{\cD}{{\mathcal{D}}}

\newcommand{\cG}{{\mathcal{G}}}
\newcommand{\cH}{{\mathcal{H}}}

\newcommand{\cL}{{\mathcal{L}}}

\newcommand{\cN}{{\mathcal{N}}}

\newcommand{\sign}{\mathrm{sign}}

\newcommand{\St}{{\mathrm{subject~to}}} 
\newcommand{\dom}{{\mathrm{dom}}} 
\newcommand{\range}{{\mathrm{range}}} 

\newcommand{\prox}{\mathbf{prox}}

\newcommand{\tnabla}{\widetilde{\nabla}}

\newcommand{\TPRS}{T_{\mathrm{PRS}}}
\newcommand{\best}{\mathrm{best}}
\newcommand{\kbest}{k_{\best}}

\DeclareMathOperator*{\argmin}{arg\,min}

\DeclareMathOperator*{\Min}{minimize}

\DeclareMathOperator*{\Fix}{Fix}
\DeclareMathOperator*{\zer}{zer}

\newcommand{\bc}{\begin{center}}
\newcommand{\ec}{\end{center}}

\newcommand{\bdm}{\begin{displaymath}}
\newcommand{\edm}{\end{displaymath}}

\newcommand{\beq}{\begin{equation}}
\newcommand{\eeq}{\end{equation}}

\newcommand{\bfl}{\begin{flushleft}}
\newcommand{\efl}{\end{flushleft}}

\newcommand{\bt}{\begin{tabbing}}
\newcommand{\et}{\end{tabbing}}

\newcommand{\beqn}{\begin{align}}
\newcommand{\eeqn}{\end{align}}

\newcommand{\beqs}{\begin{align*}} 
\newcommand{\eeqs}{\end{align*}}  

\usepackage[lined,boxed,commentsnumbered, ruled,vlined]{algorithm2e}

\newcommand\numberthis{\addtocounter{equation}{1}\tag{\theequation}}

\DeclarePairedDelimiter{\dotp}{\langle}{\rangle}
 \DeclarePairedDelimiter{\ceil}{\lceil}{\rceil}

\newcommand{\refl}{\mathbf{refl}}
\newcommand{\TFBS}{T_{\mathrm{FBS}}}
\usepackage{booktabs}
\spnewtheorem{assump}{Assumption}{\bf}{\it}
\def\cut#1{{}}
\usepackage{color}
\definecolor{blue}{rgb}{0,0,0}

\begin{document}

\title{Convergence rate analysis of several splitting schemes}
\author{Damek Davis   \and
        Wotao Yin}

\institute{D. Davis \and W. Yin\at
              Department of Mathematics, University of California, Los Angeles\\
              Los Angeles, CA 90025, USA\\
              \email{damek / wotaoyin@ucla.edu}           
}

\date{Received: date / Accepted: date}

\maketitle

\abstract{
Splitting schemes are a class of powerful algorithms that solve complicated monotone inclusions and convex optimization problems that are built from many simpler pieces. They give rise to algorithms in which the simple pieces of the decomposition are processed individually. This leads to easily implementable and highly parallelizable algorithms, which often obtain nearly state-of-the-art performance.  

In the first part of this paper, we analyze the convergence rates of several general splitting algorithms and provide examples to prove the tightness of our results.  The most general rates are proved for the \emph{fixed-point residual} (FPR) of the  Krasnosel'ski\u{i}-Mann (KM) iteration of nonexpansive operators, where we improve the known big-$O$ rate to little-$o$. We show the tightness of this result and improve it in several special cases. In the second part of this paper, we use the convergence rates derived for the KM iteration to analyze the \emph{objective error} convergence rates for the Douglas-Rachford (DRS), Peaceman-Rachford (PRS), and ADMM splitting algorithms under general convexity assumptions.  We show, by way of example, that the rates obtained for these algorithms are tight in all cases and obtain the surprising statement: The DRS algorithm is nearly as fast as the proximal point algorithm (PPA) in the ergodic sense and nearly as slow as the subgradient method in the nonergodic sense.  Finally, we provide several applications of our result to feasibility problems, model fitting, and distributed optimization. Our analysis is self-contained, and most results are deduced from a basic lemma that derives convergence rates for summable sequences, a simple diagram that decomposes each relaxed PRS iteration, and fundamental inequalities that relate the FPR to objective error.
}

\keywords{Krasnosel'ski\u{i}-Mann algorithm \and Douglas-Rachford Splitting \and Peaceman-Rachford Splitting \and Alternating Direction Method of Multipliers \and nonexpansive operator \and averaged operator \and fixed-point algorithm \and little-o convergence}
\subclass{47H05 \and 65K05 \and 65K15 \and 90C25}

\section{Introduction}

Operator-splitting and alternating-direction methods have a long history, and they have been, and still are, some of the most useful methods in scientific computing. These algorithms solve problems composed of several competing structures, such as finding a point in the intersection of two sets, minimizing the sum of two functions, and, more generally, finding a zero of the sum of two monotone operators. They give rise to algorithms that are simple to implement and converge quickly in practice. 
Since the 1950s, operator-splitting methods have been applied to solving partial differential equations (PDEs) and feasibility problems. Recently, certain operator-splitting methods such as ADMM (for alternating direction methods of multipliers) \cite{gabay1976dual,GlowinskiADMM} and Split Bregman \cite{goldstein2009split} have  found new applications in (PDE and non-PDE related) image processing, statistical and machine learning, compressive sensing, matrix completion,  finance, and control. They have also been extended to handle distributed and decentralized optimization (see \cite{boyd2011distributed,shi2013linear,wei2012distributed}).

In convex optimization, operator-splitting methods split constraint sets and objective functions  into   subproblems that are easier to solve than the original problem. Throughout this paper, we will consider two prototype optimization problems: We analyze the unconstrained problem
\begin{align}\label{eq:simplesplit}
\Min_{x\in \cH}~ f(x) + g(x) 
\end{align}
where $\cH$ is a Hilbert space. In addition, we analyze the linearly constrained variant
\begin{align*}
\Min_{x \in \cH_1,~ y \in \cH_2} & \; f(x) + g(y) \\
\St~ &  \; Ax + By = b \numberthis \label{eq:simplelinearconstrained}
\end{align*}
where $\cH_1, \cH_2$, and $\cG$ are Hilbert spaces, the vector $b$ is an element of $\cG$, and $A : \cH_1 \rightarrow \cG$ and $B : \cH_2 \rightarrow \cG$ are bounded linear operators. Our working assumption throughout the paper is that the subproblems involving $f$ and $g$ separately are much simpler to solve than the joint minimization problem.

Problem~\eqref{eq:simplesplit} is often used to model tasks in signal recovery that enforce prior knowledge of the form of the solution, such as sparsity, low rank, and smoothness \cite{combettes2011proximal}.  The knowledge-enforcing function, also known as the regularizer, often has properties that make it difficult to jointly optimize with the remaining parts of the problem. Therefore, operator-splitting methods become the natural choice. 

Problem~\eqref{eq:simplelinearconstrained} is often used to model tasks in machine learning, image processing, and distributed optimization. For special choices of $A$ and $B$, operator-splitting schemes naturally give rise to algorithms with parallel or distributed implementations \cite{bertsekas1989parallel,boyd2011distributed}. Because of the flexibility of operator-splitting algorithms, they have become a standard tool that addresses the emerging need for computational approaches to analyze a massive amount of data in a fast, parallel, distributed, or even real-time manner. 
\subsection{Goals, challenges, and approaches}
This work seeks to improve the theoretical understanding of the most well-known operator-splitting algorithms including Peaceman-Rachford splitting (PRS), Douglas-Rachford splitting (DRS), the alternating direction method of multipliers (ADMM), as well as their relaxed versions. (The proximal point algorithm (PPA) and forward-backward splitting (FBS) are covered in some limited aspects too.) When applied to convex optimization problems, they  are known to converge under rather general conditions.  However, their  \emph{objective error} convergence rates are largely unknown with only a few exceptions \cite{beck2009fast,wei2012distributed,boct2013convergence,chambolle2011first,ShefiTeboulle2014}. {Among the convergence rates known in the literature,  many are given in terms of {quantities that do not immediately imply the objective error rates}, such as the \emph{fixed-point residual} (FPR)~\cite{he2012non,cominetti2012rate,liang2014convergence} (the squared distance between two consecutive iterates). }
{Furthermore, almost all (with the exception of~\cite{beck2009fast} {and a part of \cite{ShefiTeboulle2014}}) analyze the objective error or variational inequalities evaluated at the \textit{time-averaged}, or \textit{ergodic}, iterate, rather than the last, or \textit{nonergodic}, iterate generated by the algorithm~{\color{blue}\cite{chambolle2011first,boct2013convergence,monteiro2013iteration,MSgeneral,MSFBF}}. In applications, the nonergodic iterate tends to share structural properties with the solution of the problem, such as sparsity in $\ell_1$ minimization or low-rankness in nuclear norm minimization. In contrast, the ergodic iterates tend to ``average out" these properties in the sense that the average of many sparse vectors can be dense. Thus, part of the purpose of this work is to illustrate the theoretical differences between these two iterates and to justify the use of nonergodic iterates in practice.} 

Unlike the well-developed complexity estimates for (sub)gradient methods \cite{nemirovskyproblem}, there are no known \emph{lower complexity} results for most \emph{strictly primal} or \emph{strictly dual }splitting algorithms. (As an exception, lower complexity results 
are known for the \emph{primal-dual} case \cite{chen2013optimal}.) {\color{blue} In particular, the classical complexity analysis given in \cite{nemirovskyproblem,nesterov2004introductory} does not apply to the algorithms addressed in this paper.} This paper attempts to close this gap.  The convergence rates in terms of FPR and objective errors are derived for operator-splitting algorithms applied to Problem \eqref{eq:simplesplit}. Convergence rates for constraint violations, the primal objective error, and the dual objective error are derived for ADMM, which applies to Problem  \eqref{eq:simplelinearconstrained}. Some of the derived rates are convenient to use, for example, to determine how many iterations are needed to reach a certain accuracy,  to decide when to stop an  algorithm, and to compare an algorithm to others in terms of their worst-case complexities. 

The techniques we develop in this paper are quite different from those used in classical optimization convergence analysis, mainly because splitting algorithms are driven by  fixed-point operators instead of driven by minimizing objectives. (An exception is FBS, which is driven by both.)  Splitting algorithms are fixed-point iterations derived from certain optimality conditions, and they converge due to the contraction of the  fixed-point operators. Some of them do not even  reduce objectives monotonically. Thus, objective convergence is a consequence of operator convergence rather than the cause of it. Therefore, we first perform an operator-theoretic analysis and then, based on these results, derive optimization related rates.

We now describe our contributions and our techniques as follows: 
\begin{itemize}
\item We show that the FPR of the fixed-point iterations of nonexpansive operators converge with rate $o(1/(k+1))$ (Theorem~\ref{prop:averagedconvergence}). This rate is optimal and improves on the known big-$O$ rate \cite{cominetti2012rate,liang2014convergence}. For the special cases of FBS applied to Problem \eqref{eq:simplesplit} and one-dimensional DRS,  we improve this rate to $o(1/(k+1)^2)$. In addition, we provide examples (Section~\ref{sec:optimalFPR}) to show that all of these rates are tight.  Specifically, for each rate $o(1/(k+1)^p)$, we give an example with rate $\Omega(1/(k+1)^{p+\varepsilon})$ for any $\varepsilon >0$. A detailed list of our contributions and a comparison with existing results appear in Section~\ref{sec:optimalitydisc}. The analysis is based on establishing summable and, in many cases, monotonic sequences, whose  convergence rates are summarized in Lemma \ref{lem:sumsequence}.

\item We demonstrate that even when the DRS algorithm converges in norm to a solution, it may do so \emph{arbitrarily slowly} (Theorem~\ref{thm:arbitrarilyslow}).

\item We give the objective convergence rates of the relaxed PRS algorithm and show that it is, in the worst case, nearly \emph{as slow as the subgradient method} (Theorems~\ref{thm:drsnonergodic} and~\ref{thm:DRSobjectiveslow}), yet \emph{nearly as fast as PPA} in the ergodic sense (Theorems~\ref{thm:drsergodic} and~\ref{thm:ppaslow}). The rates are obtained by relating  the objective error to  the aforementioned \emph{FPR rates} through a \emph{fundamental inequality} (Proposition~\ref{prop:DRSupper}).  {\color{blue} Note that we prove ergodic and nonergodic convergence rates and show that these rates are} sharp through several examples (Section~\ref{sec:PRSergodicoptimality}).
\item We give the convergence rates of the primal objective and feasibility of the current iterates generated by ADMM. Our analysis follows by a simple application of the Fenchel-Young inequality (Section~\ref{sec:DRSADMM})
\end{itemize}

\subsection{Notation}
In what follows,   $\cH, \cH_1, \cH_2, \cG$ denote (possibly infinite dimensional) Hilbert spaces. In fixed-point iterations, $(\lambda_j)_{j \geq 0} \subset \vR_+$ will denote a sequence of relaxation parameters and $$\Lambda_k := \sum_{i=0}^k \lambda_i$$ is its $k$th partial sum. To ease notational memory, the reader may assume that $\lambda_k \equiv (1/2)$ and $\Lambda_k =(k+1)/2$ in the DRS algorithm or that $\lambda_k \equiv 1$ and $\Lambda_k = (k+1)$ in the PRS algorithm. Given the sequence  $(x^j)_{j \geq 0}\subset \cH$, we let $\overline{x}^k = ({1}/{\Lambda_k})\sum_{i=0}^k \lambda_i x^i$ denote its $k$th average  with respect to the sequence $(\lambda_j)_{j \geq 0}$.  

We call a convergence result \emph{ergodic} if it is in terms of the sequence $(\overline{x}^j)_{j \geq 0}$, and \emph{nonergodic} if it is in terms of $(x^j)_{j \geq 0}$.

Given a closed, proper, convex function $f : \cH \rightarrow (-\infty, \infty]$, $\partial f(x)$ denotes its subdifferential at $x$ and
\begin{align}\label{eq:tnabla}
\tnabla f(x) \in \partial f(x),
\end{align}
denotes a subgradient, and the actual choice of the subgradient $\tnabla f(x)$ will always be clear from the context. (This notation was used in \cite[Eq. (1.10)]{bertsekas2011incremental}.) 

The convex conjugate of a proper, closed, and convex function $f$ is
\begin{align}
f^\ast(y) &:= \sup_{x \in \cH} \dotp{y, x} - f(x).
\end{align}
Let $I_{\cH}$ denote the identity map. Finally, for any $x \in \cH$ and scalar $\gamma \in \vR_{++}$, we let 
\begin{align}\label{eq:prox}
\prox_{\gamma f}(x) := \argmin_{y \in \cH} f(y) + \frac{1}{2\gamma} \|y - x\|^2 \quad \mathrm{and} \quad \refl_{\gamma f} := 2\prox_{\gamma f} - I_{\cH},
\end{align}
which are known as the \emph{proximal} and \emph{reflection} operators, and we define the PRS operator:
\begin{align}
\TPRS &:= \refl_{\gamma f} \circ \refl_{\gamma g}.
\end{align}
\subsection{Assumptions}

We list the the assumptions used throughout this papers as follows.

\begin{assump}
Every function we consider is closed, proper, and convex. 
\end{assump}
Unless otherwise stated, a function is not necessarily differentiable. 

\begin{assump}[Differentiability]
Every differentiable function we consider is Fr{\'e}chet differentiable \cite[Definition 2.45]{bauschke2011convex}.
\end{assump}

\begin{assump}[Solution existence]\label{assump:additivesub}
Functions $f, g : \cH \rightarrow (-\infty, \infty]$ satisfy
\begin{align}
\zer(\partial f + \partial g) \neq \emptyset.
\end{align}
\end{assump}
Note that this assumption is slightly stronger than the existence of a minimizer, because $\zer(\partial f + \partial g) \neq \zer(\partial (f + g))$, in general \cite[Remark 16.7]{bauschke2011convex}. Nevertheless, this assumption is standard.

\subsection{The Algorithms}

This paper covers several operator-splitting algorithms that are all based on the atomic evaluation of the \emph{proximal} and \emph{gradient} operators. By default, all algorithms start from an arbitrary $z^0 \in \cH$. To minimize a function $f$,  the \emph{proximal point algorithm} (PPA) iteratively applies the proximal operator of $f$ as follows:
\begin{equation}\label{ppaitr}
z^{k+1} = \prox_{\gamma f}(z^k),\quad k=0,1,\ldots
\end{equation}
where $\gamma>0$ is a tuning parameter. 
Another equivalent form of the iteration, which is often used in this paper,  is  
\begin{equation}\label{eq:backward}
z^{k+1} = z^k - \gamma \tnabla f(z^{k+1})
\end{equation} {where $\tnabla f(z^{k+1}) := (1/\gamma)(z^k - z^{k+1}) \in\partial f(z^{k+1})$.} Given $z^k$,  the point $z^{k+1}$ is unique and so is the subgradient $\tnabla f(z^{k+1})$ (Lemma~\ref{lem:optimalityofprox}). The iteration resembles the (sub)gradient descent iteration, which uses a (sub)gradient of $f$ at $z^k$ instead of its (sub)gradient at $z^{k+1}$.   

{In the literature, \eqref{eq:backward} is referred to as the \emph{backward} iteration, where the (sub)gradient is drawn at the destination $z^{k+1}$. On the contrary, a \emph{forward} iteration draws the (sub)gradient at the start $z^k$, resulting in the update rule: $z^{k+1} = z^k - \gamma \tnabla f(z^{k})$ for an arbitrary $\tnabla f(z^k) \in \partial f(z^k)$.} Most of the splitting schemes in this paper are built from forward, backward, and reflection operators.

In problem \eqref{eq:simplesplit}, let $g$ be a $C^1$ function with Lipschitz derivative. The \emph{forward-backward splitting} (FBS) algorithm is the iteration:
\begin{align}\label{eq:FBSiterates}
z^{k+1} = \prox_{\gamma f}(z^k - \gamma \nabla g(z^k)), \quad k = 0, 1, \ldots.
\end{align}
The FBS algorithm directly generalizes PPA and has the following subgradient representation:
\begin{align}\label{eq:FBSiterates1}
z^{k+1} = z^k - \gamma \tnabla f(z^{k+1}) - \gamma \nabla g(z^k)
\end{align}
{where $\tnabla f(z^{k+1}) := (1/\gamma)(z^k - z^{k+1} - \gamma \nabla g(z^k)) \in \partial f(z^{k+1})$, and  $z^{k+1}$ and $\tnabla f(z^{k+1})$ are unique  (Lemma~\ref{lem:optimalityofprox}) given  $z^k$ and $\gamma>0$. }

A direct application of the PPA algorithm \eqref{ppaitr} to minimizing $f+g$ would require computing the operator  $\prox_{\gamma(f+g)}$, which can be difficult to evaluate. The Douglas-Rachford splitting (DRS) algorithm eliminates this difficulty by separately evaluating the proximal operators of $f$ and $g$ as follows: 
\begin{align*}
\begin{cases}
x_g^k = \prox_{\gamma g}(z^k);\\
x_f^k = \prox_{\gamma f}( 2 x_g^k - z^k);\\
z^{k+1} = z^k +  (x_f^k - x_g^k).
\end{cases}
\quad k = 0, 1, \ldots,
\end{align*}
which has the equivalent operator-theoretic and subgradient form (Lemma~\ref{prop:DRSmainidentity}):
\begin{align*}
z^{k+1} &= \frac{1}{2}(I_{\cH} + \TPRS)(z^k) = z^k - \gamma(\tnabla f(x_f^k) + \tnabla g(x_g^k)), \quad k = 0, 1, \ldots
\end{align*}
{where $\tnabla f(x_f^k)\in\partial f(x_f^k)$ and  $\tnabla g(x_g^k)\in \partial g(x_g^k)$ (see Lemma~\ref{prop:DRSmainidentity} for their precise definitions)}. In the above algorithm, we can replace the $1/2$ average of $I_{\cH}$ and $\TPRS$ with any other weight, so in this paper we study the \emph{relaxed PRS} algorithm:

\begin{algorithm}[H]
\SetKwInOut{Input}{input}\SetKwInOut{Output}{output}
\SetKwComment{Comment}{}{}
\BlankLine
\Input{$z^0 \in \cH, ~\gamma > 0, ~(\lambda_j)_{j \geq 0}\subseteq  (0, 1]$}
\For{$k=0,~1,\ldots$}{
$z^{k+1} = (1-\lambda_k)z^k +   \lambda_k\refl_{\gamma f} \circ \refl_{\gamma g}(z^k) $\;}
\caption{{Relaxed Peaceman-Rachford splitting (relaxed PRS)}}
\label{alg:DRS}
\end{algorithm}
The special cases $\lambda_k \equiv 1/2$ and $\lambda_k \equiv 1$ are called the DRS and PRS algorithms, respectively.

The relaxed PRS algorithm can be applied to problem \eqref{eq:simplelinearconstrained}.  To this end we define the Lagrangian:
$$\cL_\gamma(x,y;w):=f(x)+g(y)-\dotp{w,Ax+By-b} + \frac{\gamma}{2}\|Ax + By - b\|^2.$$
Section~\ref{sec:DRSADMM} presents Algorithm \ref{alg:DRS} applied to the Lagrange dual of ~\eqref{eq:simplelinearconstrained}, which  reduces to the following algorithm:

\begin{algorithm}[H]
\SetKwInOut{Input}{input}\SetKwInOut{Output}{output}
\SetKwComment{Comment}{}{}
\BlankLine
\Input{$w^{-1} \in \cH, x^{-1} = 0, y^{-1} = 0, \lambda_{-1} = {1}/{2}$, $\gamma > 0, (\lambda_j)_{j \geq 0} \subseteq (0, 1]$}

\For{$k=-1,~0,\ldots$}{
$y^{k+1} = \argmin_{y} \cL_\gamma (x^k,y;w^k) + \gamma(2\lambda_k - 1) \dotp{By,  (Ax^{k} + By^{k} -b)}$\;
$w^{k+1} = w^{k} - \gamma (Ax^{k} +  By^{k+1} - b) - \gamma(2\lambda_k - 1)(Ax^{k} + By^{k} - b)$\;
$x^{k+1} = \argmin_{x} \cL_\gamma(x,y^{k+1};w^{k+1})$\;}
\caption{{Relaxed alternating direction method of multipliers (relaxed ADMM)}}
\label{alg:ADMM}
\end{algorithm}
If $\lambda_k\equiv 1/2$, Algorithm \ref{alg:ADMM} recovers the standard ADMM.

Each of the above algorithms is a special case of the Krasnosel'ski\u{i}-Mann (KM) iteration~\cite{krasnosel1955two,mann1953mean}. {An \textit{averaged operator} is the average of a nonexpansive operator $T : \cH \rightarrow \cH$ and the identity mapping $I_\cH$. In other words, for all $\lambda \in (0, 1)$, the operator
\begin{equation}\label{eq:taveraged}
T_\lambda := (1-\lambda) I_{\cH} + \lambda T
\end{equation} 
is called $\lambda$-\textit{averaged} and every $\lambda$-averaged operator is exactly of the form $T_\lambda$ for some nonexpansive map $T$.}

Given a nonexpansive map $T$, the fixed-point iteration of the map $T_\lambda$ is called the KM algorithm:

\begin{algorithm}[H]
\SetKwInOut{Input}{input}\SetKwInOut{Output}{output}
\SetKwComment{Comment}{}{}
\BlankLine
\Input{$z^0 \in \cH, (\lambda_j)_{j \geq 0} \subseteq (0, 1]$}
\For{$k=0,~1,\ldots$}{ $z^{k+1} = T_{\lambda_k}(z^k)$\;}
\caption{{Krasnosel'ski\u{i}-Mann (KM)}}
\label{alg:KM}
\end{algorithm}

\subsection{Basic properties of averaged operators}\label{sec:nonexpansive}
This section describes the basic properties of proximal, reflection, nonexpansive, and averaged operators.  We demonstrate that proximal  and reflection operators are nonexpansive maps, and that averaged operators have a  contractive property. These properties are included in textbooks such as \cite{bauschke2011convex}.

\begin{lemma}[Optimality conditions of $\prox$]\label{lem:optimalityofprox}
Let $x \in \cH$.  Then $x^+ = \prox_{\gamma f} (x)$ if, and only if, $({1}/{\gamma})(x-x^+) \in \partial f(x^+)$.
\end{lemma}

It is straightforward to use Lemma~\ref{lem:optimalityofprox} to deduce the firm nonexpansiveness of the proximal operator.

\begin{proposition}[Firm nonexpansiveness of $\prox$]
Let $x,y \in \cH$, let $x^+ := \prox_{\gamma f}(x)$, and let $y^+ := \prox_{\gamma f}(y)$.  Then
\begin{align}\label{eq:proxfirm}
\|x^+ - y^+\|^2 &\leq \dotp{x^+ - y^+, x - y}.
\end{align}
In particular, $\prox_{\gamma f}$ is nonexpansive.
\end{proposition}

The next proposition introduces the most important operator in this paper.

\begin{proposition}[Nonexpansiveness of the PRS operator]\label{prop:TPRS}
The operator $\refl_{\gamma f} : \cH \rightarrow \cH$ is nonexpansive. Therefore, the composition is nonexpansive:
\begin{align}\label{prop:TPRS:eq:main}
\TPRS &:= \refl_{\gamma f} \circ \refl_{\gamma g}
\end{align}
\end{proposition}

The next proposition shows that averaged operators have a nice contraction property.
\begin{proposition}[Contraction property of averaged operator]\label{prop:averagedcontraction}
Let $T : \cH \rightarrow \cH$ be a nonexpansive operator. Then for all $\lambda \in (0, 1]$ and $(x, y) \in \cH \times \cH$, the averaged operator $T_{\lambda}$ defined in \eqref{eq:taveraged} satisfies
\begin{align}\label{eq:avgdecrease}
\|T_{\lambda} x - T_{\lambda} y\|^2 &\leq \|x - y\|^2 - \frac{1-\lambda}{\lambda} \|(I_{\cH} - T_{\lambda})x - (I_{\cH} - T_{\lambda})y\|^2.
\end{align}
\end{proposition}
{Note that an operator $N : \cH \rightarrow \cH$ satisfies the 
property in Equation~\eqref{eq:avgdecrease} (with $N$ in place of $T_{\lambda}$) if, and only if, it is $\lambda$-averaged.} If $ \lambda = 1/2$, then $T_\lambda$ is called firmly nonexpansive. Rearranging Equation~\eqref{eq:avgdecrease} shows that a nonexpansive operator $T$ is firmly nonexpansive, if, and only if, for all $x, y \in \cH$, the inequality holds: $$\|Tx - Ty\|^2 \leq \dotp{Tx - Ty, x-y}.$$

%

The next corollary applies Proposition~\ref{prop:averagedcontraction} to $\prox_{\gamma f}$.

\begin{corollary}[Proximal operators are ${1}/{2}$-averaged]\label{cor:proxcontraction}
The operator $\prox_{\gamma f} :  \cH \rightarrow \cH$ is ${1}/{2}$-averaged and satisfies the following contraction property:
\begin{align}\label{cor:proxcontraction:eq:main}
\|\prox_{\gamma f}(x) - \prox_{\gamma f}(y) \|^2 &\leq \|x - y\|^2 - \|(x - \prox_{\gamma f}(x)) - (y - \prox_{\gamma f}(y))\|^2.
\end{align}
\end{corollary}

The following lemma relates the fixed points of $T_{\lambda}$ to those of $T$.
\begin{lemma}\label{lem:fixedpointaverage}
Let $T : \cH \rightarrow \cH$ be nonexpansive and  $\lambda > 0$.  Then, $T_{\lambda}$ and $T$ have the same set of fixed points.
\end{lemma}

Finally, we note that the forward and forward-backward operators are averaged whenever the implicit stepsize parameter $\gamma$ is small enough. See Section~\ref{sec:fbs} for more details.

\section{Summable sequence convergence lemma}\label{sec:sequence}
This section presents a lemma on the convergence rates of nonnegative summable  sequences. Such sequences are constructed throughout this paper to establish various rates. 


\begin{lemma}[Summable sequence convergence rates]\label{lem:sumsequence}
Suppose that the nonnegative scalar sequences  $(\lambda_j)_{j\geq0}$ and $(a_{j})_{j\geq 0}$ satisfy $\sum_{i=0}^\infty \lambda_ia_i < \infty$.  Let $\Lambda_k := \sum_{i=0}^k \lambda_i$ for $k\geq 0$.
\begin{enumerate}
\item \label{lem:sumsequence:part:main} {\bf Monotonicity:} If $(a_j)_{j \geq0}$ is \emph{monotonically nonincreasing}, then
\begin{align}\label{eq:sumsequence-bigo}
a_{k} \leq \frac{1}{\Lambda_k}\left(\sum_{i=0}^\infty\lambda_i a_i \right) && and && a_{k} = o\left(\frac{1}{\Lambda_{k} - \Lambda_{\ceil{{k}/{2}}}}\right).
\end{align}
In particular,
\begin{enumerate}
\item\label{lem:sumsequence:part:main:a} if $(\lambda_j)_{j\geq 0}$ is bounded away from $0$, then $a_k=o(1/(k+1))$;
\item\label{lem:sumsequence:part:main:b} if $\lambda_k = (k+1)^p$ for some $p\ge 0$ and all $k \geq 1$, then $a_k = o(1/(k+1)^{p+1})$;
\item\label{lem:sumsequence:part:main:c} as a special case, if $\lambda_k = (k+1)$ for all $k \geq 0$, then $a_k = o(1/(k+1)^2)$.
\end{enumerate}
\item {\bf Monotonicity up to errors:} \label{lem:sumsequence:part:e} Let $(e_j)_{j \geq 0}$ be a sequence of scalars. Suppose that $a_{k+1} \leq a_k + e_k$ for all $k$ (where $e_k$ represents an error) and that  $\sum_{i=0}^\infty \Lambda_ie_i < \infty$.  Then
\begin{align}\label{eq:sumsequence-error-bigo}
a_k \leq \frac{1}{\Lambda_k} \left(\sum_{i=0}^\infty \lambda_i a_i + \sum_{i=0}^\infty \Lambda_i e_i\right) && {and} && a_{k} = o\left(\frac{1}{\Lambda_{k} - \Lambda_{\ceil{{k}/{2}}}}\right).
\end{align}
The rates of $a_k$ in Parts \ref{lem:sumsequence:part:main:a}, \ref{lem:sumsequence:part:main:b}, and \ref{lem:sumsequence:part:main:c} continue to hold  as long as $\sum_{i=0}^\infty \Lambda_ie_i < \infty$  holds. In particular, they hold if $e_k = O(1/(k+1)^q)$ for some $q > 2$, $q> p+2$, and $q> 3$, respectively.
\item \label{lem:sumsequence:part:b} {\bf Faster rates:} Suppose $(b_j)_{j\geq0}$ and $(e_j)_{j \geq 0}$ are nonnegative scalar sequences, that $\sum_{i=0}^\infty b_j < \infty$ and $\sum_{i=0}^\infty (i+1)e_i < \infty$, and that for all $k\ge 0$ we have $\lambda_k a_k \leq b_k - b_{k+1} + e_k$. Then the following sum is finite:
\begin{align}
\sum_{i=0}^\infty (i+1)\lambda_ia_i \leq \sum_{i=0}^\infty b_i + \sum_{i=0}^\infty (i+1) e_i< \infty.
\end{align}
In particular, 
\begin{enumerate}
\item\label{lem:sumsequence:part:3a} if $(\lambda_j)_{j\geq 0}$ is bounded away from $0$, then $a_k=o(1/(k+1)^2)$;
\item\label{lem:sumsequence:part:3b} if $\lambda_k = (k+1)^p$ for some $p\ge 0$ and all $k \geq 1$, then $a_k = o(1/(k+1)^{p+2})$.
\end{enumerate}
\item \label{lem:sumsequence:part:nonmono} {\bf No monotonicity:} For all $k \geq 0$, define the sequence of indices
\begin{align*}
k_{\mathrm{best}} &:= \argmin_i\{a_i | i = 0, \cdots, k\}.
\end{align*}
Then $(a_{j_{\mathrm{best}}})_{j \geq 0}$ is monotonically nonincreasing and the above bounds continue to hold when $a_k$ is replaced with $a_{\kbest}$.  
\end{enumerate}
\end{lemma}
\begin{proof}

Part \ref{lem:sumsequence:part:main}. Because $a_k\le a_i$, $\forall k\ge i$, and the inequality holds $\lambda_ia_i\ge 0$, we get the upper bound $\Lambda_k a_{k} \leq  \sum_{i=0}^{k} \lambda_i a_i \leq \sum_{i=0}^\infty \lambda_i a_i$. This shows the left part of~(\ref{eq:sumsequence-bigo}). To prove  the right part of~(\ref{eq:sumsequence-bigo}), observe that
\begin{align*}
(\Lambda_{k} - \Lambda_{\ceil{k/2}}) a_{k} &= \sum_{i=\ceil{k/2}}^k \lambda_ia_k\leq \sum_{i=\ceil{k/2}}^k \lambda_ia_i \stackrel{k\rightarrow\infty}{\rightarrow} 0.
\end{align*}

Part \ref{lem:sumsequence:part:main:a}. Let  $\underline{\lambda} := \inf_{j \geq 0}\lambda_j > 0$. For every integer $k \geq 2$, we have $\ceil{k/2} \leq (k+1)/2$. Thus, $\Lambda_k - \Lambda_{\ceil{k/2}} \geq \underline{\lambda}(k - \ceil{k/2}) \geq \underline{\lambda}(k - 1)/2 \geq \underline{\lambda}(k + 1)/6.$ Hence,  $a_k = o(1/(\Lambda_{k} -\Lambda_{\ceil{{k}/{2}}})) = o(1/(k+1))$ follows from \eqref{eq:sumsequence-bigo}.


Part \ref{lem:sumsequence:part:main:b}. For every integer $k \geq 3$, we have $\ceil{k/2} +1 \leq (k+3)/2 \leq 3(k+1)/4$ and $\Lambda_{k} - \Lambda_{\ceil{k/2}} = \sum_{i=\ceil{k/2} + 1}^k \lambda_i = \sum_{i=\ceil{k/2} + 1}^k (i+1)^p  \geq \int_{\ceil{k/2}}^k (t+1)^pdt = (p+1)^{-1}((k+1)^{p+1} - (\ceil{k/2} + 1)^{p+1}) \geq (p+1)^{-1}(1 - (3/4)^{p+1})(k+1)^{p+1}$. Therefore,  $a_k = o(1/(\Lambda_{k} -\Lambda_{\ceil{{k}/{2}}})) = o(1/(k+1)^{p+1})$ follows from \eqref{eq:sumsequence-bigo}.


Part \ref{lem:sumsequence:part:main:c} directly follows from Part \ref{lem:sumsequence:part:main:b}.

Part~\ref{lem:sumsequence:part:e}. 
For every integer $0 \leq i \leq k$, we have $a_{k} \leq a_{i} + \sum_{j=i}^{k-1}e_{j}$. Thus, $\Lambda_k a_{k}=\sum_{i=0}^k\lambda_i a_k\le \sum_{i=0}^k\lambda_i a_i+\sum_{i=0}^k\lambda_i (\sum_{j=i}^{k-1}e_{j}) = \sum_{i=0}^k\lambda_i a_i+\sum_{i=0}^{k-1}e_i(\sum_{j=0}^{i}\lambda_j)=\sum_{i=0}^k\lambda_i a_i+\sum_{i=0}^{k-1}\Lambda_ie_i\le \sum_{i=0}^\infty\lambda_i a_i+\sum_{i=0}^\infty\Lambda_ie_i, $ from which the left part of \eqref{eq:sumsequence-error-bigo} follows. The proof for the right part of~\eqref{eq:sumsequence-error-bigo} is similar to Part \ref{lem:sumsequence:part:main}. The condition $e_k = O(1/(k+1)^q)$ for appropriate $q$ is used to ensure that  $\sum_{i=0}^\infty\Lambda_ie_i<\infty$ for each setting of $\lambda_k$ in the previous Parts \ref{lem:sumsequence:part:main:a}, \ref{lem:sumsequence:part:main:b}, and \ref{lem:sumsequence:part:main:c}.

Part~\ref{lem:sumsequence:part:b}. Note that 
\begin{align*}
\lambda_k(k+1) a_k  \leq (k+1) b_k - (k+1) b_{k+1} + (k+1) e_k = b_{k+1} + ((k+1)b_k - (k+2)b_{k+1}) + (k+1)e_k.
\end{align*}
Thus,  because the upper bound on $(k+1)\lambda_ka_k$ is the sum of a telescoping term and a summable term, we have $\sum_{i=0}^\infty (i+1)\lambda_i a_i\leq \sum_{i=0}^\infty b_i + \sum_{i=0}^\infty(i+1)e_i < \infty$. Parts \ref{lem:sumsequence:part:3a} and \ref{lem:sumsequence:part:3b} are similar to Part  \ref{lem:sumsequence:part:main:b}.

Part \ref{lem:sumsequence:part:nonmono} is straightforward, so we omit its proof.
\qed
\end{proof}

Part~\ref{lem:sumsequence:part:main} of Lemma~\ref{lem:sumsequence} is a generalization of  \cite[Theorem 3.3.1]{knopp1956infinite} and \cite[Lemma 1.2]{deng20131}, which state that a nonnegative, summable, monotonic sequence converges at the rate of $o(1/(k+1))$. This result is key for deducing the convergence rates of several quantities in this paper.



\section{Iterative fixed-point residual analysis}\label{sec:FPR}

In this section we establish the convergence rate of the \emph{fixed-point residual} (FPR), $\|Tz^k - z^k\|^2$, at the $k$th iteration of Algorithm~\ref{alg:KM}.   

The convergence of Algorithm~\ref{alg:KM} is well-studied \cite{combettes2004solving,cominetti2012rate,liang2014convergence}.  In particular, weak convergence of $(z^j)_{j \geq 0}$ to a fixed point of $T$ holds under mild conditions on the sequence $(\lambda_j)_{j \geq 0}$ \cite[Theorem 3.1]{combettes2004solving}.  Because strong convergence of Algorithm~\ref{alg:KM} may fail (in the infinite dimensional setting), the quantity $\|z^k - z^\ast\|$ where $z^\ast$ is a fixed point of $T$ may be bounded above zero for all $k \geq0$.  However, the property $\lim_{k\to\infty}\|Tz^k - z^k\| = 0$, known as \emph{asymptotic regularity} \cite{browder1966solution}, always holds when a fixed point of $T$ exists. Thus, we can always measure the convergence rate of the FPR. 

We measure $\|Tz^k - z^k\|^2$ when we could just as well measure $\|Tz^k - z^k\|$.  We choose to measure the squared norm because it naturally appears in our analysis.  In addition, it is summable and monotonic, which is  analyzable by Lemma~\ref{lem:sumsequence}.

In first-order optimization algorithms, the FPR typically relates to the size of objective gradient.  For example, in the unit-step gradient descent algorithm, $z^{k+1}= z^k - \nabla f(z^k)$, the FPR is given by $\|\nabla f(z^k)\|^2$. In the proximal point algorithm, the FPR is given by $\|\tnabla f(z^{k+1})\|^2$ {where $\tnabla f(z^{k+1}) := (z^k - z^{k+1}) \in \partial f(z^{k+1})$ (see Equation~\eqref{eq:backward})}. When the objective is the sum of multiple functions, the FPR is a combination of the (sub)gradients of those functions in the objective. Using the subgradient inequality, we will derive a rate on $f(z^k)-f(x^*)$ from a  rate on the FPR where $x^\ast$ is a minimizer of $f$.

\subsection{$o(1/(k+1))$ FPR of averaged operators}\label{sec:averaged}

We now  prove the main result of this section. We do not include the known weak convergence result \cite[Theorem 3.1]{combettes2004solving}, but we deduce a convergence rate for the FPR. The  new results in the following theorem are the little-$o$ convergence rates in Equation~\eqref{prop:averagedconvergence:eq:1} and  in Part~\ref{prop:averagedconvergence:part:inexact}; the rest of the results can be found in~\cite[Proof of Proposition 5.14]{bauschke2011convex},~\cite[Proposition 11]{cominetti2012rate}, and~\cite{liang2014convergence}. 

\begin{theorem}[Convergence rate of averaged operators]\label{prop:averagedconvergence}
Let $T : \cH \rightarrow \cH$ be a nonexpansive operator, let $z^\ast$ be a fixed point of $T$, let $(\lambda_j)_{j \geq 0} \subseteq (0, 1]$ be a sequence of positive numbers, let $\tau_k := \lambda_k(1-\lambda_k)$, and let $z^0 \in \cH$.   Suppose that $(z^j)_{j \geq 0}\subseteq\cH$ is generated by Algorithm~\ref{alg:KM}: for all $k \geq 0$, let
\begin{align}\label{eq:KMiteration}
z^{k+1} &= T_{\lambda_k}(z^k),
\end{align}
where $T_\lambda$ is defined in \eqref{eq:taveraged}.
Then, the following results hold 
\begin{enumerate}
\item 
$\|z^k - z^\ast\|^2$ is monotonically nonincreasing; \label{prop:averagedconvergence:eq:mono}
\item 
$\|Tz^k - z^k\|^2$ is monotonically nonincreasing; \label{prop:averagedconvergence:eq:FPRmono}
\item \label{prop:averagedconvergence:eq:summability} $\tau_k \|Tz^{k} - z^{k} \|^2$ is summable:\begin{align}\label{prop:averagedconvergence:eq:sum}
\sum_{i=0}^\infty \tau_i \|Tz^{i} - z^{i} \|^2 \leq \|z^0 - z^\ast\|^2;
\end{align}
\item \label{prop:averagedconvergence:eq:bounds} if $\tau_k>0$ for all $k\geq 0$, then the convergence estimates hold:
\begin{align}\label{prop:averagedconvergence:eq:1}
\|Tz^{k} - z^k\|^2 &\leq \frac{\|z^{0} - z^\ast\|^2}{\sum_{i=0}^k \tau_i} && \mbox{and} && \|Tz^{k}- z^k\|^2 = o\left(\frac{1}{\sum_{i=\ceil{\frac{k}{2}}+1}^k\tau_i}\right).
\end{align}
In particular, if $(\tau_j)_{j \geq 0} \subseteq (\varepsilon, \infty)$  for some $\varepsilon > 0$, then $\|Tz^{k}- z^k\|^2 = o(1/(k+1))$.
\item \label{prop:averagedconvergence:part:inexact} Instead of Iteration~\eqref{eq:KMiteration}, for all $k \geq 0$, let
\beq\label{km_itr_err}
z^{k+1} := T_{\lambda_k}(z^k)  + \lambda_k e^k
\eeq {\color{blue} for an error sequence $(e^j)_{j \geq 0} \subseteq \cH$ that satisfies $\sum_{i=0}^k \lambda_k \|e_k\| < \infty$ and $\sum_{i=0}^\infty (i+1)\lambda_k^2\|e^k\|^2 < \infty$. Note that these bounds hold, for example, when for all $k \geq 0$ $\lambda_k \|e^k\|\le \omega_k$ for a sequence $(\omega_j)_{j\ge 0}$ that is nonnegative, summable, and monotonically nonincreasing.}
Then if $(\tau_j)_{j \geq 0} \subseteq (\varepsilon, \infty)$  for some $\varepsilon > 0$, we continue to have $\|Tz^{k}- z^k\|^2 = o(1/(k+1))$.
\end{enumerate}
\end{theorem}
\begin{proof}
As noted before the Theorem, for Parts~\ref{prop:averagedconvergence:eq:mono} through~\ref{prop:averagedconvergence:eq:bounds}, we only need to prove the little-$o$ convergence rate. This follows  from the monotonicity of $(\|Tz^j - z^j\|^2)_{j \geq 0}$, Equation~\eqref{prop:averagedconvergence:eq:sum}, and Part~\ref{lem:sumsequence:part:main} of Lemma~\ref{lem:sumsequence}. 

Part~\ref{prop:averagedconvergence:part:inexact}: 
{\color{blue} We first show that the condition involving the sequence $(w_j)_{j \geq 0}$ is sufficient to guarantee the error bounds.} We have $ \sum_{i=0}^\infty \lambda_i \|e^i\|\le\sum_{i=0}^\infty \omega_i <\infty$ and $\sum_{i=0}^\infty (i+1)\lambda_i^2 \|e^i\|^2\le \sum_{i=0}^\infty (i+1)\omega_i^2<\infty$, where the last inequality is shown as follows. By Part~\ref{lem:sumsequence:part:main} of Lemma~\ref{lem:sumsequence}, we have $\omega_k=o(1/(k+1))$. Therefore, there exists a finite $K$ such that $(k+1)\omega_k<1$ for $k>K$. Therefore, $\sum_{i=0}^\infty (i+1)\omega_i^2 < \sum_{i=0}^K (i+1)\omega_i^2 + \sum_{i=K+1}^\infty \omega_i<\infty$. 

For simplicity, introduce   $p^{k}:=Tz^{k}-z^{k}$,  $p^{k+1}:=Tz^{k+1}-z^{k+1}$, and
$r^k:=z^{k+1} - z^k $.  Then from \eqref{eq:taveraged} and \eqref{km_itr_err}, we have $p^{k}=\frac{1}{\lambda_k}(r^k - \lambda_k e^k)$. Also introduce $q^k := Tz^{k+1}-Tz^k$. Then, $p^{k+1}-p^k=q^k-r^k$. 

We will show: (i)  $\|p^{k+1}\|^2 \leq \|p^k\|^2 + \frac{\lambda_k^2}{\tau_k}\|e^k\|^2$ and (ii)  $\sum_{i=0}^\infty \tau_i\|p^i\|^2 <\infty$. Then, applying Part \ref{lem:sumsequence:part:e} of Lemma \ref{lem:sumsequence} (with $a_k = \|p^k\|^2$, $e_k = \frac{\lambda_k^2}{\tau_k}\|e^k\|^2$ , and $\lambda_k = 1$ for which we have $\Lambda_k=\sum_{i=0}^k\lambda_i\le (k+1)$), we immediately obtain the rate $\|Tz^{k}- z^k\|^2 = o(1/(k+1))$.



To prove (i),  we have 
$$
\|p^{k+1}\|^2=\|p^k\|^2+ \|p^{k+1}-p^k\|^2+2\dotp{p^{k+1}-p^k, p^k}=\|p^k\|^2+ \|q^k-r^k\|^2+\frac{2}{\lambda_k}\dotp{q^k-r^k, r^k - \lambda_k e^k}.
$$
By the nonexpansiveness of $T$, we have $\|q^k\|^2\le \|r^k\|^2$ and thus
$$2\dotp{q^k-r^k,r^k}=\|q^k\|^2-\|r^k\|^2-\|q^k-r^k\|^2\le- \|q^k-r^k\|^2.$$
Therefore,
\begin{align*}\|p^{k+1}\|^2 & \le \|p^k\|^2- \frac{1-\lambda_k}{\lambda_k}\|q^k-r^k\|^2+2\dotp{q^k-r^k, e^k}.\\
&=\|p^k\|^2-\frac{1-\lambda_k}{\lambda_k}\left\|q^k-r^k-\frac{\lambda^k}{1-\lambda_k}e^k\right\|^2+\frac{\lambda^k}{1-\lambda_k}\|e^k\|^2\\
&\le \|p^k\|^2 +\frac{\lambda_k^2}{\tau_k}\|e^k\|^2. 
\end{align*}

To prove (ii): First, $\|z^{k} - z^\ast\|$ is uniformly bounded because  $\|z^{k+1} - z^\ast\|  \leq (1-\lambda_k)\|z^k - z^{\ast}\| + \lambda_k\|Tz^k - z^\ast\| + \lambda_k \|e^k\| \leq \|z^k - z^\ast\| + \lambda_k\|e^k\|$ by the triangle inequality and the nonexpansiveness of $T$. 
 From~\cite[Corollary 2.14]{bauschke2011convex}, we have
\begin{align*}
\|z^{k+1} - z^\ast\|^2 &= \|(1-\lambda_k)(z^{k} - z^\ast) + \lambda_k(Tz^k - z^\ast + e^k)\|^2 \\
&= (1-\lambda_k)\|z^{k} - z^\ast\|^2 + \lambda_k\|Tz^k - z^\ast +e^k\|^2 - \lambda_k(1-\lambda_k)\|p^k + e^k\|^2 \\
&= (1-\lambda_k)\|z^k - z^\ast\|^2 + \lambda_k\left(\|Tz^k - z^\ast\|^2 + 2\lambda_k\dotp{ Tz^k - z^\ast, e^k} + \lambda_k \|e^k\|^2\right) \\
&- \lambda_k(1-\lambda_k)\left(\|p^k\|^2 + 2\dotp{p^k, e^k} + \|e^k\|^2\right) \\
&\leq \|z^k - z^\ast\|^2 - \tau_k\|p^k\|^2 
+ \underbrace{\lambda_k^2\|e^k\|^2 + 2\lambda_k\|Tz^k - z^\ast\|\|e^k\| + 2\tau_k\|p^k\|\|e^k\|}_{=:\xi_k} \\
& = \|z^k - z^\ast\|^2 - \tau_k\|p^k\|^2 + \xi_k. 
\end{align*}
Because we have shown (a) $\|Tz^k - z^\ast\|$ and $\|p^k\|$ are bounded, (b) $\sum_{k=0}^\infty\tau_k\|e^k\|\le\sum_{k=0}^\infty\lambda_k\|e^k\|<\infty$, and (c) $\sum_{k=0}^\infty\lambda_k^2\|e^k\|^2<\infty$, we have $\sum_{i=0}^\infty \xi_k <\infty$ and thus $\sum_{i = 0}^\infty \tau_k\|Tz^i - z^i\|^2 \leq \|z^0 - z^\ast\|^2 +\sum_{i=0}^\infty \xi_k <\infty$.
\qed\end{proof}

\subsubsection{Notes on Theorem~\ref{prop:averagedconvergence}}

The FPR, $\|Tz^k - z^k\|^2$, is a normalized version of the successive iterate differences $z^{k+1} - z^k = \lambda_k(Tz^k - z^k)$.  Thus, the convergence rates of $\|Tz^k - z^k\|^2$ naturally induce convergence rates of $\|z^{k+1} - z^k\|^2$.

 Note that  $o(1/(k+1))$ is the optimal convergence rate for the class of nonexpansive operators \cite[Remarque 4]{brezis1978produits}.  In the special case that $T = \prox_{\gamma f}$ for some closed, proper, and convex function $f$, the rate of $\|Tz^{k} - z^{k}\|^2$ improves to $O(1/(k+1)^2)$ \cite[Th\'eor\`eme 9]{brezis1978produits}.  See section~\ref{sec:optimalFPR} for more optimality results. Also, the little-$o$ convergence rate of the fixed-point residual associated to the resolvent of a maximal monotone linear operator was shown in~\cite[Proposition 4]{brezis1978produits}. Finally, we mention the parallel work~\cite{CormanandYuan}, which proves a similar little-$o$ convergence rate for the fixed-point residual of relaxed PPA. 

In general, it is possible that the nonexpansive operator, $T : \cH \rightarrow \cH$, is already averaged, i.e. there exists a nonexpansive operator $N : \cH \rightarrow \cH$ and a positive constant $\alpha \in (0, 1]$ such that $T = (1-\alpha)I_{\cH} + \alpha N$. In this case, Lemma~\ref{lem:fixedpointaverage} shows that $T$ and $N$ share the same fixed point set.  Thus, we can apply Theorem~\ref{prop:averagedconvergence} to $N = (1 - (1/\alpha))I_{\cH} + (1/\alpha)T$.  Furthermore, $N_{\lambda} = (1 - \lambda/\alpha)I_\cH + (\lambda/\alpha) T$. Thus, when we translate this back to an iteration on $T$, it enlarges the region of relaxation parameters to $\lambda_k \in (0, 1/\alpha)$ and modifies $\tau_k$ accordingly to $\tau_k = \lambda_k(1-\alpha\lambda_k)/\alpha$, and the same convergence results continue to hold.

To the best of our knowledge, The little-$o$ rates produced in Theorem~\ref{prop:averagedconvergence} have never been established for the KM iteration. See \cite{cominetti2012rate,liang2014convergence} for similar big-$O$ results. {\color{blue} Note that our rate in Part~\ref{prop:averagedconvergence:part:inexact} is strictly better than the one shown in~\cite{liang2014convergence}, and it is given under a much weaker condition on the error. Indeed, \cite{liang2014convergence} achieves an $O(1/(k+1))$ convergence rate whenever $\sum_{i=0}^\infty (i+1)\|e^k\| < \infty$, which implies that $\min_{i = 0, \cdots, k}\{\|e^k\|\} = o(1/(k+1)^2)$ by Lemma~\ref{lem:sumsequence}. In contrast, any error sequence of the form $\|e^k\| = O(1/(k+1)^\alpha)$ with $\alpha > 1$ will satisfy Part~\ref{prop:averagedconvergence:part:inexact} of our Theorem~\ref{prop:averagedconvergence}.} Finally, note that in the Banach space case, we cannot improve the big-$O$ rates to little-$o$ \cite[Section 2.4]{cominetti2012rate}.  

\subsection{$o(1/(k+1))$ FPR of relaxed PRS}
In this section, we apply Theorem~\ref{prop:averagedconvergence} to the $\TPRS$ operator defined in Proposition~\ref{prop:TPRS}.
For the special case of DRS ($(1/2)$-averaged PRS), it is straightforward to establish the rate of the FPR
\begin{align*}
\|(\TPRS)_{{1}/{2}} z^k - z^k\|^2 &= O\left(\frac{1}{k+1}\right)
\end{align*}
from  two existing results: (i) the DRS iteration is   a  proximal iteration applied to a certain monotone operator~\cite[Section 4]{eckstein1992douglas}; (ii) the  convergence rate of the FPR for proximal iterations is $O(1/(k+1))$ \cite[Proposition 8]{brezis1978produits} whenever a fixed point exists. Our results below are established for general averaged PRS operators and the rate is improved  to $o(1/(k+1))$. 

The following corollary is an  immediate consequence of Theorem~\ref{prop:averagedconvergence}.
\begin{corollary}[Convergence rate of relaxed PRS]\label{cor:DRSaveragedconvergence}
Let $z^\ast $ be a fixed point of $\TPRS$, let $(\lambda_j)_{j \geq 0} \subseteq (0, 1]$ be a sequence of positive numbers, let $\tau_k := \lambda_k(1-\lambda_k)$ for all $k \geq 0$, and let $z^0 \in \cH$.  Suppose that $(z^j)_{j \geq 0}\subseteq\cH$ is generated by Algorithm~\ref{alg:DRS}. 
Then the sequence $\|z^k - z^\ast\|^2$ is monotonically nonincreasing and the following inequality holds:
\begin{align}\label{prop:DRSaveragedconvergence:eq:sum}
\sum_{i=0}^\infty \tau_i \|\TPRS z^{i} -  z^{i} \|^2 \leq \|z^0 - z^\ast\|^2.
\end{align}
Furthermore,  if $\underline{\tau} := \inf_{j \geq 0} \tau_j > 0$, then the following convergence rates hold:
\begin{align}\label{cor:DRSaveragedconvergence:eq:main}
\|\TPRS z^{k} - z^k\|^2 \leq \frac{\|z^{0} - z^\ast\|^2}{\underline{\tau}(k+1)} && and &&  \|\TPRS z^{k} - z^k\|^2 = o\left(\frac{1}{\underline{\tau}(k+1)}\right).
\end{align}
\end{corollary}

\subsection{$o(1/(k+1)^2)$ FPR of FBS and PPA}\label{sec:fbs}
In this section,  we assume that $\nabla g$ is $({1}/{\beta})$-Lipschitz, and we analyze the convergence rate of FBS algorithm given in Equations~\eqref{eq:FBSiterates} and \eqref{eq:FBSiterates1}.  If $g = 0$, FBS reduces to PPA and $\beta = \infty$. If $f= 0$, FBS reduces to gradient descent. The FBS algorithm can be written in the following operator form:
\begin{align*}
\TFBS := \prox_{\gamma f}\circ (I - \gamma \nabla g).
\end{align*}
Because $\prox_{\gamma f}$ is $(1/2)$-averaged and $I - \gamma \nabla g$ is $\gamma/(2\beta)$-averaged \cite[Theorem 3(b)]{nobhuiko}, it follows that $\TFBS$ is $\alpha_{\mathrm{FBS}}$-averaged for
\begin{align*}
\alpha_{\mathrm{FBS}} := \frac{2\beta}{4\beta - \gamma} \in (1/2, 1)
\end{align*}
whenever $\gamma < 2\beta$ \cite[Proposition 4.32]{bauschke2011convex}. Thus, we have $\TFBS = (1 - \alpha_{\mathrm{FBS}}) I + \alpha_{\mathrm{FBS}} T$ for a certain nonexpansive operator $T$, and $\TFBS(z^k) - z^k = \alpha_{\mathrm{FBS}}(Tz^k - z^k)$. In particular, for all $\gamma < 2\beta$ the following sum is finite:
\begin{align*}
\sum_{i=0}^\infty \|\TFBS(z^k) - z^k\|^2 &\stackrel{\eqref{prop:averagedconvergence:eq:sum}}{\leq} \frac{\alpha_{\mathrm{FBS}}\|z^0 - z^\ast\|^2}{(1- \alpha_{\mathrm{FBS}})}.
\end{align*}

To analyze the FBS algorithm we need to derive a joint subgradient inequality for $f+g$. First, we recall the following sufficient descent property for Lipschitz differentiable functions.
\begin{theorem}[Descent theorem {\cite[Theorem 18.15(iii)]{bauschke2011convex}}]\label{thm:descent}
If $g$ is differentiable and $\nabla g$ is $({1}/{\beta})$-Lipschitz, then for all $x, y \in \cH$ we have the upper bound
\begin{align}\label{eq:lipschitzderivative}
g(x) &\leq g(y) + \dotp{ x- y, \nabla g(y)} + \frac{1}{2\beta} \|x - y\|^2.
\end{align}
\end{theorem}

\begin{corollary}[Joint descent theorem]\label{cor:jointdescent}
If $g$ is differentiable and $\nabla g$ is $({1}/{\beta})$-Lipschitz, then for all points $x, y\in \dom(f)$ and $z \in \cH$, and subgradients $\tnabla f(x) \in \partial f(x)$, we have
\begin{align}\label{cor:jointdescent:eq:main}
f(x) + g(x) &\leq f(y) + g(y) + \dotp{ x- y, \nabla g(z) + \tnabla f(x)} + \frac{1}{2\beta} \|z - x\|^2.
\end{align}
\end{corollary}
\begin{proof}
Inequality \eqref{cor:jointdescent:eq:main} follows from adding  the upper bound
\begin{align*}
g(x) - g(y) &\leq g(z) - g(y) + \dotp{ x- z, \nabla g(z)} + \frac{1}{2\beta}\|z - x\|^2 \leq \dotp{x - y, \nabla g(z)} + \frac{1}{2\beta}\|z - x\|^2
\end{align*}
with the subgradient inequality: $f(x) \le f(y) + \dotp{x-y,\tnabla f(x)}$.
\qed
\end{proof}

We now improve the $O(1/(k+1)^2)$ FPR rate for PPA in \cite[Th\'eor\`eme 9]{brezis1978produits} by showing that the FPR rate of FBS  is actually $o(1/(k+1)^2)$.

\begin{theorem}[Objective and FPR convergence of FBS]\label{thm:PPAconvergence}
Let $z^0 \in \dom(f)\cap \dom(g)$ and let $x^\ast$ be a minimizer of $f+g$. Suppose that $(z^j)_{j \geq 0}$ is generated by FBS (iteration \eqref{eq:FBSiterates}) where $\nabla g$ is $({1}/{\beta})$-Lipschitz and  $\gamma < 2\beta$. Then for all $k\ge 0$, \begin{align*}
f(z^{k+1}) + g(z^{k+1}) - f(x^\ast) - g(x^\ast) \leq  \frac{\|z^0 - x^\ast\|^2}{k+1} \times \begin{cases}
\frac{1}{2\gamma} & \text{if $\gamma \leq \beta$}; \\
 \left(\frac{1}{2\gamma} + \left(\frac{1}{2\beta} - \frac{1}{2\gamma}\right)\frac{\alpha_{\mathrm{FBS}}}{(1-\alpha_{\mathrm{FBS}})}\right)  &\text{otherwise}.
 \end{cases},
 \end{align*}
and
$$f(z^{k+1}) + g(z^{k+1}) - f(x^\ast) - g(x^\ast) = o(1/(k+1)).$$
 In addition, for all $k \geq 0$, we have $ \|\TFBS z^{k+1} - z^{k+1}\|^2 = o(1/(k+1)^2)$ and 
 \begin{align*}
 \|\TFBS z^{k+1} - z^{k+1}\|^2 &\leq \frac{\|z^0 - x^\ast\|^2}{\left(\frac{1}{\gamma} - \frac{1}{2\beta}\right)(k+1)^2} \times\begin{cases}
 \frac{1}{2\gamma} & \text{if $\gamma \leq \beta$}; \\
 \left(\frac{1}{2\gamma} + \left(\frac{1}{2\beta} - \frac{1}{2\gamma}\right)\frac{\alpha_{\mathrm{FBS}}}{(1-\alpha_{\mathrm{FBS}})}\right)  &\text{otherwise}.
 \end{cases}
 \end{align*}
\end{theorem}
\begin{proof}
{Recall that $z^{k} - z^{k+1} = \gamma \tnabla f(z^{k+1}) + \gamma \nabla g(z^k) \in \gamma \partial f(z^{k+1}) + \gamma\nabla g(z^k)$} for all $k \geq 0$. Thus,  the joint descent theorem shows that for all $x \in \dom(f)$, we have
\begin{align*}
f(z^{k+1}) + g(z^{k+1}) -  f(x) - g(x) &\stackrel{\eqref{cor:jointdescent:eq:main}}{\leq} \frac{1}{\gamma}\dotp{z^{k+1} - x, z^k - z^{k+1}} + \frac{1}{2\beta}\|z^k - z^{k+1}\|^2\\
&= \frac{1}{2\gamma}\left( \|z^{k} - x\|^2 -  \|z^{k+1} - x\|^2\right) + \left(\frac{1}{2\beta} - \frac{1}{2\gamma}\right) \|z^{k+1} - z^k\|^2. \numberthis \label{eq:summableFBS}
\end{align*}
Let $h:=f+g$. If we set $x = x^\ast$ in Equation~\eqref{eq:summableFBS}, we see that $(h(z^{j+1}) - h(x^\ast))_{j \geq 0}$ is positive, summable, and
\begin{align*}
\sum_{i=0}^\infty \left(h(z^{i+1}) - h(x^\ast)\right) \leq \begin{cases}
\frac{1}{2\gamma}\|z^0 - x^\ast\|^2 & \text{if $\gamma \leq \beta$}; \\
\left(\frac{1}{2\gamma} + \left(\frac{1}{2\beta} - \frac{1}{2\gamma}\right)\frac{\alpha_{\mathrm{FBS}}}{(1-\alpha_{\mathrm{FBS}})}\right) \|z^0 - x^\ast\|^2 &\text{otherwise}. 
\end{cases}\numberthis\label{thm:PPAconvergence:eq:sum}
\end{align*}
In addition, if we set $x = z^k$ in Equation~\eqref{eq:summableFBS}, then we see that $(h(z^{j+1}) - h(x^\ast))_{j \geq 0}$ is decreasing:
\begin{align*}
\left(\frac{1}{\gamma} - \frac{1}{2\beta}\right)\|z^{k+1} - z^{k}\|^2 \leq h(z^{k}) - h(z^{k+1}) = (h(z^{k}) - h(x^\ast)) - (h(z^{k+1}) - h(x^\ast)).
\end{align*}
Therefore, the rates for $f(z^{k+1}) + g(z^{k+1}) - f(x^\ast) - g(x^\ast)$ follow by Lemma~\ref{lem:sumsequence} Part~\ref{lem:sumsequence:part:main:a}, with $a_k = h(z^{k+1}) - h(x^\ast)$ and $\lambda_k \equiv 1$. 

Now we prove the rates for $\|T_{\mathrm{FBS}}z^{k+1} - z^{k+1}\|^2$. We apply Part~\ref{lem:sumsequence:part:b} of Lemma~\ref{lem:sumsequence} with $a_{k} = \left({1}/{\gamma} - {1}/({2\beta})\right)\|z^{k+2} - z^{k+1}\|^2$, $\lambda_k \equiv 1$, $e_k = 0$,  and $b_k= h(z^{k+1}) - h(x^\ast)$ for all $k \geq 0$, to show that $\sum_{i=0}^\infty (i+1)a_i$ is less than the sum in Equation~\eqref{thm:PPAconvergence:eq:sum}. Part~\ref{prop:averagedconvergence:eq:FPRmono} of Theorem~\ref{prop:averagedconvergence} shows that $(a_j)_{j \geq 0}$ is monotonically nonincreasing. Therefore, the convergence rate of $(a_j)_{j \geq 0}$  follows from Part~\ref{lem:sumsequence:part:main:b} of Lemma \ref{lem:sumsequence}.
 \qed
\end{proof}

When $f = 0$, the objective error upper bound in Theorem~\ref{thm:PPAconvergence} is strictly better than the bound provided in \cite[Corollary 2.1.2]{nesterov2004introductory}.  In FBS, the objective error rate is the same as the one derived in \cite[Theorem 3.1]{beck2009fast}, when $\gamma \in (0, \beta]$, and is the same as the one given in~\cite{bredies2009forward} in the case that $\gamma \in (0, 2\beta)$.  The little-$o$ FPR rate is new in all cases except for the special case of PPA ($g \equiv 0$) under the condition that the sequence $(z^j)_{j \geq 0}$ strongly converges to a minimizer \cite{guler}. 

\subsection{$o(1/(k+1)^2)$ FPR of one dimensional DRS}

Whenever the operator $(\TPRS)_{1/2}$ is applied in $\vR$, the convergence rate of the FPR improves to $o(1/(k+1)^2)$.

\begin{theorem}\label{thm:1DDRS}
Suppose that $\cH = \vR$, and suppose that $(z^j)_{j \geq 0}$ is generated by the DRS algorithm, i.e. Algorithm~\ref{alg:DRS} with $\lambda_k \equiv 1/2$.  Then for all $k \geq 0$,
\begin{align*}
|(\TPRS)_{{1}/{2}} z^{k+1} - z^{k+1}|^2 = \frac{|z^0 - z^\ast|^2}{2(k+1)^2}  && \mathrm{and} && |(\TPRS)_{{1}/{2}} z^{k+1} - z^{k+1}|^2 = o\left(\frac{1}{(k+1)^2}\right).
\end{align*}

\end{theorem}
\begin{proof}
Note that  $(\TPRS)_{1/2}$ is $(1/2)$-averaged, and, hence, it is the resolvent of some maximal monotone operator on $\vR$ \cite[Corollary 23.8]{bauschke2011convex}. Furthermore, every maximal monotone operator on $\vR$ is the subdifferential operator of a closed, proper, and convex function \cite[Corollary 22.19]{bauschke2011convex}. Therefore, DRS is equivalent to the proximal point algorithm applied to a certain convex function on $\vR$. Thus, the result follows by Theorem~\ref{cor:jointdescent} applied to this function. \qed

\end{proof}

\subsection{$O(1/\Lambda_k^2)$ ergodic FPR of Fej\'er monotone sequences}

The following definition has proved to be quite useful in the analysis of optimization algorithms \cite{combettes2001quasi}.

\begin{definition}[Fej\'er monotone sequences]
A sequence $(z^j)_{j \geq 0}\subseteq \cH$ is \emph{Fej\'er monotone} with respect to a nonempty set $C \subseteq \cH$ if for all $z \in C$, we have $\|z^{k+1} - z\|^2 \leq \|z^k - z\|^2.$
\end{definition}

The following fact is  trivial, but allows us to deduce ergodic convergence rates of many algorithms.

\begin{theorem}\label{eq:ergodicFPR}
Let $(z^j)_{j \geq 0}$ be a Fej\'er monotone sequence with respect to a nonempty set $C \subseteq \cH$. Suppose that $z^{k+1} - z^k = \lambda_k (x^k - y^k)$ for a sequence $((x^j, y^j))_{j \geq 0} \subseteq \cH^2$, and a sequence of positive real numbers $(\lambda_j)_{j \geq 0}$.  For all $k \geq 0$, let $\overline{z}^k := ({1}/{\Lambda_k}) \sum_{i=0}^k \lambda_kz^k$, let $\overline{x}^k := ({1}/{\Lambda_k}) \sum_{i=0}^k \lambda_k x^k$, and let $\overline{y}^k := ({1}/{\Lambda_k}) \sum_{i=0}^k \lambda_k y^k$. Then we get the following bound for all $z \in C$:
\begin{align*}
\|\overline{x}^k - \overline{y}^k\|^2 \leq \frac{4\|z^{0} - z\|^2}{\Lambda_k^2}.
\end{align*}
\end{theorem}
\begin{proof}
It follows directly from the inequality: $\lambda_k\|\overline{x}^k - \overline{y}^k\| = \left\|\sum_{i=0}^k \left(z^{k+1} - z^k\right)\right\| = \left\|z^{k+1} - z^0\right\| \leq 2\left\|z^0 - z\right\|. $\qed
\end{proof}

In view of Part~\ref{prop:averagedconvergence:eq:mono} of Theorem~\ref{prop:averagedconvergence}, we see that any sequence $(z^j)_{j \geq 0}$ generated by Algorithm~\ref{alg:KM} is Fej\'er monotone with respect to the set of fixed-points of $T$. Therefore, Theorem~\ref{eq:ergodicFPR} directly applies to the KM iteration in Equation~\eqref{eq:KMiteration} with the choice $x^k = Tz^k$ and $y^k = z^k$ for all $k \geq 0$.

The interested reader can proceed to Section~\ref{sec:optimalFPR} for several examples that show the optimality of the rates predicted in this section.

\section{Subgradients and fundamental inequalities}\label{sec:subfi}

We now shift the focus from operator-theoretic analysis to function minimization. This section establishes fundamental inequalities that connect the \emph{FPR} in Section~\ref{sec:FPR} to the \emph{objective error} of the relaxed PRS algorithm.

In first-order optimization algorithms, we only have access to (sub)gradients and function values.  Consequently, the FPR at each iteration is usually some linear combination of (sub)gradients. \cut{We can relate  objective  values to the FPR through the subgradient inequality in Equation \eqref{eq:subineq}.} In simple first-order algorithms, for example the (sub)gradient method, a (sub)gradient is drawn from a single point at each iteration. In splitting algorithms for problems with multiple convex functions, each function draws at subgradient at a different point.  There is no natural point at which we can evaluate the entire objective function; this complicates the analysis of the relaxed PRS algorithm.

In the relaxed PRS  algorithm, there are two objective functions $f$ and $g$, and the two operators $\refl_{\gamma f}$ and $\refl_{\gamma g}$ are calculated one after another  at different points, neither of which equals $z^k$ or $z^{k+1}$. Consequently, the expression $z^k-z^{k+1}$ is more complicated, and the  analysis for standard (sub)gradient iteration does not carry through.

We \textbf{let $x_f$ and $x_g$ be the points where subgradients of $f$ and $g$ are drawn, respectively}, and introduce a triangle diagram in Figure~\ref{fig:DRSTR} for deducing the algebraic relations among points $z$, $x_f$ and $x_g$. These relations will be used frequently in our analysis. Propositions~\ref{prop:DRSupper} and~\ref{prop:DRSlower} use this diagram to bound the objective error in terms of the FPR. \cut{Based on this diagram, fundamental  upper and lower bounds of   objective errors are deduced in terms of the FPR, $\|z^{k+1}-z^k\|^2$, at each iteration, in  Propositions~\ref{prop:DRSupper} and~\ref{prop:DRSlower}, respectively.} In these bounds, the objective errors of $f$ and $g$ are measured at two points $x_f$ and $x_g$ such that $x_f \neq x_g$. Later we will assume that one of the objectives is Lipschitz continuous and evaluate both functions at the same point (See Corollaries~\ref{cor:drsergodiclipschitz} and~\ref{cor:drsnonergodiclipschitz}).

We conclude this introduction by combining the subgradient notation in Equation~\eqref{eq:tnabla} and Lemma~\ref{lem:optimalityofprox} to arrive at the expressions
\begin{align}\label{lem:proxbackward:eq:main}
\prox_{\gamma f}(x) = x - \gamma \tnabla f(\prox_{\gamma f}(x)) && \mathrm{and} && \refl_{\gamma f}(x)  = x - 2\gamma \tnabla f(\prox_{\gamma f}(x)).
\end{align}
With this notation, we can decompose the FPR at each iteration of the relaxed PRS algorithm in terms of subgradients drawn at certain points.

\subsection{A subgradient representation of relaxed PRS}

In this section we write the relaxed PRS algorithm in terms of subgradients.  Lemma~\ref{prop:DRSmainidentity}, Table~\ref{table:DRSidentities}, and Figure~\ref{fig:DRSTR} summarize a single iteration of relaxed PRS. 

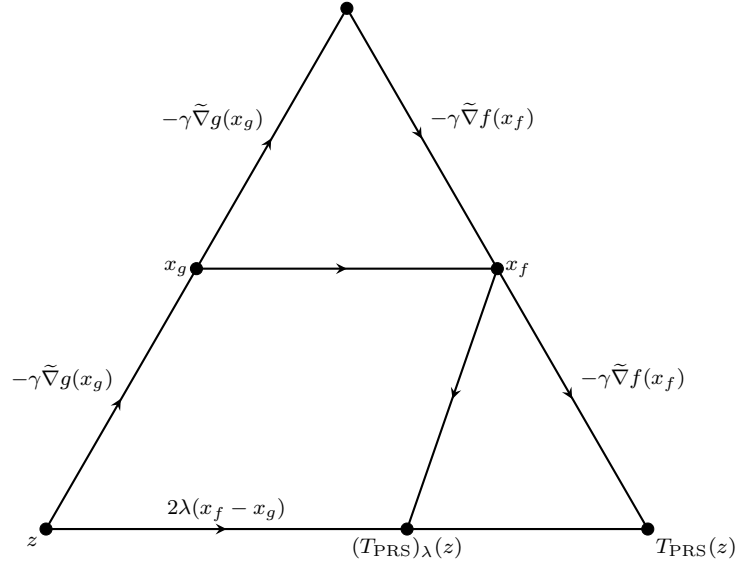
\begin{figure}[h!]
  \centering
    \begin{tikzpicture}[scale=2]

    \draw[directed, thick] (0, 0) -- (1, 1.7321);
    \draw[directed, thick] (1, 1.7321) -- (2, 3.4641);
    \draw[directed, thick] (2, 3.4641) -- (3, 1.7321);
    \draw[directed, thick] (3, 1.7321) -- (4, 0);
    \draw[directed, thick] (3, 1.7321) -- (2.4, 0);
    \draw[directed, thick] (0, 0) -- (2.4, 0);
    \draw[thick] (2.4, 0) -- (4, 0);
    \draw[directed, thick] (1, 1.7321) -- (3, 1.7321);
    \draw[fill] (0, 0) circle [radius=.040];
    \draw[fill] (4, 0) circle [radius=.040];
    \draw[fill] (2, 3.4641) circle [radius=.040];
    \draw[fill] (1, 1.7321) circle [radius=.040];
    \draw[fill] (3, 1.7321) circle [radius=.040];
    \draw[fill] (2.4000, 0) circle [radius=.040];
    \node [below left] at (0, 0) {$z$};
    \node [below right] at (4, 0) {$\TPRS (z)$};
    \node [below] at (2.4, 0) {$(\TPRS )_{\lambda}(z)$};
    \node [left] at (1, 1.7321) {$x_g$};
    \node [right] at (3, 1.7321) {$x_f$};
    \node [above left] at (.5, 0.8660) {$-\gamma \tnabla g(x_g)$};
    \node [above left] at (1.5, 2.5981) {$-\gamma \tnabla g(x_g)$};
    \node [above right] at (2.5, 2.5981) {$-\gamma \tnabla f(x_f)$};
    \node [above right] at (3.5, 0.8660) {$-\gamma \tnabla f(x_f)$};
    \node[above] at (1.2, 0) {$2\lambda(x_f - x_g)$};
    \end{tikzpicture}
    \caption{A single relaxed PRS iteration, from $z$ to $(\TPRS )_{\lambda}(z)$.}
    \label{fig:DRSTR}

\end{figure}

The way to read Figure~\ref{fig:DRSTR} is as follows: Given input $z$, relaxed PRS takes a \emph{backward--forward} step with respect to $g$,  then takes a \emph{backward--forward} step with respect to $f$, resulting in the point $\TPRS (z)$. (Refer to the discussion below \eqref{eq:backward} for the concepts of ``backward'' and ``forward.'') Finally, it averages the input and output: $(\TPRS )_{\lambda}(z) = (1-\lambda) z + \lambda \TPRS (z)$.  \cut{Notice that if $g$ is differentiable, the upper triangle in Figure~\ref{fig:DRSTR} is a single forward--backward step applied to $f$ and $g$.}

Lemma~\ref{prop:DRSmainidentity} summarizes and proves the identities depicted in Figure~\ref{fig:DRSTR}.
\begin{lemma}\label{prop:DRSmainidentity}
Let $z\in \cH$. Define auxiliary points $x_g := \prox_{\gamma g}(z)$ and $x_f := \prox_{\gamma f}(\refl_{\gamma g}(z))$. Then the identities hold:
\begin{align}
x_g = z - \gamma \tnabla g(x_g)  && \mathrm{and} && x_f &= x_g - \gamma \tnabla g(x_g) - \gamma \tnabla f(x_f). \label{prop:DRSmainidentity:f}
\end{align}
{where $\tnabla g(x_g) := (1/\gamma)(z - x_g) \in \partial g(x_g)$ and $\tnabla f(x_f) := (1/\gamma)(2x_g - z - x_f) \in \partial f(x_f)$.}
In addition, each relaxed PRS step has the following representation:
\begin{align}\label{eq:DRSmainidentity2}
(\TPRS )_{\lambda}(z) - z = 2\lambda(x_f- x_g) = -2  \lambda\gamma(\tnabla g(x_g)+\tnabla f(x_f)).
\end{align}
\end{lemma}
\begin{proof}
Figure~\ref{fig:DRSTR} provides an illustration of the identities.  Equation~(\ref{prop:DRSmainidentity:f}) follows from $\refl_{\gamma g}(z) = 2x_g - z =  x_g - \gamma \tnabla g(x_g)$ and Equation~\eqref{lem:proxbackward:eq:main}. Now, we can compute $\TPRS (z) - z$:
\begin{align*}
\TPRS (z) - z \stackrel{\eqref{prop:TPRS:eq:main}}{=} \refl_{\gamma f}(\refl_{\gamma g}(z)) - z = 2x_f - \refl_{\gamma g}(z)- z = 2x_f - (2x_g - z) - z &= 2(x_f - x_g).
\end{align*}
The subgradient identity in \eqref{eq:DRSmainidentity2} follows from (\ref{prop:DRSmainidentity:f}).  Finally, Equation~(\ref{eq:DRSmainidentity2}) follows from $(\TPRS )_{\lambda}(z)  - z = (1-\lambda)z + \lambda \TPRS (z) - z = \lambda(\TPRS (z) - z)$. \qed
\end{proof}

\begin{center}
\begin{table}
\centering
    \begin{tabular}{lll}
    \toprule
    Point    & Operator identity    &   Subgradient identity                   \\ \toprule
    $x_g^s$ & $=\prox_{\gamma g}(z^s)$ & $=z^s - \gamma \tnabla g(x_g^s)$ \\\midrule
    $x_f^s$ & $=\prox_{\gamma f}(\refl_{\gamma g}(z^s))$ & $=x_g^s - \gamma (\tnabla g(x_g^s) + \tnabla f(x_f^s))$  \\  \midrule
    $(\TPRS )_{\lambda}(z^s)$ & $=(1-\lambda) z^s + \lambda \TPRS (z^s)$ & $=z^s - 2\gamma \lambda (\tnabla g(x_g^s) + \tnabla f(x_f^s))$\\ \bottomrule
    \end{tabular}
 \caption{Overview of the main identities used throughout the paper. The letter $s$ denotes a superscript (e.g. $s = k$ or $s = \ast$). The vector $z^s \in \cH$ is an arbitrary input point.  See Lemma~\ref{prop:DRSmainidentity} for a proof.}\label{table:DRSidentities}
\end{table}

\end{center}

\subsection{Optimality conditions of relaxed PRS}
The following lemma characterizes the zeros of $\partial f + \partial g$  in terms of the fixed points of the PRS operator. The intuition is the following:  If $z^\ast$ is a fixed point of $\TPRS $, then the base of the triangle in Figure~\ref{fig:DRSTR} has length zero.  Thus, $x^\ast: = x_g^\ast = x_f^\ast$, and if we travel around the perimeter of the triangle, we will start and begin at $z^\ast$. This shows that $-2\gamma \tnabla g(x^\ast) = 2\gamma \tnabla f(x^\ast)$, i.e. $x^\ast \in \zer(\partial f +  \partial g)$.

\begin{lemma}[Optimality conditions of $\TPRS $]\label{lem:PRSoptimality}
The following identity holds:
\begin{align}
\zer(\partial f + \partial g) &= \{ \prox_{\gamma g}(z) \mid z \in \cH, \TPRS  z = z\}.\numberthis \label{eq:setoptimalitydrs}
\end{align}
That is, if $z^\ast$ is a fixed point of $\TPRS$, then $x^\ast = x_g^\ast = x_f^\ast$ is a solution to Problem~\ref{eq:simplesplit} and
\begin{align}\label{eq:gradoptimality}
z^\ast - x^\ast = \gamma \tnabla g(x^\ast) \in \gamma \partial g(x^\ast).
\end{align}
\end{lemma}
\begin{proof}
See \cite[Proposition 25.1]{bauschke2011convex} for the proof of Equation~\eqref{eq:setoptimalitydrs}. Equation~\eqref{eq:gradoptimality} follows because $x^\ast = \prox_{\gamma g}(z^\ast)$ if, and only if, $z^\ast - x^\ast \in \gamma \partial g(x^\ast)$.\qed
%
\end{proof}

\subsection{Fundamental inequalities}\label{sec:fundamentalinequalities}

We now  deduce inequalities on the objective function $f + g$. In particular, we compute upper and lower bounds of the quantities $f(x_f^k) + g(x_g^k) - g(x^\ast) - f(x^\ast)$.  Note that $x_f^k$ and  $x_g^k$ are not necessarily equal, so this quantity can be negative.

The most important properties of the inequalities we establish below are:
\begin{enumerate}
\item The upper fundamental inequality has a telescoping structure in $z^k $ and $z^{k+1}$.
\item They can be bounded in terms of $\|z^{k+1} - z^k\|^2$.
\end{enumerate}
Properties 1 and 2 will be used to deduce ergodic and nonergodic rates, respectively.

\begin{proposition}[Upper fundamental inequality]\label{prop:DRSupper}
Let $z \in \cH$, let $z^+ := (\TPRS)_{\lambda}(z)$, and let $x_f$ and $x_g$ be defined as in Lemma~\ref{prop:DRSmainidentity}. Then for all $x\in \dom(f) \cap \dom(g)$
\begin{align*}
4\gamma\lambda(f(x_f) + g(x_g) &- f(x) - g(x)) \leq \|z - x\|^2 - \|z^+ - x\|^2 + \left(1 - \frac{1}{\lambda}\right)\|z^{+} - z\|^2. \numberthis \label{prop:DRSupper:eq:main}
\end{align*}
\end{proposition}
\begin{proof}
We use the subgradient inequality and \eqref{eq:DRSmainidentity2} multiple times in the following derivation:
\begin{align*}
4\gamma\lambda(f(x_f) + g(x_g) - f(x) - g(x)) & \le 4\lambda \gamma \left( \dotp{ x_f - x, \tnabla f(x_f)} + \dotp{ x_g - x, \tnabla g(x_g) } \right)\\
&= 4\lambda\gamma \left( \dotp{ x_f-x_g,\tnabla f(x_f) } + \dotp{ x_g - x, \tnabla f(x_f) +\tnabla g(x_g) }\right)\\
& = 2\left( \dotp{ z^+-z,\gamma \tnabla f(x_f) } + \dotp{ x-x_g, z^+-z} \right)\\
& = 2 \dotp{ z^+-z, x+(z-x_g+\gamma \tnabla f(x_f)) - z }\\
& = 2 \dotp{ z^+ -z , x+ \gamma (\tnabla g(x_g) + \tnabla f(x_f))-z }\\
& = 2\dotp{ z^+ - z, x-\frac{1}{2\lambda} (z^+-z) - z }\\
&= \|z - x\|^2 - \|z^+ - x\|^2 + \left(1 - \frac{1}{\lambda}\right)\|z^{+} - z\|^2. \qquad \qed
\end{align*}
\end{proof}

\begin{proposition}[Lower fundamental inequality]\label{prop:DRSlower}
Let $z^\ast$ be a fixed point of $\TPRS $ and let $x^\ast := \prox_{\gamma g}(z^\ast)$.  Then for all $x_f \in  \dom(f)$ and $x_g\in \dom(g)$, the lower bound holds:
\begin{align*}
f(x_f) + g(x_g) - f(x^\ast) - g(x^\ast) &\geq \frac{1}{\gamma }\dotp{x_g - x_f,  z^\ast - x^\ast}. \numberthis \label{prop:DRSlower:eq:main}
\end{align*}
\end{proposition}
\begin{proof}
This proof essentially follows from the subgradient inequality.  Indeed, let $\tnabla g(x^\ast) = (z^\ast - x^\ast)/\gamma \in \partial g(x^\ast)$ and let $\tnabla f(x^\ast) = -\tnabla g(x^\ast) \in \partial f(x^\ast)$. Then the result follows by adding the following equations:
\begin{align*}
f(x_f) - f(x^\ast) & \geq \dotp{x_f - x^\ast, \tnabla f(x^\ast)},\\ 
g(x_g) - g(x^\ast) 
&\geq \dotp{ x_g - x_f, \tnabla g(x^\ast)} + \dotp{x_f - x^\ast, \tnabla g(x^\ast)}. \qquad \qed 
\end{align*}
\end{proof}

\section{Objective convergence rates}\label{sec:DRSconvergence}
In this section we will prove ergodic and nonergodic convergence rates of relaxed PRS when $f$ and $g$ are closed, proper, and convex functions that are possibly nonsmooth.

To ease notational memory, we note that the reader may assume that $\lambda_k = (1/2)$ for all $k \geq 0$.  This simplification implies that  $\Lambda_k = (1/2) (k+1)$, and $\tau_k = \lambda_k(1-\lambda_k)= (1/4)$ for all $k \geq 0$.

Throughout this section the point $z^\ast$  denotes an arbitrary fixed point of $\TPRS$, and we  define a minimizer of $f+g$ by the formula (Lemma~\ref{lem:PRSoptimality}):
\begin{align*}
x^\ast = \prox_{\gamma g}(z^\ast).
\end{align*}
The constant $(1/\gamma)\|z^\ast - x^\ast\|$ appears in the bounds of this section. This term is independent of $\gamma$: For any fixed point $z^\ast$ of $\TPRS$, the point $x^\ast = \prox_{\gamma g}(z^\ast)$ is a minimizer and $z^\ast - \prox_{\gamma g}(z^\ast) = \gamma \tnabla g(x^\ast) \in \gamma \partial g(x^\ast)$. Conversely, if $x^\ast \in \zer(\partial f + \partial g)$ and $\tnabla g(x^\ast) \in (-\partial f(x^\ast)) \cap \partial g(x^\ast)$, then $z^\ast = x^\ast + \gamma \tnabla g(x^\ast)$ is a fixed point. {Note that in all of our bounds, we can always replace $(1/\gamma)\|z^\ast - x^\ast\| = \|\tnabla g(x^\ast)\|$ by the infimum $\inf_{z^\ast \in \Fix(\TPRS)} (1/\gamma)\|z^\ast - x^\ast\|$ (although the infimum might not be attained).}

\subsection{Ergodic convergence rates}\label{sec:ergodic}

In this section, we analyze the ergodic convergence of relaxed PRS. The proof follows the telescoping property of the upper and lower fundamental inequalities and an application of Jensen's inequality.

\begin{theorem}[Ergodic convergence of relaxed PRS]\label{thm:drsergodic}
For all $k \geq 0$, let $\lambda_k \in (0, 1]$. Then we have the following convergence rate
\begin{align*}
- \frac{2\|z^0 - z^\ast\|\|z^\ast - x^\ast\|}{\gamma\Lambda_k} \leq f(\overline{x}_f^k) + g(\overline{x}_g^k) - f(x^\ast) - g(x^\ast) \leq \frac{1}{4\gamma\Lambda_k}\|z^0 - x^\ast\|^2.
\end{align*}
In addition, the following feasibility bound holds:
\begin{align}\label{thm:drsergodic:eq:feasibility}
\|\overline{x}_g^k - \overline{x}_f^k\|& \leq \frac{2\|z^0 - z^\ast\|}{\Lambda_k}.
\end{align}
\end{theorem}
\begin{proof}

Equation~\eqref{thm:drsergodic:eq:feasibility} follows directly from Theorem~\ref{eq:ergodicFPR} because $(z^j)_{j \geq 0}$ is Fej{\'e}r monotone with respect to $\Fix(T)$ and for all $k \geq 0$, we have $z^{k+1} - z^k = \lambda_k(x_f^k - x_g^k)$.

Recall the upper fundamental inequality from Proposition~\ref{prop:DRSupper} :
\begin{align}\label{thm:drsergodic:eq:upper}
4\gamma\lambda_k(f(x_f^k) + g(x_g^k) - f(x^\ast) - g(x^\ast) ) &\leq \|z^{k} - x^\ast\|^2 - \|z^{k+1} - x^\ast\|^2 + \left(1 - \frac{1}{\lambda_k}\right)\|z^{k+1} - z^k\|^2.
\end{align}
Because $\lambda_k \leq 1$, it follows that $(1-(1/\lambda_k)) \leq 0$.  Thus, we sum Equation~(\ref{thm:drsergodic:eq:upper}) from $i = 0$ to $k$, divide by $\Lambda_k$, and apply Jensen's inequality to get
\begin{align*}
\frac{1}{4\gamma\Lambda_k}(\|z^0 - x^\ast\|^2 - \|z^{k+1} - x^\ast\|^2) &\geq \frac{1}{\Lambda_k}\sum_{i=0}^k\lambda_i(f(x_f^i) + g(x_g^i) - f(x^\ast) - g(x^\ast)) \\
&\geq f(\overline{x}_f^k) + g(\overline{x}_g^k) - f(x^\ast) - g(x^\ast) .
\end{align*}

The lower bound is a consequence of the fundamental lower inequality and Equation~\eqref{thm:drsergodic:eq:feasibility}
\begin{align*}
f(\overline{x}_f^k) + g(\overline{x}_g^k) - f(x^\ast) - g(x^\ast) &\stackrel{\eqref{prop:DRSlower:eq:main}}{\geq} \frac{1}{\gamma}\dotp{\overline{x}_g^k - \overline{x}_f^k, z^\ast - x^\ast}  \numberthis \label{thm:drsergodic:eq:upper3} \stackrel{\eqref{thm:drsergodic:eq:feasibility}}{\geq} - \frac{2\|z^0 - z^\ast\|\|z^\ast - x^\ast\|}{\gamma\Lambda_k}. \qquad \qed
\end{align*}
\end{proof}

In general, $x_f^k \notin\dom(g)$ and $x_g^k \notin \dom(f)$, so we cannot evaluate $g$ at $x_f^k$ or $f$ at $x_g^k$. However, the conclusion of Theorem~\ref{thm:drsergodic} can be improved if $f$ or $g$ is Lipschitz continuous. The following proposition gives a sufficient condition for Lipschitz continuity on a ball:  

\begin{proposition}[Lipschitz continuity on a ball]\label{prop:lipschitzCont}
Suppose that $f : \cH \rightarrow (-\infty, \infty]$ is proper and convex.  Let $\rho > 0$ and let $x_0 \in \cH$.  If $\delta = \sup_{x, y \in B(x_0, 2\rho)} |f(x) - f(y)| < \infty$, then $f$ is $({\delta}/{\rho})$-Lipschitz on $B(x_0, \rho)$.
\end{proposition}
\begin{proof}
See \cite[Proposition 8.28]{bauschke2011convex}.
\qed\end{proof}

To use this fact, we need to show that the sequences $(x_f^j)_{j \geq 0}$, and $(x_g^j)_{j \geq 0}$ are bounded.  Recall that $x_g^s = \prox_{\gamma g}(z^s)$ and $x_f^s = \prox_{\gamma f}(\refl_{\gamma g}(z^s))$, for $s \in \{\ast, k\}$. Proximal and reflection maps are nonexpansive, so we have the following simple bound: 
\begin{align*}
\max\{ \|x_f^k - x^\ast\|, \|x_g^k - x^\ast\|\} &\leq \|z^k - z^\ast\| \leq \|z^0 - z^\ast\|.
\end{align*}
Thus, $(x_f^j)_{j \geq 0}, (x_g^j)_{j \geq 0} \subseteq \overline{B(x^\ast, \|z^0 - z^\ast\|)}.$ {By the convexity of the closed ball, we also have $(\overline{x}_f^j)_{j \geq 0}, (\overline{x}_g^j)_{j \geq 0} \subseteq \overline{B(x^\ast, \|z^0 - z^\ast\|)}.$}

\begin{corollary}[Ergodic convergence with single Lipschitz function]\label{cor:drsergodiclipschitz}
Let the notation be as in Theorem~\ref{thm:drsergodic}. Suppose that $f$ (respectively $g$) is $L$-Lipschitz continuous on $\overline{B(x^\ast, \|z^0 - z^\ast\|)}$, and let $x^k = x_g^k$ (respectively $x^k = x_f^k$). Then the following convergence rate holds
\begin{align*}
0 &\leq f(\overline{x}^k) + g(\overline{x}^k) - f(x^\ast) - g(x^\ast) \leq  \frac{1}{4\gamma\Lambda_k}\|z^0 - x^\ast\|^2 + \frac{2L\|z^0 - z^\ast\|}{\Lambda_k}.
\end{align*}
\end{corollary}
\begin{proof}
From Equation~(\ref{thm:drsergodic:eq:feasibility}),  we have  $\|\overline{x}_g^k - \overline{x}_f^k\| \leq (2/\Lambda_k)\|z^0 - z^\ast\|$. In addition, $(x_f^j)_{j \geq 0}, (x_g^j)_{j \geq 0} \subseteq \overline{B(x^\ast, \|z^0 - z^\ast\|)}$. Thus, it follows that
\begin{align*}
0 \leq f(\overline{x}^k) + g(\overline{x}^k) - f(x^\ast) - g(x^\ast) &\leq f(\overline{x}_f^k) + g(\overline{x}_g^k) - f(x^\ast) - g(x^\ast)  + L\|\overline{x}_f^k - \overline{x}_g^k\|\\
&\stackrel{\eqref{thm:drsergodic:eq:feasibility}}{\leq} f(\overline{x}_f^k) + g(\overline{x}_g^k) - f(x^\ast) - g(x^\ast)  + \frac{2L \|z^0 - z^\ast\|}{\Lambda_k}. 
\end{align*}
The upper bound follows from this equation and Theorem~\ref{thm:drsergodic}.
\qed\end{proof}

\subsection{Nonergodic convergence rates}\label{sec:DRSnonergodic}

In this section, we prove the nonergodic convergence rate of the Algorithm~\ref{alg:DRS} whenever $\underline{\tau}  := \inf_{j \geq 0} \tau_j > 0$.  The proof uses Theorem~\ref{prop:averagedconvergence} to bound the fundamental inequalities in Propositions~\ref{prop:DRSupper} and~\ref{prop:DRSlower}.

\begin{theorem}[Nonergodic convergence of relaxed PRS]\label{thm:drsnonergodic}
 For all $k \geq 0$, let $\lambda_k \in (0, 1)$. Suppose that $\underline{\tau} := \inf_{j \geq 0} \lambda_k(1-\lambda_k) > 0$. Then we have the convergence rates:
 \begin{enumerate}
\item \label{thm:drsnonergodic:part:general} In general, we have the bounds:
\begin{align*}
-\frac{\|z^0 - z^\ast\|\|z^\ast - x^\ast\|}{2\gamma \sqrt{\underline{\tau}(k+1)}} \leq f(x_f^k) + g(x_g^k) - f(x^\ast) - g(x^\ast) \leq \frac{(\|z^0 - z^\ast\| + \|z^\ast - x^\ast\|)\|z^0 - z^\ast\|}{2\gamma \sqrt{\underline{\tau}(k+1)}}
\end{align*}
and $|f(x_f^k) + g(x_g^k) - f(x^\ast) - g(x^\ast) | = o\left({1}/{\sqrt{k+1}}\right).$
\item \label{thm:drsnonergodic:part:1d} If $\cH = \vR$  and $\lambda_k \equiv {1}/{2}$, then for all $k \geq 0$,
\begin{align*}
\frac{\|z^0 - z^\ast\|\|z^\ast - x^\ast\|}{\sqrt{2}\gamma (k+1)} \leq f(x_f^{k+1}) + g(x_g^{k+1}) - f(x^\ast) - g(x^\ast) \leq \frac{(\|z^0 - z^\ast\| + \|z^\ast - x^\ast\|)\|z^0 - z^\ast\|}{\sqrt{2}\gamma (k+1)} 
\end{align*}
and $|f(x_f^{k+1}) + g(x_g^{k+1}) - f(x^\ast) - g(x^\ast) | = o\left({1}/({k+1})\right).$
\end{enumerate}
\end{theorem}
\begin{proof}
We prove Part~\ref{thm:drsnonergodic:part:general} first. For all $\lambda \in [0, 1]$, let $z_{\lambda} = (\TPRS)_{\lambda}(z^k)$. Evaluate the upper inequality in Equation~(\ref{prop:DRSupper:eq:main}) at $x = x^\ast$ to get
\begin{align*}
4\gamma\lambda(f(x_f^k) + g(x_g^k) &- f(x^\ast) - g(x^\ast) ) \leq \|z^{k} - x^\ast\|^2 - \|z_\lambda - x^\ast\|^2 + \left(1 - \frac{1}{\lambda}\right)\|z_\lambda - z^k\|^2.
\end{align*}
Recall the following identity:
\begin{align*}
\|z^{k} - x^\ast\|^2 - \|z_\lambda - x^\ast\|^2 - \|z_\lambda - z^k\|^2 &= 2\dotp{z_\lambda - x^\ast, z^k - z_\lambda}.
\end{align*}
By the triangle inequality, because $\|z_\lambda- z^\ast\| \leq \|z^k - z^\ast\|$,  and because $(\|z^j - z^\ast\|)_{j \geq 0}$ is monotonically nonincreasing (Corollary~\ref{cor:DRSaveragedconvergence}), it follows that
\begin{align*}
\|z_\lambda - x^\ast\| &\leq \|z_\lambda - z^\ast\| + \|z^\ast - x^\ast\| \leq \|z^0 - z^\ast\| + \|z^\ast - x^\ast\|. \numberthis \label{thm:drsnonergodic:eq:triangle}
\end{align*}
Thus, we have the bound:
\begin{align*}
f(x_f^k) + g(x_g^k) - f(x^\ast) - g(x^\ast) &\leq \inf_{\lambda \in [0, 1]}\frac{1}{4\gamma \lambda}\left(2\dotp{z_\lambda - x^\ast, z^k - z_\lambda} +2\left(1 - \frac{1}{2\lambda}\right) \|z_\lambda - z^k\|^2\right) \\
&\leq \frac{1}{\gamma }\|z_{1/2} - x^\ast\|\|z^k - z_{1/2}\|\\
&\stackrel{\eqref{thm:drsnonergodic:eq:triangle}}{\leq}  \frac{1}{\gamma }\left(\|z^0 - z^\ast\|+ \|z^\ast - x^\ast\|\right)\|z^k - z_{1/2}\|  \numberthis\label{eq:drsnonergodic:eq:oupper} \\
&\stackrel{\eqref{cor:DRSaveragedconvergence:eq:main}}{\leq}\frac{(\|z^0 - z^\ast\| + \|z^\ast - x^\ast\|)\|z^0 - z^\ast\|}{2\gamma \sqrt{\underline{\tau}(k+1)}}.
\end{align*}

The lower bound follows from the identity $x_g^k - x_f^k = ({1}/{2\lambda_k})(z^k - z^{k+1})$ and the fundamental lower inequality in Equation~(\ref{prop:DRSlower:eq:main}):
\begin{align*}
f(x_f^k) + g(x_g^k) - f(x^\ast) - g(x^\ast) \geq \frac{1}{2\gamma \lambda_k}\dotp{z^{k} - z^{k+1}, z^\ast - x^\ast}  &\geq -\frac{\|z^{k+1} - z^k\|\|z^\ast - x^\ast\|}{2\gamma\lambda_k}  \numberthis\label{eq:drsnonergodic:eq:olower} \\
&\stackrel{\eqref{cor:DRSaveragedconvergence:eq:main}}{\geq} -\frac{\|z^0 - z^\ast\|\|z^\ast - x^\ast\|}{2\gamma \sqrt{\underline{\tau}(k+1)}}.
\end{align*}

Finally, the $o(1/\sqrt{k+1})$ convergence rate follows from Equations~(\ref{eq:drsnonergodic:eq:oupper}) and~(\ref{eq:drsnonergodic:eq:olower}) combined with Corollary~\ref{cor:DRSaveragedconvergence} {because each upper bound is of the form (bounded quantity)$\times\sqrt{\text{FPR}}$, and $\sqrt{\text{FPR}}$ has rate $o(1/\sqrt{k+1})$.}

Part~\ref{thm:drsnonergodic:part:1d} follows by the same analysis but uses Theorem~\ref{thm:PPAconvergence} to estimate the FPR convergence rate.
\qed\end{proof}

Whenever $f$ or $g$ is Lipschitz, we can compute the convergence rate of $f + g$ evaluated at the same point. The following theorem is analogous to Corollary~\ref{cor:drsergodiclipschitz} in the ergodic case.  The proof essentially follows by combining the nonergodic convergence rate in Theorem~\ref{thm:drsnonergodic} with the convergence rate of $\|x_f^k - x_g^k\| = (1/\lambda_k)\|z^{k+1} - z^k\|$ deduced in Corollary~\ref{cor:DRSaveragedconvergence}.

\begin{corollary}[Nonergodic convergence with Lipschitz assumption]\label{cor:drsnonergodiclipschitz}
Let the notation be as in Theorem~\ref{thm:drsnonergodic}. Suppose that $f$ (respectively $g$) is $L$-Lipschitz continuous on $\overline{B(x^\ast, \|z^0 - z^\ast\|)}$, and let $x^k = x_g^k$ (respectively $x^k = x_f^k$).  Then we have the convergence rates of the nonnegative term:
\begin{enumerate}
\item \label{cor:drsnonergodiclipschitz:part:general} In general, we have the bounds:
\begin{align*}
0 \leq f(x^k) + g(x^k) &- f(x^\ast) - g(x^\ast) \leq \frac{\left(\|z^0 - z^\ast\| + \|z^\ast - x^\ast\| + \gamma L\right)\|z^0 - z^\ast\|}{2\gamma \sqrt{\underline{\tau}(k+1)}}
\end{align*}
and $ f(x^k) + g(x^k) - f(x^\ast) - g(x^\ast) = o\left({1}/{\sqrt{k+1}}\right).$
\item \label{cor:drsnonergodiclipschitz:part:1d}If $\cH = \vR$  and $\lambda_k \equiv {1}/{2}$, then for all $k \geq 0$,
\begin{align*}
0 \leq f(x^{k+1}) + g(x^{k+1}) &- f(x^\ast) - g(x^\ast) \leq \frac{\left(\|z^0 - z^\ast\| + \|z^\ast - x^\ast\| + \gamma L\right)\|z^0 - z^\ast\|}{\sqrt{2}\gamma (k+1)} \end{align*}
and $f(x^{k+1}) + g(x^{k+1}) - f(x^\ast) - g(x^\ast) = o\left({1}/({k+1})\right).$
\end{enumerate}
\end{corollary}
\begin{proof}
We prove Part~\ref{cor:drsnonergodiclipschitz:part:general} first. First recall that $\|x_g^k - x_f^k\| = (1/(2\lambda_k))\|z^{k+1} - z^k\|$. In addition, $(x_f^j)_{j \geq 0}, (x_g^j)_{j \geq 0} \subseteq \overline{B(x^\ast, \|z^0 - z^\ast\|)}$ (See Section~\eqref{sec:ergodic}). Thus, it follows that
\begin{align*}
f(x^k) + g(x^k) - f(x^\ast) - g(x^\ast) &\leq f(x_f^k) + g(x_g^k) - f(x^\ast) - g(x^\ast)  + L\|x_f^k - x_g^k\|\\
&= f(x_f^k) + g(x_g^k) - f(x^\ast) - g(x^\ast)  + \frac{L\|z^{k+1} - z^k\| }{2\lambda_k} \numberthis \label{cor:drsnonergodiclipschitz:oterm}\\
&\stackrel{(\ref{cor:DRSaveragedconvergence:eq:main})}{\leq}f (x_f^k) + g(x_g^k) - f(x^\ast) - g(x^\ast)  + \frac{\gamma L\|z^0 - z^\ast\|}{2\gamma  \sqrt{\underline{\tau}(k+1)}}. \numberthis \label{cor:drsnonergodiclipschitz:Oterm}
\end{align*}
Therefore, the upper bound follows from Theorem~\ref{thm:drsnonergodic} and Equation~(\ref{cor:drsnonergodiclipschitz:Oterm}). In addition, the $o(1/\sqrt{k+1})$ bound follows from Theorem~\ref{thm:drsnonergodic} combined with Equation~(\ref{cor:drsnonergodiclipschitz:oterm}) and Corollary~\ref{cor:DRSaveragedconvergence} {because each upper bound is of the form (bounded quantity)$\times\sqrt{\text{FPR}}$, and $\sqrt{\text{FPR}}$ has rate $o(1/\sqrt{k+1})$.}

Part~\ref{cor:drsnonergodiclipschitz:part:1d} follows by the same analysis, but uses Theorem~\ref{thm:PPAconvergence} to estimate the FPR convergence rate.
\qed\end{proof}

\section{Optimal FPR rate and arbitrarily slow convergence}\label{sec:optimalFPR}
In this section, we provide two examples where the DRS algorithm converges slowly.  Both examples are a  special cases of the following example, which originally appeared in \cite[Section 7]{bauschke2013rate}.

\begin{example}[DRS applied to two subspaces]\label{example:subspaces}
Let $\cH = \ell_2^2(\vN) = \{ (z_j)_{j \geq 0} \mid \forall j \in \vN, z_j \in \vR^2, \sum_{i=0}^\infty \|z^j\|^2 < \infty \}$. 
Let $R_{\theta}$ denote counterclockwise rotation in $\vR^2$ by $\theta$ degrees.  Let $e_0 := (1, 0)$ denote the standard unit vector, and let $e_{\theta} := R_\theta e_0$.  Suppose that $(\theta_j)_{j\geq0}$ is a sequence of angles in $(0, {\pi}/{2}]$ such that $\theta_i \rightarrow 0$ as $i \rightarrow \infty$. We define two subspaces:
{\color{blue}\begin{align}
U :=  \bigoplus_{i =0}^\infty \vR e_0  && \mathrm{and} && V := \bigoplus_{i=0}^\infty \vR e_{\theta_i},
\end{align}
where $\vR e_0=\{\alpha e_0:\alpha\in \vR\}$ and $\vR e_{\theta_i}=\{\alpha e_{\theta_i}:\alpha\in \vR\}$.
}
See Figure~\ref{fig:DRScounterexample} for an illustration. 

Note that \cite[Section 7]{bauschke2013rate} shows the projection identities
\begin{align*}
(P_V)_i &= \begin{bmatrix} \cos^2(\theta_i) & \sin(\theta_i)\cos(\theta_i) \\ \sin(\theta_i)\cos(\theta_i) & \sin^2(\theta_i) \end{bmatrix} && \mathrm{and} && (P_U)_i = \begin{bmatrix} 1 & 0 \\ 0 & 0\end{bmatrix},
\end{align*}
the DRS operator identity
\begin{align}\label{eq:TPRSlinearslow}
T := (\TPRS)_{{1}/{2}} &= c_0R_{\theta_0} \oplus c_1R_{\theta_1} \oplus \cdots,
\end{align}
and that $(z^j)_{j \geq0}$ converges in norm to $z^\ast = 0$ for any initial point $z^0$. 
\qed
\end{example}

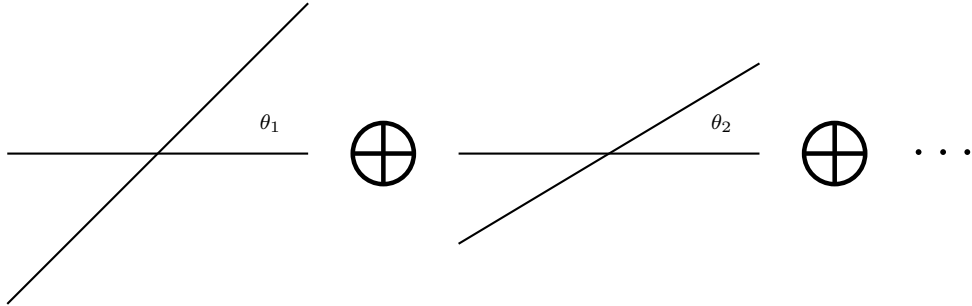
\begin{figure}[h!]
  \centering
\begin{tikzpicture}[scale=2]
\draw[-, thick] (-5, 0) -- (-3, 0);
\draw[-, thick] (-5, -1) -- (-3, 1);
\node at (-3.25, .20) {$\theta_1$};
\node at (-2.5, 0) {{\Huge$\bigoplus$}};

\draw[-, thick] (-2., 0) -- (.0, 0);
\draw[-, thick] (-2, -.60) -- (.0, .60);
\node at (-.25, .20) {$\theta_2$};
\node at (.5, 0) {{\Huge$\bigoplus$}};

\node[draw=none, font=\huge] (ellipsis1) at (1.25,0) {$ \cdots$};

\end{tikzpicture}
    \caption{Illustration of Example~\ref{example:subspaces}. Each pair of lines represents a $2$-dimensional component of $U \cup V$. The angles $\theta_k$ are converging to $0$.}
    \label{fig:DRScounterexample}
\end{figure}
\subsection{Optimal FPR rates}

The following theorem shows that the FPR estimates derived in Corollary~\ref{cor:DRSaveragedconvergence} are essentially optimal.  We note that this is the first optimality result for the FPR of the DRS iteration in the case of variational problems.
\begin{theorem}[Lower FPR complexity of DRS]\label{thm:optimalFPR}
There exists a Hilbert space $\cH$ and two closed subspaces $U$ and $V$ with zero intersection, $U \cap V = \{0\}$, such that for every $\alpha > {1}/{2}$, there exists $z^0 \in \cH$ such that if $(z^j)_{j \geq 0}$ is generated by $T = (\TPRS)_{1/2}$ applied to $f = \iota_V$ and $g = \iota_U$, then for all $k \geq 1$, we have the bound:
\begin{align*}
\|Tz^k - z^k\|^2 &\geq \frac{1}{(k+1)^{2\alpha}}.
\end{align*}
\end{theorem}
\begin{proof}
We assume the setting of Example~\ref{example:subspaces}.  For all $i \geq 0$ set $c_i = (i/(i+1))^{1/2}$, and let $w^0 = (w^0_{j})_{j \geq 0} \in \cH$, where each $w^0_{i} \in \vR^2$ satisfies $\|w^0_i\| =  \sqrt{2\alpha e}(i+1)^{-(1+2\alpha)/2}$. Then for all $k \geq 1$,
\begin{align*}
\|T^{k}w^0\|^2 &= \sum_{i=0}^\infty c_i^{2k}\|w^0_{i}\|^2 \geq \sum_{i=k}^\infty\left(\frac{i}{i+1}\right)^k\frac{2\alpha e}{(i+1)^{1+ 2\alpha}} \geq \frac{1}{(k+1)^{2\alpha}}
\end{align*}
where we have used the bound $(i/(i+1))^k \geq e^{-1}$ for $i \geq k$ and the lower integral approximation of the sum.

Now we will show that $w^0$ is in the range of $I - T$. Indeed, for all $i \geq 1$ each block of $I-T$ is of the form 
\begin{align*}
I_{\vR^2} - \cos(\theta_i)R_{\theta_i} &= \begin{bmatrix} \sin^2(\theta_i) & \sin(\theta_i)\cos(\theta_i) \\-\sin(\theta_i) \cos(\theta_i) & \sin^2(\theta_i) \end{bmatrix} = \begin{bmatrix} \frac{1}{i+1}&  \frac{\sqrt{i}}{i+1} \\  -  \frac{\sqrt{i}}{i+1}   & \frac{1}{i+1} \end{bmatrix}. \numberthis \label{eq:trigmatrix}
\end{align*}
Therefore, the point $z^0 = (\sqrt{2\alpha e}((1/(j+1)^\alpha, 0))_{j\geq 0} \in \cH$ has image
\begin{align*}
w^0 &= (I-T)z^0 = \left(\sqrt{2\alpha e}\left(\frac{1}{(j+1)^{\alpha + 1}},  \frac{-\sqrt{j}}{(j+1)^{\alpha + 1}}\right)\right)_{j \geq 0}.
\end{align*}
 In addition, for all $i \geq 1$, we have $\|w^0_i\| = \sqrt{2\alpha e}(i+1)^{-(1+2\alpha)/2}$, and the inequality follows.\qed\end{proof}

\begin{remark}
The proof of Theorem~\ref{thm:optimalFPR} crucially relies on the strictness of inequality, $\alpha > 1/2$: if $\alpha = 1/2$, then $\|z^0\| = \infty.$
\end{remark}

\subsubsection{Notes on Theorem~\ref{thm:optimalFPR}}\label{sec:optimalitydisc}

With this new optimality result in hand, we can make the following list of optimal FPR rates, not to be confused with optimal rates in objective error, for a few standard splitting schemes:

 {\bf Proximal point algorithm (PPA):} For the general class of monotone operators, the counterexample furnished in \cite[Remarque 4]{brezis1978produits} shows that there exists a maximal monotone operator $A$ such that when iteration~(\ref{eq:KMiteration}) is applied to the resolvent $J_{\gamma A}$, the rate $o(1/(k+1))$ is tight.  In addition, if $A = \partial f$ for some closed, proper, and convex function $f$, then the FPR rate improves to $O(1/(k+1)^2)$ \cite[Th\'eor\`eme 9]{brezis1978produits}. We improve this result to $o(1/(k+1)^2)$ in Theorem~\ref{thm:PPAconvergence}. This result appears to be new and is optimal by \cite[Remarque 6]{brezis1978produits}.

 {\bf Forward backward splitting (FBS):} The FBS method reduces to the proximal point algorithm when the differentiable (or single valued operator) term is trivial.  Thus, for the general class of monotone operators, the $o(1/(k+1))$ FPR rate is optimal by \cite[Remarque 4]{brezis1978produits}.  We improve this rate to $o(1/(k+1)^2)$ in Theorem~\ref{thm:PPAconvergence}. This result appears to be new, and is optimal by \cite[Remarque 6]{brezis1978produits}.

 {\bf Douglas-Rachford splitting/ADMM:} Theorem~\ref{thm:optimalFPR} shows that the optimal FPR rate is $o(1/(k+1))$. Because the DRS iteration is equivalent to a proximal point iteration applied to a special monotone operator \cite[Section 4]{eckstein1992douglas}, Theorem~\ref{thm:optimalFPR} provides an alternative counterexample to \cite[Remarque 4]{brezis1978produits}. In particular, Theorem~\ref{thm:optimalFPR} shows that, in general, there is no closed, proper, and convex function $f$ such that  $(\TPRS)_{1/2} = \prox_{\gamma f}$. In the one dimensional case, we improve the FPR to $o({1}/{(k+1)^2})$ in Theorem~\ref{thm:1DDRS}.

 {\bf Miscellaneous methods:} By similar arguments we can deduce the tight FPR iteration complexity for the following methods, each of which at least has rate $o(1/{(k+1)})$ by Theorem~\ref{prop:averagedconvergence}: \emph{Standard Gradient descent} $o(1/(k+1)^2)$: (the rate follows from Theorem~\ref{thm:PPAconvergence}. Optimality follows from the fact that PPA is equivalent to gradient descent on Moreau envelope \cite[Proposition 12.29]{bauschke2011convex} and \cite[Remarque 4]{brezis1978produits}); \emph{Forward-Douglas Rachford splitting} \cite{briceno2012forward}: $o(1/(k+1))$ (choose a trivial cocoercive operator and use Theorem~\ref{thm:optimalFPR}); \emph{Chambolle and Pock's primal-dual algorithm} \cite{chambolle2011first} $o({1}/{(k+1)})$:  (reduce to DRS $(\sigma = \tau = 1)$ \cite[Section 4.2]{chambolle2011first} and apply Theorem~\ref{thm:optimalFPR} using the transformation $z^k = \text{primal}_k + \text{dual}_k$ \cite[Equation~(24)]{chambolle2011first} and the lower bound 
\begin{align*}
\|z^{k+1} - z^k\|^2 \leq 2\|\text{primal}_{k+1} - \text{primal}_k\|^2 + 2\|\text{dual}_{k+1} - \text{dual}_{k}\|^2;
\end{align*}
 \emph{V\~{u}/Condat's primal-dual algorithm} \cite{vu2013splitting,condat2013primal} $o({1}/{(k+1)})$: (reduces to Chambolle and Pock's method \cite{chambolle2011first}). 

Note that the rate established in Theorem~\ref{prop:averagedconvergence} has broad applicability, and this list is hardly extensive. For PPA, FBS, and standard gradient descent, the FPR always has rate that is the square of the objective value convergence rate.  We will see that the same is true for DRS in Theorem~\ref{thm:DRSobjectiveslow}.

\subsection{Arbitrarily slow convergence}

In \cite[Section 7]{bauschke2013rate}, the DRS setting in Example~\ref{example:subspaces} is shown to converge in norm, but not linearly. We improve their result by showing that a proper choice of parameters yields arbitrarily slow convergence in norm.

The following technical lemma will help us construct a sequence that convergenes arbitrarily slowly. The idea of the proof follows directly from the proof of \cite[Theorem 4.2]{franchetti1986neumann}, which shows that the alternating projection algorithm can converge arbitrarily slowly.
\begin{lemma}\label{lem:slowconvergencesequence}
Suppose that $h : \vR_+ \rightarrow (0, 1)$ is a function that is strictly decreasing to zero such that $\{1/(j+1) \mid j \in \vN \backslash \{0\}\} \subseteq \range(h).$ Then there exists a monotonic sequence $(c_j)_{j \geq 0} \subseteq (0, 1)$ such that $c_k \rightarrow 1^-$ as $k \rightarrow \infty$  and an increasing sequence of integers $(n_j)_{j \geq 0} \subseteq \vN \cup \{0\}$ such that for all $k \geq 0$,
\begin{align}
\frac{c_{n_k}^{k+1}}{n_k+1} > h(k+1)e^{-1}.
\end{align}
\end{lemma}
\begin{proof}
Let $h_2$ be the inverse of the strictly increasing function $(1/h) - 1$, let $[x]$ denote the integer part of $x$, and for all $k \geq 0$, let
\begin{align}
c_{k} &= \frac{ h_2(k+1)}{1 + h_2(k+1)}.
\end{align}
Note that because $\{1/(j+1) \mid j \in \vN \backslash \{0\}\} \subseteq \range(h)$, $c_k$ is well defined. Indeed, $k + 1\in \dom(h_2) \cap \vN$ if, and only if, there is a $y\in \vR_+$ such that $(1/h(y)) - 1 = k+1 \Longleftrightarrow  h(y) = 1/(k+2)$. It follows that $(c_j)_{j \geq 0}$ is monotonic and $c_k \rightarrow 1^-$.  

Now, for all $x \geq 0$, we have $h_2^{-1}(x) = {1}/{h(x)} - 1 \leq [{1}/{h(x)}]$, thus, $x \leq h_2([{1}/{h(x)}])$.  To complete the proof, choose $n_k \geq 0$ such that $n_k + 1 = [{1}/{h(k+1)}]$ and note that
\begin{align*}
\frac{c_{n_k}^{k+1}}{n_k+1} &\geq h(k+1)\left(\frac{k+1}{1+(k+1)}\right)^{k+1} \geq h(k+1) e^{-1}. \qquad \qed 
\end{align*}
\end{proof}

\begin{theorem}[Arbitrarily slow convergence of DRS]\label{thm:arbitrarilyslow}
There is a point $z_0 \in \ell_2^2(\vN)$, such that for every function $h : \vR_+ \rightarrow (0, 1)$ that strictly decreases to zero, there exists two closed subspaces $U$ and $V$ with zero intersection, $U\cap V = \{0\}$, such that the relaxed PRS sequence $(z^j)_{j \geq 0}$ generated with the functions $f = \iota_{V}$ and $g = \iota_{U}$ and relaxation parameters $\lambda_k\equiv {1}/{2}$ satisfies the bound
\begin{align*}
  \|z^k - z^\ast\| \geq e^{-1} h(k)
  \end{align*}
but $(\|z^j - z^\ast\|)_{j \geq 0}$ converges to $0$.  
\end{theorem}
\begin{proof}
We assume the setting of Example~\ref{example:subspaces}.  Suppose that $z^0 = (z^0_{j})_{j \geq 0}$, where for all $k \geq 0$,  $z^0_k \in \vR^2$, and $\|z^0_{k}\| = {1}/({k+1})$.  Then it follows that $\|z^0\|_{\cH}^2 = \sum_{i=0}^\infty 1/(k+1)^{2} < \infty$ and so $z^0 \in \cH$.  Thus, for all $k, n \geq 0$
\begin{align}
\|T^{k+1} z^0 \| &\geq c_n^{k+1}\|z^0_n\| = \frac{1}{n+1}c_{n}^{k+1}.
\end{align}
Therefore, we can achieve arbitrarily slow convergence by picking $(c_j)_{j \geq 0}$, and a subsequence $(n_j)_{j \geq 0} \subseteq \vN$ using Lemma~\ref{lem:slowconvergencesequence}.
\qed\end{proof}

\section{Optimal objective rates} \label{sec:PRSergodicoptimality}

In this section we construct four examples that show the nonergodic  and ergodic convergence rates in Corollary~\ref{cor:drsnonergodiclipschitz} and Theorem~\ref{thm:drsergodic} are optimal up to constant factors.  In particular, we provide examples of optimal ergodic convergence in the minimization case and in the feasibility case, where no objective is driving the minimization.

\subsection{Ergodic convergence of feasibility problems}\label{example:PRSergodic}

\begin{proposition}\label{prop:ergodicfeasibilityopt}
The ergodic feasibility convergence rate in Equation~\eqref{thm:drsergodic:eq:feasibility} is optimal up to a factor of two.
\end{proposition}
\begin{proof}
Figure~\ref{fig:PRSerg} shows Algorithm~\ref{alg:DRS} with $\lambda_k = 1$ for all $k \geq 0$ (i.e. PRS) applied to the functions $f = \iota_{\{(x_1, x_2) \in \vR^2 | x_1 = 0\}}$ and $g  = \iota_{\{(x_1, x_2) \in \vR^2 | x_2 = 0\}}$ with the initial iterate $z^0 = (1, 1) \in \vR^2$. Because $\TPRS  = -I_{\cH}$,  it is easy to see that the only fixed point of $\TPRS $ is $z^\ast = (0, 0)$. In addition, the following identities are satisfied:
\begin{align*}
x_g^k = \begin{cases} (1, 0) & \text{even}~k; \\ (-1, 0) & \text{odd}~k. \end{cases} && z^{k} = \begin{cases} (1, 1) & \text{even}~k; \\ (-1, -1) & \text{odd}~k. \end{cases}  && x_f^k = \begin{cases} (0, -1) & \text{even}~k; \\ (0, 1) & \text{odd}~k.\end{cases}
\end{align*}
Thus, the PRS algorithm oscillates around the solution $x^\ast = (0, 0)$.  However, note that the averaged iterates satisfy:
\begin{align*}
\overline{x}_g^k = \begin{cases} (\frac{1}{k+1}, 0) &  \text{even}~k; \\ (0, 0) & \text{odd}~k. \end{cases} && \mathrm{and} && \overline{x}_f^k =  \begin{cases} (0, \frac{-1}{k+1}) & \text{even}~k; \\ (0, 0) & \text{odd}~k. \end{cases}
\end{align*}
It follows that $\|\overline{x}_g^k - \overline{x}_f^k\| = {(1/(k+1))} {\|(1, -1)\|}= (1/(k+1)){\|z^0 - z^\ast\|}$, for all $k \geq 0$.
\begin{figure}[h!]
  \centering
\begin{tikzpicture}[scale=2]
\draw[-] (-2, 0) -- (2, 0) node[below]{$x_1$};
\draw[-] (0, -2) -- (0, 2) node[right]{$x_2$};
\draw[->, color=red, very thick] (1, 1) -- (1,-1) -- (-1, -1) -- (-1, 1) -- (1, 1);
\draw[fill, darkgray] (1, 0) circle [radius=.04];
\draw[fill, blue] (0, -1) circle [radius=.04];
\draw[fill, darkgray] (-1, 0) circle [radius=.04];
\draw[fill, blue] (0, 1) circle [radius=.04];
\draw[fill] (-1, -1) circle [radius=.04];
\draw[fill] (1, 1) circle [radius=.04];
\draw[fill] (0, 0) circle [radius=.04];
\node[above right] at (1, 1) {$z^{\text{even}}$};
\node[below left] at (-1, -1) {$z^{\text{odd}}$};
\node[above right] at (1, 0) {$x_g^{\text{even}}$};
\node[below right] at (0, -1) {$x_f^{\text{even}}$};
\node[above left] at (-1, 0) {$x_g^{\text{odd}}$};
\node[above left] at (0, 1) {$x_f^{\text{odd}}$};
\node[above left] at (.05, .05) {$x^\ast$};

\draw[fill, darkgray] (.3333, 0) circle [radius=.04];
\draw[fill, darkgray] (.2, 0) circle [radius=.030];
\draw[fill, darkgray] (.1429, 0) circle [radius=.020];


\draw[fill, blue] (0, -.3333) circle [radius=.04];
\draw[fill, blue] (0, -.2) circle [radius=.03];
\draw[fill, blue] (0, -.1429) circle [radius=.02];

\end{tikzpicture}
    \caption{Example \ref{example:PRSergodic} of PRS. $z^k$ hops between  $(1,1)$ and $(-1,-1)$ while the ergodic iterates $\overline{x}_g^k$ and $\overline{x}_f^k$ (dots of decreasing size) approach $x^\ast$.}
    \label{fig:PRSerg}
\end{figure}
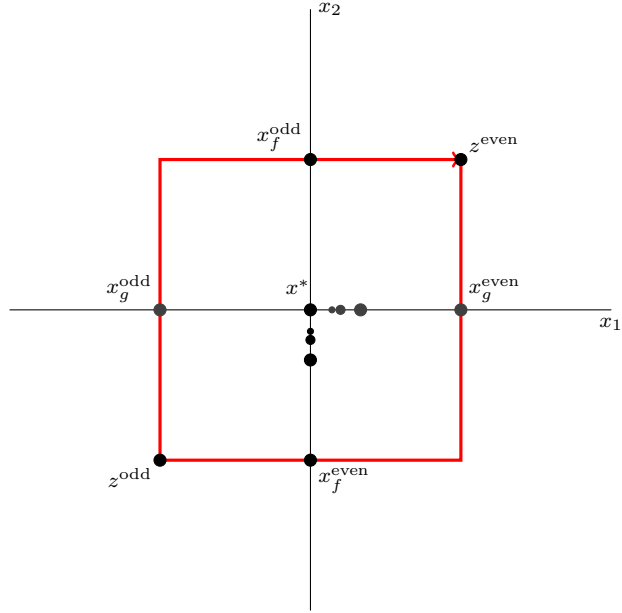

\qed\end{proof}

\subsection{Ergodic convergence of minimization problems}
In this section, we will construct an example where the ergodic rates of convergence in Section~\ref{sec:ergodic} are optimal up to constant factors.  In addition, the example we construct only converges in the ergodic sense and diverges otherwise. Throughout this section, we let $\gamma =1$ and $\lambda_k \equiv 1$, we work in the Hilbert space $\cH = \vR$, and we consider the following objective functions: for all $x \in \vR$, let
\begin{align*}
g(x) = 0, && \mathrm{and} && f(x) &= |x| \numberthis \label{eq:fgergodiccounter}.
\end{align*}
Recall that for all $x \in \vR$
\begin{align*}
\prox_{ g}(x) = x,  && \mathrm{and} && \prox_{ f}(x) = \max\left(|x| - 1, 0\right)\sign(x). \numberthis \label{eq:gproxcounter}
\end{align*}

The following lemma characterizes the minimizer of $f+g$ and the fixed points of $\TPRS $.  The proof is simple so we omit it.
\begin{lemma}\label{lem:optimalityofabs}
The minimizer of $f+g$ is unique and equal to $0 \in \vR$.  Furthermore, $0$ is the unique fixed point of $\TPRS $.
\end{lemma}

Because of Lemma~\ref{lem:optimalityofabs}, we will use the notation:
\begin{align*}
z^\ast = 0 && \mathrm{and} && x^\ast = 0. \numberthis \label{eq:absoptimality}
\end{align*}
We are ready to prove our main optimality result.

\begin{proposition}[Optimality of ergodic convergence rates]\label{prop:ergodicoptimality}
Suppose that $z^0 = 2 - \varepsilon$ for some $\varepsilon \in (0, 1)$. Then the PRS algorithm applied to $f$ and $g$ with initial point $z^0$ does not converge.

Furthermore,  as $\varepsilon$ goes to $0$, the ergodic objective convergence rate in Theorem~\ref{thm:drsergodic} is tight, and the ergodic objective convergence rate in Corollary~\ref{cor:drsergodiclipschitz} is tight up to a factor of ${5}/{2}$.  In addition, the feasibility convergence rate of Theorem~\ref{thm:drsergodic} is tight up to a factor of $4$.
\end{proposition}
\begin{proof}
We will now compute the sequences $(z^j)_{j \geq 0}$, $(x_g^j)_{j \geq 0}$, and $(x_f^j)_{j \geq 0}$.  We proceed by induction: First $x_g^0 = \prox_{\gamma g} (z^0) = z^0$ and $x_f^0 = \prox_{\gamma f}\left(2x_g^0 - z^0\right) = \max\left( {|z^0| - 1}, 0\right)\sign(z^0) = 1- \varepsilon.$ Thus, it follows that $z^1 = z^0 + 2(x_f^0 - x_g^0) = 2-\varepsilon + 2(1-\varepsilon - (2- \varepsilon)) = z^0 = -\varepsilon$.  Similarly, $x_g^1 = z^1 = -\varepsilon$.  Finally,  $x_f^1 = \max\left(\varepsilon - {1}, 0\right)\sign(-\varepsilon) = 0$ and $z^2 = z^1 + 2(x_f^1 - x_g^1) = z^1 + 2(\varepsilon) = \varepsilon$. Thus, by induction we have the following identities: For all $k \geq 1$,
\begin{align*}
z^k = (-1)^k\varepsilon, &&  x_g^k &= (-1)^k\varepsilon, && x_f^k = 0.\numberthis \label{example:ergodicex1:eq:3}
\end{align*}
Notice that that $(z^j)_{j \geq 0}$ and $(x_g^j)_{j \geq 0}$ do not converge, but they oscillate around the fixed point of $\TPRS $.

We will now compute the ergodic iterates:
\begin{align*}
\overline{x}_g^k = \frac{1}{k+1} \sum_{i=0}^k x_g^i \stackrel{(\ref{example:ergodicex1:eq:3})}{=} \begin{cases} \frac{2-\varepsilon}{k+1} & \text{if } k \text{ is even}; \\ \frac{2 -2 \varepsilon}{k+1} & \text{otherwise.} \end{cases}  && \mathrm{and} && \overline{x}_f^k = \frac{1}{k+1} \sum_{i=0}^k x_f^i \stackrel{(\ref{example:ergodicex1:eq:3})}{=} \frac{1-\varepsilon}{k+1} \numberthis \label{example:ergodicex1:eq:4}.
\end{align*}
Let us use these formulas to compute the objective values:
\begin{align*}
f(\overline{x}_f^k) + g(\overline{x}_f^k) - f(0) - g(0) \stackrel{(\ref{example:ergodicex1:eq:4})}{=} \frac{1-\varepsilon}{k+1} && \mathrm{and} && f(\overline{x}_g^k) + g(\overline{x}_g^k) - f(0) - g(0) \stackrel{(\ref{example:ergodicex1:eq:4})}{=}  \begin{cases} \frac{2-\varepsilon}{k+1} & \text{if } k \text{ is even}; \\ \frac{2 -2 \varepsilon}{k+1} & \text{otherwise.} \end{cases}\numberthis \label{example:ergodicex1:eq:6}
\end{align*}
We will now compare the theoretical bounds from Theorem~\ref{thm:drsergodic} and Corollary~\ref{cor:drsergodiclipschitz} with the rates we observed in Equation~(\ref{example:ergodicex1:eq:6}).  Theorem~\ref{thm:drsergodic} bounds the objective error at $\overline{x}_f^k$ by 
\begin{align*}
\frac{|z^0 - x^\ast|^2}{4(k+1)} &= \frac{4 - 4\varepsilon}{4(k+1)}  +\frac{ \varepsilon^2} {4(k+1)} = \frac{1- \varepsilon}{k+1} + \frac{ \varepsilon^2} {4(k+1)}. \numberthis \label{example:ergodicex1:eq:7}
\end{align*}
By taking $\varepsilon$ to $0$, we see that this bound is tight.

Because $f$ is $1$-Lipschitz continuous, Corollary~\ref{cor:drsergodiclipschitz} bounds the objective error at $\overline{x}_g^k$ with
\begin{align*}
\frac{|z^0 - x^\ast|^2}{4(k+1)} + \frac{2|z^0 - z^\ast|}{(k+1)} &\stackrel{(\ref{example:ergodicex1:eq:7})}{=} \frac{1- \varepsilon}{k+1} + \frac{ \varepsilon^2} {4(k+1)} + 2\frac{2- \varepsilon}{k+1} = \frac{5 - 3\varepsilon}{k+1} + \frac{\varepsilon^2}{4(k+1)}.\numberthis \label{example:ergodicex1:eq:8}
\end{align*}
As we take $\varepsilon$ to $0$, we see that this bound it tight up to a factor of ${5}/{2}$.

Finally, consider the feasibility convergence rate:
\begin{align*}
|\overline{x}_g^k - \overline{x}_f^k| & \stackrel{(\ref{example:ergodicex1:eq:3})}{=}  \begin{cases} \frac{1}{k+1} & \text{if } k \text{ is even}; \\ \frac{1-\varepsilon}{k+1} & \text{otherwise.} \end{cases}.\numberthis \label{example:ergodicex1:eq:9}
\end{align*}
Theorem~\ref{thm:drsergodic} predicts the following upper bound for Equation~(\ref{example:ergodicex1:eq:9}):
\begin{align*}
\frac{2|z^0 - z^\ast|}{k+1} &= 2\frac{2 - \varepsilon}{k+1} = \frac{4 - 2\varepsilon}{k+1}. \numberthis\label{example:ergodicex1:eq:10}
\end{align*}
By taking $\varepsilon$ to $0$, we see that this bound is tight up to a factor of $4$.
\qed\end{proof}

\subsection{Optimal nonergodic objective rates}

Our aim in this section is to show that if $\lambda_k \equiv 1/2$, then the non-ergodic convergence rate of $o({1}/{\sqrt{k+1}})$ in Corollary~\ref{cor:drsnonergodiclipschitz} is essentially tight.  In particular, for every $\alpha > {1}/{2}$, we provide examples of $f$ and $g$ such that $f$ is $1$-Lipschitz and 
\begin{align*}
f(x_g^k) + g(x_g^k) - f(x^\ast) - g(x^\ast) = \Omega\left(\frac{1}{(k+1)^\alpha}\right).
\end{align*}

Throughout this section, we will be working with the proximal operator of a distance functions.
\begin{proposition}\label{prop:distprox}
Let $C$ be a closed, convex subset of $\cH$ and let $d_C(x) = \min_{y \in C} \|x - y\|$. Then $d_C(x)$ is $1$-Lipschitz and for all $x \in \cH$
\begin{align}
\prox_{\gamma d_C}(x) = \theta P_{C}(x) + (1-\theta)x && \mathrm{where} && \theta =\begin{cases}
\frac{\gamma}{d_C(x)} & \text{if } \gamma \leq d_C(x); \\
1 & \text{otherwise}.
\end{cases}
\end{align}
\end{proposition}
\begin{proof}
Follows directly from the formula for the subgradient of $d_C$ \cite[Example 16.49]{bauschke2011convex}.
\qed\end{proof}

Proposition~\ref{prop:distprox} says that $\prox_{\gamma d_C}(x)$ reduces to a projection map whenever $x$ is close enough to $C$.  Proposition~\ref{prop:sameproxsequencedist} constructs a family of examples such that if $\gamma$ is chosen large enough, then DRS does not distinguish between indicator functions and distance functions.

\begin{proposition}\label{prop:sameproxsequencedist}
Suppose that $V$ and $U$ are linear subspaces of $\cH$ and $U \cap V = \{0\}$. If $\gamma \geq \|z^0\|$ and $\lambda_k = {1}/{2}$ for all $k \geq 0$, then Algorithm~\ref{alg:DRS} applied to the either pair of objective functions $(f = \iota_V, g = \iota_U)$ and $(f = d_V, g = \iota_U)$ produces the same sequence $(z^j)_{j \geq 0}$
\end{proposition}
\begin{proof}
Let $(z^j)_{j \geq 0}$ be the sequence generated by the functions $(f = \iota_V, g = \iota_U)$. Observe that $x^\ast = 0$ is a minimizer of both functions pairs and $z^\ast = 0$ is a fixed point of $(\TPRS)_{{1}/{2}}$. In particular, we set $\tnabla \iota_V(x^\ast) = P_V(\refl_{g}(z^\ast)) - x^\ast = 0$.  Therefore, we just need to show that $\prox_{\gamma d_V} (\refl_{g}(z^k)) = P_V(\refl_{g}(z^k))$ for all $k \geq 0$.  Note that by definition, $x_{\iota_V}^k = P_V(\refl_{g}(z^k))$ and $\tnabla \iota_V(x_{\iota_V}^k) = \refl_{ g}(z^k) - P_V(\refl_{ g}(z^k)) \in \partial \iota_{V}(x_{\iota_V}^k)$. In view of Proposition~\ref{prop:distprox}, the identity will follow if
\begin{align*}
\gamma \geq d_V(\refl_{g}(z^k)) = \|\refl_{ g}(z^k) - P_V(\refl_{ g}(z^k))\| &=  \|\tnabla \iota_V(x_{\iota_V}^k) \| = \|\tnabla \iota_V(x_{\iota_V}^k) - \tnabla \iota_V(x^\ast)\|.
\end{align*}
However, this is always the case because
\begin{align*}
 \|\tnabla \iota_V(x_{\iota_V}^k) - \tnabla \iota_V(x^\ast)\|^2 + \|x_{\iota_V}^k - x^\ast\|^2 &\stackrel{\eqref{cor:proxcontraction:eq:main}}{\leq} \|\refl_{g}(z^k) - \refl_{g}(z^\ast)\|^2 \leq \|z^k - z^\ast\|^2 \leq \|z^0 - z^\ast\|^2 = \|z^0\|^2 \leq \gamma^2.
\end{align*}
\qed\end{proof}

\begin{theorem}\label{prop:lowerboundobjective}
Assume the notation of Theorem~\ref{thm:optimalFPR}. Then for all $\alpha > {1}/{2}$, there exists a point $z^0\in \cH$ such that if $\gamma \geq \|z^0\|$ and $(z^j)_{j \geq 0}$ is generated by DRS applied to the functions $(f = d_V, g = \iota_U)$, then $d_{V}(x^\ast) = 0$ and 
\begin{align}\label{prop:lowerboundobjective:eq:main}
d_{V}(x_g^k) &= \Omega\left(\frac{1}{(k+1)^\alpha}\right).  
\end{align}
\end{theorem}
\begin{proof}
 Let $z^0 = (({1}/{(j+1)^\alpha}, 0))_{j \geq 0} \in \cH$.  Now, choose $\gamma ^2 \geq \|z^0\|^2 = \sum_{i=0}^\infty {1}/{(i+1)^{2\alpha}}$.  Define $w^0 \in \cH$ using Equation~\eqref{eq:trigmatrix}:
 \begin{align*}
 w^0 &= (I - T)z^0 = \left(\frac{1}{(j+1)^\alpha}\left(\frac{1}{j+1}, \frac{-\sqrt{j}}{j+1}\right)\right)_{j \geq 0}.
 \end{align*}
 Then $\|w^0_i\| = {1}/{(1+i)^{{(1+2\alpha)}/{2}}}$. 
 
 Now we will calculate $d_{V}(x_g^k) = \|P_Vx_g^k - x_g^k\|$. First, recall that  $T^k = c_0^kR_{k \theta_0} \oplus c_1^k R_{k\theta_1} \oplus \cdots $, where 
\begin{align*}
R_{\theta} = \begin{bmatrix} \cos(\theta) & -\sin(\theta) \\ \sin(\theta) & \cos(\theta) \end{bmatrix}.
\end{align*}
Thus, 
 \begin{align*}
 x_g^k := P_{U}(z^k) = \left(\begin{bmatrix} 1 & 0 \\ 0 & 0\end{bmatrix} c_j^k R_{k\theta} \left(\frac{1}{(j+1)^\alpha}, 0\right) \right)_{j \geq 0} &= \left(\begin{bmatrix} 1 & 0 \\ 0 & 0 \end{bmatrix} c_j^k \frac{1}{(j+1)^\alpha}\left( \cos(k \theta_j), \sin(k\theta_j)\right)  \right)_{j \geq 0} \\
&= \left(c_j^k \frac{\cos(k\theta_j)}{(j+1)^\alpha}(1, 0)  \right)_{j \geq 0}.
 \end{align*}
 Furthermore, from the identity
 \begin{align*}
(P_V)_i &= \begin{bmatrix} \cos^2(\theta_i) & \sin(\theta_i)\cos(\theta_i) \\ \sin(\theta_j)\cos(\theta_i) & \sin^2(\theta_i) \end{bmatrix} = \begin{bmatrix} \frac{i}{i+1} & \frac{\sqrt{i}}{i+1} \\ \frac{\sqrt{i}}{i+1} & \frac{1}{i+1}\end{bmatrix},
\end{align*}
we have
 \begin{align*}
 P_V x_g^k &= \left(c_j^k \frac{\cos(k\theta_j)}{(j+1)^\alpha}\left(\frac{j}{j+1}, \frac{\sqrt{j}}{j+1}\right)  \right)_{j \geq 0}.
 \end{align*}
Thus, the the difference has the following form:
\begin{align*}
x_g^k - P_Vx_g^k &= \left(c_j^k \frac{\cos(k\theta_j)}{(j+1)^\alpha}\left(\frac{1}{j+1}, \frac{-\sqrt{j}}{j+1}\right)  \right)_{j \geq 0}.
\end{align*}
Now we derive the lower bound:
\begin{align*}
d_{V}(x_g^k)^2 =\|x_g^k - P_Vx_g^k\|^2 =\sum_{i=0}^\infty c_i^{2k} \frac{\cos^2(k \theta_i)}{(i+1)^{2\alpha + 1}} &= \sum_{i=0}^\infty c_i^{2k} \frac{\cos^2\left(k \cos^{-1}\left(\sqrt{\frac{i}{i+1}}\right)\right)}{(i+1)^{2\alpha + 1}} \\
&\geq \frac{1}{e} \sum_{i=k}^\infty  \frac{\cos^2\left(k \cos^{-1}\left(\sqrt{\frac{i}{i+1}}\right)\right)}{(i+1)^{2\alpha + 1}} \numberthis \label{eq:nonergodicoptimalitysum}.
\end{align*}

The next several lemmas will focus on estimating the order of the sum in Equation~\eqref{eq:nonergodicoptimalitysum}.  After which, Theorem~\ref{prop:lowerboundobjective} will follow from Equation~\eqref{eq:nonergodicoptimalitysum} and Lemma~\ref{lem:integralestimation}, below. This completes the proof of Theorem~\ref{prop:lowerboundobjective}.\qed
\end{proof}

\begin{lemma}\label{lem:integralapproximationbound}
Let $h : \vR_+ \rightarrow \vR_+$ be a continuously differentiable function such that $h \in L_1(\vR_+)$ and $\sum_{i=1}^\infty h(i) < \infty$.  Then for all positive integers $k$, 
\begin{align*}
\left| \int_{k}^\infty h(y)dy - \sum_{i=k}^\infty h(i)\right| &\leq \sum_{i=k}^\infty \max_{y \in [i, i+1]} \left| h'(y) \right|.
\end{align*}
\end{lemma}
\begin{proof}
We just apply the Mean Value Theorem and combine the integral with the sum
\begin{align*}
\left|\int_{k}^\infty h(y)dy - \sum_{i=k}^\infty h(i)\right| \leq \left| \sum_{i=k}^\infty \int_{i}^{i+1} (h(y) - h(i))dy\right| &\leq \sum_{i=k}^\infty \int_{i}^{i+1} |h(y) - h(i)|dy  \leq \sum_{i=k}^\infty \max_{y \in [i, i+1]} |h'(y)|. \qquad \qed
\end{align*}
\end{proof}

The following Lemma will quantify the deviation of integral from the sum. 

\begin{lemma}\label{lem:integralapproximationofcosinesum}
The following approximation bound holds:
\begin{align}\label{eq:integralsumboundcosine}
\left| \sum_{i=k}^\infty  \frac{\cos^2\left(k \cos^{-1}\left(\sqrt{\frac{i}{i+1}}\right)\right)}{(i+1)^{2\alpha + 1}}  -  \int_{k}^\infty  \frac{\cos^2\left(k \cos^{-1}\left(\sqrt{\frac{y}{y+1}}\right)\right)}{(y+1)^{2\alpha + 1}} dy\right| = O\left(\frac{1}{(k+1)^{2\alpha + {1}/{2}}}\right).
\end{align}
\end{lemma}
\begin{proof}
 We will use Lemma~\ref{lem:integralapproximationbound} with 
 \begin{align*}
 h(y) &=  \frac{\cos^2\left(k \cos^{-1}\left(\sqrt{\frac{y}{y+1}}\right)\right)}{(y+1)^{2\alpha + 1}}.
 \end{align*}
 to deduce an upper bound on the absolute value. Indeed, 
 \begin{align*}
| h'(y)| &= \left|\frac{k\sin\left(k\cos^{-1}\left(\sqrt{\frac{y}{y+1}}\right)\right)\cos\left(k\cos^{-1}\left(\sqrt{\frac{y}{y+1}}\right)\right)}{\sqrt{y}(y+1)(y+1)^{2\alpha + 1}} - \frac{\cos^2\left(k\cos^{-1}\left(\sqrt{\frac{y}{y+1}}\right)\right)}{(y+1)^{2\alpha+ 2}}\right|  \\
&= O\left( \frac{k}{(y+1)^{2\alpha + 1 + {3}/{2}}} + \frac{1}{(y+1)^{2\alpha + 2}}\right). 
 \end{align*}
 Therefore, we can bound Equation~\eqref{eq:integralsumboundcosine} by the following sum:
 \begin{align*}
 \sum_{i=k}^\infty \max_{y \in [i, i+1]} |h'(y) | &= O\left(\frac{k}{(k+1)^{2\alpha + {3}/{2}}} + \frac{1}{(k+1)^{2\alpha + 1}}\right) = O\left(\frac{1}{(k+1)^{2\alpha + {1}/{2}}}\right). \qquad \qed
 \end{align*}
\end{proof}

In the following Lemma, we estimate the order of the oscillatory integral approximation to the sum in Equation~\eqref{eq:nonergodicoptimalitysum}. The proof follows by a change of variables and an integration by parts.

\begin{lemma}\label{lem:integralestimation}
The following bound holds: 
\begin{align}\label{eq:lemmaintegralestimation:main}
\sum_{i=k}^\infty  \frac{\cos^2\left(k \cos^{-1}\left(\sqrt{\frac{i}{i+1}}\right)\right)}{(i+1)^{2\alpha + 1}}  dy &= \Omega\left(\frac{1}{(k+1)^{2\alpha}}\right).
\end{align}
\end{lemma}
\begin{proof} 
Fix $k \geq 1$. We first perform a change of variables $u = \cos^{-1}(\sqrt{{y}/{(y+1)}})$ on the integral approximation of the sum:
\begin{align}\label{eq:changeofvariables}
 \int_{k}^\infty  \frac{\cos^2\left(k \cos^{-1}\left(\sqrt{\frac{y}{y+1}}\right)\right)}{(y+1)^{2\alpha + 1}} dy &= 2\int_{0 }^{\cos^{-1}\left(\sqrt{{k}/{(k+1)}}\right)} \cos^2(k u) \cos(u) \sin^{4\alpha - 1}(u)du.
\end{align}
We will show that the right hand side of Equation~\eqref{eq:changeofvariables} is of order $\Omega\left({1}/{(k+1)^{2\alpha}}\right)$.  Then Equation~\eqref{eq:lemmaintegralestimation:main} will follow by Lemma~\ref{lem:integralapproximationofcosinesum}.

Let $\rho:=\cos^{-1}(\sqrt{k/(k+1)})$. We have
\begin{align*}
2\int_0^{\rho} \cos^2(ku) \cos(u) \sin^{4\alpha -1}(u)d u =\int_0^{\rho} (1+\cos(2ku))  \cos(u) \sin^{4\alpha -1}(u)d u =p_1+p_2+p_3
\end{align*}
where  
\begin{align*}
p_1 & = \int_0^{\rho} 1\cdot \cos(u) \sin^{4\alpha -1}(u)d u = \frac{1}{4\alpha}\sin^{4\alpha}(\rho);\\
p_2 & = \frac{1}{2k}\sin(2k\rho)\cos(\rho) \sin^{4\alpha -1}(\rho);\\
p_3 & = -\frac{1}{2k}\int_0^\rho  \sin(2ku) d(\cos(u) \sin^{4\alpha -1}(u));
\end{align*} 
and we have applied integration by parts for $\int_0^{\rho}\cos(2ku)  \cos(u) \sin^{4\alpha -1}(u)d u=p_2+p_3$. 

Because $\sin(\cos^{-1}(x)) = \sqrt{1-x^2}$, for all $\eta > 0$, we get $$\sin^{\eta}(\rho) =\sin^\eta\cos^{-1}\left(\sqrt{k/(k+1)}\right)= \frac{1}{(k+1)^{\eta/2}}.$$
In addition, we have $\cos(\rho) =\cos \cos^{-1}\left(\sqrt{k/(k+1)}\right)=\sqrt{k/(k+1)} $ and the trivial bounds $|\sin(2k\rho)|\le 1$ and $|\sin(2ku)|\le 1$.

Therefore, the following bounds hold:
\begin{align*}p_1 = \frac{1}{4\alpha(k+1)^{2\alpha}} && \text{and} && 
|p_2| \le\frac{\sqrt{k/(k+1)}}{2k(k+1)^{2\alpha -1/2}}  =O\left(\frac{1}{(k+1)^{2\alpha+{1/2}}}\right).
\end{align*}
In addition, for $p_3$, we have $ d(\cos(u) \sin^{4\alpha -1}(u))= \sin^{4\alpha-2}(u)((4\alpha-1)\cos(u)-\sin^2(u)) du$. Furthermore, for $u\in[0,\rho]$ and $\alpha>1/2$, we have $\sin^{4\alpha-2}(u)\in[0,{1}/{(k+1)^{2\alpha-1}}]$ and the following lower bound: $(4\alpha-1)\cos(u)-\sin^2(u)\ge(4\alpha-1)\cos(\rho)-\sin^2(\rho)=(4\alpha-1)\sqrt{k/(k+1)}-{1}/{(k+1)}>0$ as long as $k\ge 1$. Therefore, we have $\sin^{4\alpha-2}(u)((4\alpha-1)\cos(u)-\sin^2(u))\ge0$ for all $u\in[0,\rho]$ and, thus,  
\begin{align*}
|p_3|&\le\frac{1}{2k}\cos(\rho) \sin^{4\alpha -1}(\rho)=\frac{\sqrt{k/(k+1)}}{2k(k+1)^{2\alpha -1/2}}=O\left(\frac{1}{(k+1)^{2\alpha+{1/2}}}\right).    
\end{align*}
Therefore, $p_1+p_2+p_3\ge p_1 -|p_2|-|p_3|=\Omega\left((k+1)^{-2\alpha}\right).$
\qed\end{proof}

We deduce the following theorem from the sum estimation in Lemma~\ref{lem:integralestimation}:
\begin{theorem}[Lower complexity of DRS]\label{thm:DRSobjectiveslow}
There exists closed, proper, and convex functions $f, g : \cH \rightarrow (-\infty, \infty]$ such that $f$ is $1$-Lipschitz and for every $\alpha > 1/2$, there is a point $z^0 \in \cH$ and $\gamma \in \vR_{++}$ such that if $(z^j)_{j \geq 0}$ is generated by Algorithm~\ref{alg:DRS} with $\lambda_k = {1}/{2}$ for all $k \geq 0$, then 
\begin{align*}
f(x_g^k) + g(x_g^k) - f(x^\ast) - g(x^\ast) &= \Omega\left(\frac{1}{(k+1)^\alpha}\right).
\end{align*}
\end{theorem}
\begin{proof}
Assume the setting of Theorem~\ref{prop:lowerboundobjective}.  Then $f = d_V$ and $g = \iota_U$, and by Lemma~\ref{lem:integralestimation}, we have
\begin{align*}
f(x_g^k) + g(x_g^k) - f(x^\ast) - g(x^\ast) = d_V(x_g^k) &= \Omega\left(\frac{1}{(k+1)^\alpha}\right). \qquad \qed
\end{align*}
\end{proof}

Theorem~\ref{thm:DRSobjectiveslow} shows that the DRS algorithm \emph{is nearly as slow} as the subgradient method. We use the word nearly because the subgradient method has complexity $O(1/\sqrt{k+1})$, while DRS has complexity $o(1/\sqrt{k+1})$.  To the best of our knowledge, this is the first \emph{lower complexity} result for DRS algorithm.  Note that Theorem~\ref{thm:DRSobjectiveslow} implies the same lower complexity for the Forward Douglas Rachford splitting algorithm \cite{briceno2012forward}.

\subsection{Optimal objective and FPR rates with Lipschitz derivative}

The following examples show that the objective and FPR rates derived in Theorem~\ref{thm:PPAconvergence} are essentially optimal. The setup of the following counterexample already appeared in \cite[Remarque 6]{brezis1978produits} but the objective function lower bounds were not shown.

\begin{theorem}[Lower complexity of PPA]\label{thm:ppaslow}
There exists a Hilbert space $\cH$, and a closed, proper, and convex function $f$ such that for all $\alpha > {1}/{2}$, there exists $z^0 \in \cH$ such that if $(z^j)_{j \geq 0}$ is generated by PPA (Equation~\eqref{ppaitr}), then
\begin{align*}
\|\prox_{\gamma f}(z^k) - z^k\|^2 \geq \frac{\gamma^2}{(1+2\alpha)e^{2\gamma}(k+\gamma)^{1+2\alpha}}  
&& \mathrm{and} && f(z^{k+1}) - f(x^\ast)  \geq\frac{1}{4\alpha e^{2\gamma}(k+ 1 + \gamma)^{2\alpha}}.
\end{align*}
\end{theorem}
\begin{proof}
Let $\cH = \ell_2(\vR)$, and define a linear map $A : \cH \rightarrow \cH$:
\begin{align*}
A\left(z_1, z_2, \cdots, z_n, \cdots\right) &= \left(z_1, \frac{z_2}{2}, \cdots, \frac{z_n}{n}, \cdots\right).
\end{align*}
For all $z \in \cH$, define $f(x) = (1/2)\dotp{Az, z}$. Thus, we have the proximal identity for $f$ and 
\begin{align*}
\prox_{\gamma f}(z) &= (I+ \gamma A)^{-1}(z) = \left( \frac{j}{j+\gamma}z_j\right)_{j \geq 1} && \mathrm{and} && (I -\prox_{\gamma f})(z) = \left(\frac{\gamma}{j+\gamma}z_j\right)_{j \geq 1}.
\end{align*}

Now let $z^0 = ({1}/{(j+\gamma)^\alpha})_{j \geq 1} \in \cH$, and set $T = \prox_{\gamma f}$.  Then we get the following FPR lower bound:
\begin{align*}
\|z^{k+1} - z^k\|^2 = \|T^k(T - I)z^0\|^2 = \sum_{i=1}^\infty \left(\frac{i}{i+\gamma}\right)^{2k} \frac{\gamma^2}{(i+\gamma)^{2 + 2\alpha}} &\geq  \sum_{i=k}^\infty \left(\frac{i}{i+\gamma}\right)^{2k} \frac{\gamma^2}{(i+\gamma)^{2 + 2\alpha}} \\
&\geq \frac{\gamma^2}{(1+2\alpha)e^{2\gamma}(k+\gamma)^{1+2\alpha}}.
\end{align*}
Furthermore, the objective lower bound holds
\begin{align*}
f(z^{k+1})  - f(x^\ast) = \frac{1}{2}\dotp{Az^{k+1}, z^{k+1}} = \frac{1}{2}\sum_{i=1}^\infty \frac{1}{i}  \left(\frac{i}{i+\gamma}\right)^{2(k+1)} \frac{1}{(i+\gamma)^{2\alpha}} &\geq \frac{1}{2}\sum_{i=k+1}^\infty  \left(\frac{i}{i+\gamma}\right)^{2(k+1)} \frac{1}{(i+\gamma)^{1+ 2\alpha}} \\
&\geq \frac{1}{4\alpha e^{2\gamma}(k+ 1 + \gamma)^{2\alpha}}. \qquad \qed
\end{align*}
\end{proof}

\section{From relaxed PRS to relaxed ADMM}\label{sec:DRSADMM}

It is well known that ADMM is equivalent to DRS applied to the Lagrange dual of Problem~\eqref{eq:simplelinearconstrained} \cite{gabay1983chapter}. Thus, if we let $d_f(w) := f^\ast(A^\ast w)$ and $d_g(w) := g^\ast(B^\ast w) - \dotp{w, b}$, then relaxed ADMM is equivalent to relaxed PRS applied to the following problem:
\begin{align}\label{eq:dualproblem2}
\Min_{w \in \cG} & \; d_f(w) + d_g(w).
\end{align}

{We make two assumptions regarding $d_f$ and $d_g$:
\begin{assump}[Solution existence]\label{assump:additivesub}
Functions $f, g : \cH \rightarrow (-\infty, \infty]$ satisfy
\begin{align}
\zer(\partial d_f + \partial d_g) \neq \emptyset.
\end{align}
\end{assump}
This is a restatement of Assumption~\ref{assump:additivesub}, which we in our analysis of the primal case.

\begin{assump}\label{assump:precomposgradient}
The following differentiation rule holds:
\begin{align*}
\partial d_f(x) = A^\ast \circ (\partial f^\ast) \circ A && \mathrm{and} && \partial d_g(x) = B^\ast \circ (\partial g^\ast) \circ B - b.
\end{align*}
\end{assump}
{See \cite[Theorem 16.37]{bauschke2011convex} for conditions that imply this identity, of which the weakest are  $0 \in \mathrm{sri}(\range(A^\ast) - \dom(f^\ast))$ and $0 \in \mathrm{sri}(\range(B^\ast) - \dom(g^\ast))$, where $\mathrm{sri}$ is the strong relative interior of a convex set.} {\color{blue}This assumption may seem strong, but it is standard in the analysis of ADMM because it implies the dual proximal operator identities in~\eqref{eq:proxd2}.}}

Given an initial vector $z^0 \in \cG$, Lemma~\ref{prop:DRSmainidentity} shows that at each iteration relaxed PRS performs the following computations:
\begin{align}\label{eq:DRSADMM}
\begin{cases}
w_{d_g}^k &= \prox_{\gamma d_g}(z^k); \\
w_{d_f}^k &= \prox_{\gamma d_f}(2w_{d_g}^k - z^k);  \\
z^{k+1} &= z^k + 2\lambda_k(w_{d_f}^k - w_{d_g}^k).
\end{cases}
\end{align}
In order to apply the relaxed PRS algorithm, we need to compute the proximal operators of the dual functions $d_f$ and $d_g$.

\begin{lemma}[Proximity operators on the dual]\label{lem:proxdual}
Let $w, v\in \cG$.  Then the update formulas $w^+ = \prox_{\gamma d_f}(w)$ and $v^+ = \prox_{\gamma d_g}(v)$ are equivalent to the following computations
\begin{align*}
\begin{cases}  x^+ = \argmin_{x \in \cH_1} f(x) - \dotp{w, Ax} + \frac{\gamma}{2} \|Ax\|^2; \\
 w^+ = w - \gamma Ax^+.  
\end{cases} && \mathrm{and} &&  \begin{cases} y^+ = \argmin_{y \in \cH_2} g(y) - \dotp{v, By - b} + \frac{\gamma}{2} \|By-b\|^2; \\
 v^+ = v - \gamma(By^+ - b).
\end{cases} \numberthis \label{eq:proxd2}
\end{align*}
respectively.  In addition, the subgradient inclusions hold: $A^\ast w^+ \in \partial f(x^+)$ and $ B^\ast v^+ \in \partial g(y^+)$. Finally, $w^+$ and $v^+$ are independent of the choice of $x^+$ and $y^+$, respectively, even if they are not unique solutions to the minimization subproblems.
\end{lemma}


We can use Lemma~\ref{lem:proxdual} to derive the relaxed form of ADMM in Algorithm~\ref{alg:ADMM}.  Note that this form of ADMM eliminates the ``hidden variable" sequence $(z^j)_{j \geq 0}$ in Equation~(\ref{eq:DRSADMM}). This following derivation is not new, but is included for the sake of completeness. See \cite{gabay1983chapter} for the original derivation.
\begin{proposition}[Relaxed ADMM]\label{prop:relaxedADMM}
Let $z^0 \in \cG$, and let $(z^j)_{j \geq 0}$ be generated by the relaxed PRS algorithm applied to the dual formulation in Equation~(\ref{eq:dualproblem2}).  Choose initial points $w_{d_g}^{-1} = z^0, x^{-1} = 0$ and $y^{-1} = 0$ and initial relaxation $\lambda_{-1} = {1}/{2}$. Then we have the following identities starting from $k = -1$:
\begin{align*}
y^{k+1} &= \argmin_{y \in \cH_2} g(y) - \dotp{w_{d_g}^{k},Ax^{k} +  By - b} + \frac{\gamma}{2} \|Ax^k + By - b + (2\lambda_k - 1)(Ax^{k} + By^{k} -b) \|^2 \\
w_{d_g}^{k+1} &= w_{d_g}^{k} - \gamma (Ax^{k} +  By^{k+1} - b) - \gamma(2\lambda_k - 1)(Ax^{k} + By^{k} - b) \\
x^{k+1} &= \argmin_{x \in \cH_1} f(x) - \dotp{w_{d_g}^{k+1}, Ax + By^{k+1} - b} + \frac{\gamma}{2} \|Ax + By^{k+1} - b\|^2 \\
w_{d_f}^{k+1} &= w_{d_g}^{k+1} - \gamma (Ax^{k+1} + By^{k+1} - b)
\end{align*}
\end{proposition}
\begin{proof}
See Appendix~\ref{app:DRSADMM}.\qed
\end{proof}

\begin{remark}
Proposition~\ref{prop:relaxedADMM} proves that $w_{d_f}^{k+1} = w_{d_g}^{k+1} - \gamma (Ax^{k+1} + By^{k+1} - b)$. Recall that by Equation~(\ref{eq:DRSADMM}), $z^{k+1} - z^k = 2\lambda_k(w_{d_f}^{k} - w_{d_g}^{k})$.  Therefore, it follows that
\begin{align}\label{eq:ADMMfeasibilityFPR}
z^{k+1} - z^k &= -2\gamma \lambda_k( Ax^{k} + By^k - b).
\end{align}

\end{remark}

%
%

\subsection{Dual feasibility convergence rates}\label{sec:ADMMdualrate}

We can apply the results of Section~\ref{sec:DRSconvergence} to deduce convergence rates for the dual objective functions.  Instead of restating those theorems, we just list the following bounds on the feasibility of the primal iterates.

\begin{theorem}\label{thm:ADMMdualconvergence}
Suppose that $(z^j)_{j \geq 0}$ is generated by Algorithm~\ref{alg:ADMM}, and let $(\lambda_j)_{j \geq 0} \subseteq (0, 1]$. Then the following convergence rates hold:
\begin{enumerate}
\item \label{thm:ADMMdualconvergence:part:ergodic} {\bf Ergodic convergence:} 
The feasibility convergence rate holds:
\begin{align*}
\|A\overline{x}^k + B\overline{y}^k - b\|^2 &= \frac{4\|z^0 - z^\ast\|^2}{\gamma\Lambda_k^2}. \numberthis \label{thm:ADMMdualconvergence:eq:ergodicfeasibility}
\end{align*}
\item  \label{thm:ADMMdualconvergence:part:nonergodic} {\bf Nonergodic convergence:} Suppose that $\underline{\tau} = \inf_{j \geq 0} \lambda_j(1-\lambda_j) > 0$. Then
\begin{align}\label{thm:ADMMdualconvergence:eq:nonergodicfeasibility}
\|Ax^k + By^k - b\|^2 \leq \frac{\|z^0 - z^\ast\|^2}{4\gamma^2 \underline{\tau}(k+1)} && \mathrm{and} &&  \|Ax^k + By^k - b\|^2 = o\left(\frac{1}{k+1}\right).
\end{align}
\end{enumerate}
\end{theorem}
\begin{proof}
Parts~\ref{thm:ADMMdualconvergence:part:ergodic} and~\ref{thm:ADMMdualconvergence:part:nonergodic} are straightforward applications of Corollary~\ref{cor:DRSaveragedconvergence}. and \cut{Theorem~\ref{thm:drsergodic}, Theorem~\ref{thm:drsnonergodic} and}the FPR identity:
\begin{align*}
z^{k} - z^{k+1} \stackrel{\eqref{eq:ADMMfeasibilityFPR}}{=} 2\gamma \lambda_k(Ax^k + By^k - b).\qquad \qed
\end{align*}
\end{proof}

\subsection{Converting dual inequalities to primal inequalities}\label{sec:convertinequalities}

The ADMM algorithm generates $5$ sequences of iterates:
\begin{align*}
(z^j)_{j \geq 0}, (w_{d_f}^j)_{j \geq 0}, \text{ and } (w_{d_g}^j)_{j \geq 0} \subseteq \cG && \text{and} &&  (x^j)_{j \geq 0} \in \cH_1, (y^j)_{j \geq 0} \in \cH_2.
\end{align*}
The dual variables do not necessarily have a meaningful interpretation, so it is desirable to derive convergence rates involving the primal variables. In this section we will apply the Fenchel-Young inequality \cite[Proposition 16.9]{bauschke2011convex} to convert the dual objective into a primal expression.

The following two propositions prove two fundamental inequalities that bound the primal objective.

\begin{proposition}[ADMM primal upper fundamental inequality]\label{prop:ADMMupper}
Let $z^\ast$ be a fixed point of $\TPRS$ and let $w^\ast = \prox_{\gamma d_g}(z^\ast)$. Then for all $k \geq 0$, we have the bound:
\begin{align*}
4\gamma \lambda_k (f(x^k) + g(y^k) &- f(x^\ast) - g(y^\ast))\\
&\leq \|z^k - (z^\ast - w^\ast)\|^2 - \|z^{k+1} - (z^\ast - w^\ast)\|^2  + \left(1- \frac{1}{\lambda_k} \right) \|z^{k} - z^{k+1}\|^2. \numberthis \label{prop:ADMMupper:eq:main}
\end{align*}
\end{proposition}
\begin{proof}
See Appendix~\ref{app:DRSADMM}.\qed
\end{proof}

\begin{remark}
Note that Equation~\eqref{prop:ADMMupper:eq:main} is nearly identical to the upper inequality in Proposition~\ref{prop:DRSupper}, except that $z^\ast - w^\ast$ appears in the former where $x^\ast$ appears in the latter.
\end{remark}

\begin{proposition}[ADMM primal lower fundamental inequality]\label{prop:ADMMlower}
Let $z^\ast$ be a fixed point of $\TPRS$ and let $w^\ast = \prox_{\gamma d_g}(z^\ast)$.  Then for all $x \in \cH_1$ and $y \in \cH_2$ we have the bound:
\begin{align}\label{prop:ADMMlower:eq:main}
f(x) + g(y) - f(x^\ast) - g(y^\ast) &\geq \dotp{Ax + By - b, w^\ast}.
\end{align}

\end{proposition}
\begin{proof}
The lower bound follows from the subgradient inequalities:
\begin{align*}
f({x}) - f(x^\ast)  \geq \dotp{{x} - x^\ast, A^\ast w^\ast} && \text{and} && g({y}) - g(y^\ast) \geq \dotp{{y} - y^\ast, B^\ast w^\ast} .
\end{align*}
We add these inequalities together and use the identity $Ax^\ast + By^\ast = b$ to get Equation~\eqref{prop:ADMMlower:eq:main}.
\qed\end{proof}

\begin{remark}
We use Inequality~\eqref{prop:ADMMlower:eq:main} in two special cases: 
\begin{align}\label{eq:admmlower}
f(x^k) + g(y^k) - f(x^\ast) - g(y^\ast) &\geq \frac{1}{\gamma}\dotp{w_{d_g}^k - w_{d_f}^k, w^\ast} \\
f(\overline{x}^k) + g(\overline{y}^k) - f(x^\ast) - g(y^\ast) &\geq \frac{1}{\gamma}\dotp{\overline{w}_{d_g}^k - \overline{w}_{d_f}^k, w^\ast}. \numberthis \label{eq:admmlowerergodic}
\end{align}
These bounds are nearly identical to the fundamental lower inequality in Proposition~\ref{prop:DRSlower}, except that $w^\ast$ appears in the former where $z^\ast - x^\ast$ appeared in the latter.
\end{remark}

\subsection{Converting dual convergence rates to primal convergence rates}\label{sec:dualtoprimalrates}

In this section, we use the inequalities deduced in Section~\ref{sec:convertinequalities} to derive convergence rates for the primal objective values. The structure of the of the proofs of the following theorems are exactly the same as in the primal convergence case in Section~\ref{sec:DRSconvergence}, except that we use the upper and lower inequalities derived in the Section~\ref{sec:convertinequalities} instead of the fundamental upper and lower inequalities in Propositions~\ref{prop:DRSupper} and~\ref{prop:DRSlower}.  This amounts to replacing the term $z^\ast - x^\ast$ and $x^\ast$ by $w^\ast$ and $z^\ast - w^\ast$, respectively, in all of the inequalities from Section~\ref{sec:DRSconvergence}. Thus, we omit the proofs.

\begin{theorem}[Ergodic primal convergence of ADMM]\label{thm:ADMMergodicconvergenceprimal}
Define the ergodic primal iterates by the formulas: $\overline{x}^k = ({1}/{\Lambda_k})\sum_{i=0}^k \lambda_ix^i$ and $\overline{y}^k = ({1}/{\Lambda_k}) \sum_{i=0}^k \lambda_i y^i.$ Then
\begin{align*}
-\frac{2\|w^\ast\|\|z^0 - z^\ast\|}{\gamma\Lambda_k} \leq f(\overline{x}^k) + g(\overline{y}^k) - f(x^\ast) - g(y^\ast) \leq \frac{\|z^0 - (z^\ast - w^\ast)\|^2}{4\gamma \Lambda_k}. \numberthis \label{thm:ADMMergodicconvergenceprimal:eq:main}
\end{align*}
\end{theorem}

{The ergodic rate presented here is stronger and easier to interpret than the one in \cite{he20121} for the ADMM algorithm ($\lambda_k \equiv 1/2$). Indeed, the rate presented in \cite[Theorem 4.1]{he20121} shows the following bound: for all $k \geq 1$ and for any bounded set $\cD\subseteq \dom(f) \times \dom(g) \times \cG$, we have the following variational inequality bound
\begin{align*}
\sup_{(x, y, w) \in \cD} &\left(f(\overline{x}^{k-1})+ g(\overline{y}^k) - f(x) - g(y) + \dotp{\overline{w}_{d_g}^k, Ax + By - b}- \dotp{A\overline{x}^{k-1} + B\overline{y}^k - b, w}\right)\\
&\leq \frac{\sup_{(x, y, w) \in \cD} \|(x, y, w) - (x^0, y^0, w_{d_g}^0)\|^2}{2(k+1)}.
\end{align*}
If  $(x^\ast, y^\ast, w^\ast) \in \cD$, then the supremum is positive and bounds the deviation of the primal objective from the lower fundamental inequality.}

\begin{theorem}[Nonergodic primal convergence of ADMM]\label{thm:ADMMprimalnonergodic}
 For all $k \geq 0$, let $\tau_k = \lambda_k(1-\lambda_k)$. In addition, suppose that $\underline{\tau} = \inf_{j \geq 0} \tau_j > 0$. Then
 \begin{enumerate}
 \item \label{thm:ADMMprimalnonergodic:part:1} In general, we have the bounds:
\begin{align*}
 \frac{-\|z^0 - z^\ast\|\|w^\ast\|}{2\sqrt{\underline{\tau}(k+1)}} \leq f(x^k) + g(y^k) - f(x^\ast) - g(y^\ast)\leq \frac{\|z^{0} - z^\ast\|(\|z^{0} - z^\ast\| + \|w^\ast\|)}{2\gamma\sqrt{\underline{\tau}(k+1)}}\numberthis \label{thm:ADMMprimalnonergodic:eqmain}
\end{align*}
and $|f(x^k) + g(y^k) - f(x^\ast) - g(y^\ast)| = o(1/\sqrt{k+1})$.
\item \label{thm:ADMMprimalnonergodic:part:2} If $\cG = \vR$  and $\lambda_k \equiv {1}/{2}$, then for all $k \geq 0$,
\begin{align*}
 \frac{-\|z^0 - z^\ast\|\|w^\ast\|}{\sqrt{2}(k+1)} \leq f(x^{k+1}) + g(y^{k+1}) - f(x^\ast) - g(x^\ast)\leq \frac{\|z^{0} - z^\ast\|(\|z^{0} - z^\ast\| + \|w^\ast\|)}{\sqrt{2}\gamma(k+1)}
 \end{align*}
 and $|f(x^{k+1}) + g(x^{k+1}) - f(x^\ast) - g(x^\ast)| = o(1/(k+1))$.
\end{enumerate}
\end{theorem}

{The rates presented in Theorem~\ref{thm:ADMMprimalnonergodic} are new and, to the best of our knowledge, they are the first nonergodic convergence rate results for ADMM primal objective error.}

\section{Examples}

In this section, we apply relaxed PRS and relaxed ADMM to concrete problems and explicitly bound the associated objectives and FPR terms with the convergence rates we derived in the previous sections. 

\subsection{Feasibility problems}\label{sec:feasibility}

Suppose that $C_f$ and $C_g$ are closed convex subsets of $\cH$, with nonempty intersection.  The goal of the feasibility problem is the find a point in the intersection of $C_f$ and $C_g$. In this section, we present one way to model this problem using convex optimization and apply the relaxed PRS algorithm to reach the minimum. 

In general, we cannot expect linear convergence of relaxed PRS algorithm for the feasibility problem.  We showed this in Theorem~\ref{thm:arbitrarilyslow} by constructing an example for which the DRS iteration converges in norm but does so \emph{arbitrarily slow}.  A similar result holds for the alternating projection (AP) algorithm~\cite{bauschke2009characterizing}.  Thus, in this section we focus on the convergence rate of the \emph{FPR}.

Let $\iota_{C_f}$ and $\iota_{C_g}$ be the indicator functions of $C_f$ and $C_g$. Then $x \in C_f \cap C_g$, if, and only if, $\iota_{C_f}(x) + \iota_{C_g}(x) = 0$, and the sum is infinite otherwise.  Thus, a point is in the intersection of $C_f$ and $C_g$ if, and only if, it is the minimizer of the following problem:
\begin{align}\label{sec:feasibility:eq:chiminimize}
\Min_{x \in \cH} \iota_{C_f}(x) + \iota_{C_g}(x). 
\end{align}
The relaxed PRS algorithm applied to this problem, with $f = \iota_{C_f}$ and $g = \iota_{C_g}$, has the following form: Given $z^0 \in \cH$, for all $k \geq 0$, let
\begin{align}\label{sec:feasibility:eq:DRSchi}
\begin{cases}
x_g^k = P_{C_g}(z^k); \\
x_f^k = P_{C_f}(2x_g^k - z^k); \\
z^{k+1} = z^k + 2\lambda_k(x_f^k - x_g^k).
\end{cases}
\end{align}

Because $f = \iota_{C_f}$ and $g = \iota_{C_g}$ only take on the values $0$ and $\infty$, the objective value convergence rates derived earlier do not provide meaningful information, other than $x_f^k \in C_f$ and $x_g^k \in C_g$.  However, from the FPR identity $x_{f}^k - x_g^k = 1/(2\lambda_k)(z^{k+1} - z^k),$ we find that after $k$ iterations, Corollary~\ref{cor:DRSaveragedconvergence} produces the bound
\begin{align}\label{eq:feasibilitybounddistancenonergodic}
\max\{ d^2_{C_g}(x_f^k), d^2_{C_f}(x_g^k)\} \leq \|x_f^k - x_g^k\|^2 &= o\left(\frac{1}{k+1}\right)
\end{align}
whenever $(\lambda_j)_{j \geq 0}$ is bounded away from $0$ and $1$. Theorem~\ref{thm:optimalFPR} showed that this rate is optimal. Furthermore, if we average the iterates over all $k$, Theorem~\ref{thm:drsergodic} gives the improved bound
\begin{align}\label{eq:drsergodicdistancebound}
\max\{ d^2_{C_g}(\overline{x}_f^k), d^2_{C_f}(\overline{x}_g^k)\} \leq \|\overline{x}_f^k - \overline{x}_g^k\|^2 &= O\left(\frac{1}{\Lambda_k^2}\right),
\end{align}
which is optimal by Proposition~\ref{prop:ergodicfeasibilityopt}. Note that the averaged iterates satisfy $\overline{x}_f^k = ({1}/{\Lambda_k})\sum_{i=0}^k \lambda_i x_f^i \in C_f$ and $\overline{x}_g^k = ({1}/{\Lambda_k})\sum_{i=0}^k \lambda_i x_g^i \in C_g$, because $C_f$ and $C_g$ are convex.  Thus, we can state the following proposition:
{\begin{prop}
After $k$ iterations the relaxed PRS algorithm produces a point in each set with distance of order $O({1}/{\Lambda_k})$ from each other.
\end{prop}}

\subsection{Parallelized model fitting and classification}\label{sec:modelfiting}
The following general scenario appears in \cite[Chapter 8]{boyd2011distributed}. Consider the following general convex model fitting problem: Let $M : \vR^n \rightarrow \vR^m$ be a \emph{feature matrix}, let $b \in \vR^m$ be the  \emph{output} vector, let $l : \vR^m \rightarrow (-\infty, \infty]$ be a \emph{loss function} and let $r : \vR^n \rightarrow (-\infty, \infty]$ be a \emph{regularization function}. The \emph{model fitting problem} is formulated as the following minimization:
\begin{align}\label{sec:modelfiting:eq:problem}
\Min_{x\in \vR^n } \; l(Mx -b) + r(x).
\end{align}
The function $l$ is used to enforce the constraint $Mx = b + \nu$ up to some noise $\nu$ in the measurement, while $r$ enforces the \emph{regularity} of $x$ by incorporating \emph{prior knowledge} of the form of the solution. The function $r$ can also be used to enforce the uniqueness of the solution of $Mx = b$ in ill-posed problems. 

{We can solve Equation~\eqref{sec:modelfiting:eq:problem} by a direct application of relaxed PRS and  obtain $O(1/\Lambda_k)$ ergodic convergence and $o\left({1}/{\sqrt{k+1}}\right)$ nonergodic convergence rates.  Note that these rates do not require differentiability of $f$ or $g$. In contrast, the FBS algorithm requires differentiability of one of the objective functions and a knowledge of the Lipschitz constant of its gradient. The advantage of FBS is the $o(1/(k+1))$ convergence rate shown in Theorem~\ref{thm:PPAconvergence}.  However, we do not necessarily assume that $l$ is differentiable, so we may need to compute $\prox_{\gamma l(M(\cdot) - b)}$, which can be significantly more difficult than computing $\prox_{\gamma l}$\cut{, depending on the structure of $M$}. Thus, in this section we separate $M$ from $l$ by  rephrasing Equation~\eqref{sec:modelfiting:eq:problem} in the form of Problem~\eqref{eq:simplelinearconstrained}.}  

In this section, we present several different ways to split Equation~\eqref{sec:modelfiting:eq:problem}. Each splitting gives rise to a different algorithm and can be applied to general convex $l$ and $r$.  Our results predict convergence rates that hold for primal objectives, dual objectives, and the primal feasibility.  Note that in parallelized model fitting, it is not always desirable to take the time average of all of the iterates.  Indeed, when $r$ enforces sparsity, averaging the current $r$-iterate with old iterates, all of which are sparse, can produce a non-sparse iterate. This will slow down vector additions and prolong convergence.

\subsubsection{Auxiliary variable}
We can split Equation~\eqref{sec:modelfiting:eq:problem} by defining an auxiliary variable for $My$:
\begin{align*}
\Min_{x \in \vR^m, y \in\vR^n} & \; l\left(x \right) + r(y) \\
\text{subject to }  & \; My - x = b.\numberthis \label{sec:modelfiting:eq:pullAout}
\end{align*}
{The constraint in Equation~\eqref{sec:modelfiting:eq:pullAout} reduces to $Ax + By = b$ where  $B = M$ and  $A = -I_{\vR^m}$. If we set $f = l$ and $g = r$ and apply ADMM, the analysis of Section~\ref{sec:dualtoprimalrates} shows that
\begin{align*}
|l(x^k) + r(y^k) - l(My^\ast - b) - r(y^\ast)| = o\left(\frac{1}{\sqrt{k+1}}\right) && \mathrm{and} && \|My^k - b- x^k\|^2 = o\left(\frac{1}{{k+1}}\right).
\end{align*}
In particular, if $l$ is Lipschitz, then $|l(x^k) - l(My^k - b)| = o\left(1/\sqrt{k+1}\right)$. Thus, we have
\begin{align*}
|l(My^k - b) + r(y^k) - l(My^\ast - b) - r(y^\ast)| = o\left(\frac{1}{\sqrt{k+1}}\right).
\end{align*} 
A similar analysis shows that 
\begin{align*}
|l(M\overline{y}^k - b) + r(\overline{y}^k) - l(My^\ast - b) - r(y^\ast)| = O\left(\frac{1}{\Lambda_k}\right) && \mathrm{and} && \|M\overline{y}^k - b- \overline{x}^k\|^2 = O\left(\frac{1}{\Lambda_k^2}\right).
\end{align*}}

In the following two splittings, we leave the derivation of convergence rates to the reader.

\subsubsection{Splitting across examples}
We assume that $l$ is block separable: we have $l(Mx - b) = \sum_{i=1}^R l_i(M_i x - b_i)$ where
\begin{align*}
M = \begin{bmatrix} M_1 \\ \vdots \\ M_R\end{bmatrix} && \mathrm{and} && b = \begin{bmatrix} b_1 \\ \vdots \\b_R\end{bmatrix}.
\end{align*}
Each $M_i \in \vR^{m_i \times n}$ is a submatrix of $M$,  each $b_i \in \vR^{m_i}$ is a subvector of $b$, and $\sum_{i=1}^R m_i = m$.  Therefore, an equivalent form of Equation~\eqref{sec:modelfiting:eq:problem} is given by
\begin{align*}
\Min_{x_1, \cdots, x_R, y\in \vR^n} &\; \sum_{i =1}^R l_i(M_ix_i - b_i) + r(y) \\
\text{subject to } & \; x_r - y = 0, \quad r = 1, \cdots, R. \numberthis \label{sec:modelfiting:eq:acrossexamples}
\end{align*}
We say that Equation~\eqref{sec:modelfiting:eq:acrossexamples} is \emph{split across examples}.  Thus, to apply ADMM to this problem, we simply stack the vectors $x_i$, $i=1, \cdots, R$ into a vector $x = (x_1, \cdots, x_R)^T \in\vR^{nR}$. Then the constraints in Equation~\eqref{sec:modelfiting:eq:acrossexamples} reduce to $Ax + By = 0$ where $A = I_{\vR^{nR}}$ and $B y = (-y, \cdots, -y)^T$.

\subsubsection{Splitting across features}

We can also split Equation~\eqref{sec:modelfiting:eq:problem} \emph{across features}, whenever $r$ is block separable in $x$, in the sense that there exists $C > 0$, such that $r = \sum_{i=1}^C r_i(x_i)$, and $x_i \in \vR^{n_i}$ where $\sum_{i = 1}^C n_i = n$. This splitting corresponds to partitioning the columns of $M$, i.e. $M = \begin{bmatrix} M_1, \cdots,  M_C\end{bmatrix},$
and $M_i \in \vR^{m\times n_i}$, for all $i = 1, \cdots, C$. Note that for all $y \in \vR^n$, $My = \sum_{i=1}^C M_iy_i$.  With this notation, we can derive an equivalent form of Equation~\eqref{sec:modelfiting:eq:problem} given by 
\begin{align*}
\Min_{x, y \in\vR^n} & \; l\left(\sum_{i=1}^C x_i - b\right) + \sum_{i=1}^C r_i(y_i) \\
\text{subject to }  & \; x_i - M_iy_i  = 0, \quad i = 1, \cdots, C.\numberthis \label{sec:modelfiting:eq:acrossfeatures}
\end{align*}
The constraint in Equation~\eqref{sec:modelfiting:eq:acrossfeatures} reduces to $Ax + By = 0$ where $A = I_{\vR^{mC}}$ and $By = -(M_1y_1, \cdots, M_Cy_C)^T \in \vR^{nC}$.

\subsection{Distributed ADMM}

In this section our goal is to use Algorithm~\ref{alg:ADMM} to
\begin{align}\label{eq:predistributedadmm}
\Min_{x \in \cH} \sum_{i=1}^m f_i(x)
\end{align}
by using the splitting in \cite{schizas2008consensus}. {Note that we could minimize this function by reformulating it in the product space $\cH^m$ as follows:
\begin{align}
\Min_{\vx \in \cH^m} \sum_{i=1}^m f_i(x_i) + \iota_{D}(\vx),
\end{align}
where $D = \{(x, \cdots, x) \in \cH^m \mid x\in \cH\}$ is the diagonal set. Applying relaxed PRS to this problem results in a parallel algorithm where each function performs a local minimization step and then communicates its local variable to a \emph{central processor}.}  In this section, we assign each function a local variable but we never communicate it to a central processor.  Instead, each function only communicates with \emph{neighbors}.  

Formally, we assume there is a simple, connected and undirected graph $G = (V, E)$ on $|V| = m$ vertices with edges, $E$, that describe a neighbor relation among the different functions.  We introduce a new variable $x_i \in \cH$ for each function $f_i$, and, hence, we set $\cH_1 = \cH^m$, (see Section~\ref{sec:DRSADMM}). We can encode the constraint that each node communicates with neighbors by introducing an auxiliary variable for each edge in the graph:
\begin{align*}
\Min_{\vx \in \cH^m, \vy\in \cH^{|E|}} &\; \sum_{i=1}^m f_i(x_i) \\
\text{subject to } &\; x_i = y_{ij}, x_j = y_{ij}, \text{ for all } (i,j) \in E. \numberthis \label{eq:predistributedadmm2}
\end{align*}
The linear constraints in Equation~\eqref{eq:predistributedadmm2} can be written in the form of $A\vx + B\vy = 0$ for proper matrices $A$ and $B$. Thus, we reformulate Equation~\eqref{eq:predistributedadmm2} as
\begin{align*}
\Min_{\vx \in \cH^m, \vy\in \cH^{|E|}} &\; \sum_{i=1}^m f_i(x_i) + g(\vy)\\
\text{subject to } &\; A\vx + B \vy = 0, \numberthis \label{eq:distributedadmm}
\end{align*}
where $g : \cH^{|E|} \rightarrow \vR$ is the zero map.  

Because we only care about finding the value of the variable $\vx \in \cH^m$, the following simplification can be made to the sequences generated by ADMM applied to Equation~\eqref{eq:distributedadmm} with $\lambda_k = {1}/{2}$ for all $k \geq 1$ \cite{shi2013linear}: Let $\cN_i$ denote the set of neighbors of $i \in V$ and set $x_i^0 = \alpha_i^0 = 0$ for all $i \in V$. Then for all $k \geq 0$, 
\begin{align}\label{alg:distributedadmm}
\begin{cases}
x_{i}^{k+1} = \argmin_{x_i \in \cH} f_i(x) + \frac{\gamma|\cN_i|}{2}\| x_{i} - x_i^k - \frac{1}{|\cN_i|}\sum_{j \in \cN_i} x_j^k + \frac{1}{\gamma|\cN_i|}\alpha_i\|^2 + \frac{\gamma|\cN_i|}{2}\|x_i\|^2 \\
\alpha_i^{k+1} = \alpha_i^k + \gamma\left(|\cN_i|x_i^{k+1} - \sum_{j \in \cN_i}x_j^{k+1}\right).
\end{cases}
\end{align}
Equation~\eqref{alg:distributedadmm} is truly distributed because each node $i \in V$ only requires information from its local neighbors at each iteration.

In \cite{shi2013linear}, linear convergence is shown for this algorithm provided that $f_i$ are strongly convex and $\nabla f_i$ are Lipschitz.  For general convex functions, we can deduce the nonergodic rates from Theorem~\ref{thm:ADMMprimalnonergodic}
\begin{align*}
\left|\sum_{i=1}^m f_i(x_i^k) - f(x^\ast)\right| = o\left(\frac{1}{\sqrt{k+1}}\right) && \text{and} && \sum_{\substack{i \in V \\ j \in N_i}} \|x_i^k - z_{ij}^k\|^2 + \sum_{\substack{i \in V \\ i \in N_j}} \|x_j^k - z^k_{ij}\|^2 = o\left(\frac{1}{k+1}\right),
\end{align*}
and the ergodic rates from Theorem~\ref{thm:ADMMergodicconvergenceprimal}
\begin{align*}
\left|\sum_{i=1}^m f_i(\overline{x}_i^k) - f(x^\ast)\right| = O\left(\frac{1}{k+1}\right) && \text{and} && \sum_{\substack{i \in V \\ j \in N_i}} \|\overline{x}_i^k - \overline{z}_{ij}^k\|^2 + \sum_{\substack{i \in V \\ i \in N_j}} \|\overline{x}_j^k - \overline{z}_{ij}^k\|^2 = O\left(\frac{1}{(k+1)^2}\right).
\end{align*}
These convergence rates are new and complement the linear convergence results in \cite{shi2013linear}. In addition, they complement the similar ergodic rate derived in \cite{wei2012distributed} for a different distributed splitting.  

\section{Conclusion}
{In this paper, we provided a comprehensive convergence rate analysis of the FPR and objective error of several splitting algorithms under general convexity assumptions.  We showed that the convergence rates are essentially optimal in all cases. All results follow from some combination of a lemma that deduces convergence rates of summable monotonic sequences (Lemma~\ref{lem:sumsequence}), a simple diagram (Figure~\ref{fig:DRSTR}), and fundamental inequalities (Propositions~\ref{prop:DRSupper} and \ref{prop:DRSlower}) that relate the FPR to the objective error of the relaxed PRS algorithm. The most important open question is whether and how the rates we derived will improve when we enforce stronger assumptions, such as Lipschitz differentiability and/or strong convexity, on $f$ and $g$.  This will be the subject of future work.}

\begin{acknowledgements}D. Davis' work is partially supported by NSF GRFP grant DGE-0707424. W. Yin's work is partially supported by NSF grants DMS-0748839 and DMS-1317602.
\end{acknowledgements}

\bibliographystyle{spmpsci}
\bibliography{bibliography}

\appendix
\section{Proofs from Section~\ref{sec:DRSADMM}}\label{app:DRSADMM}

\begin{proof}[Proposition~\ref{prop:relaxedADMM}]
By Equation~(\ref{eq:DRSADMM}) and Lemma~\ref{lem:proxdual}, we get the following formulation for the $k$-th iteration: Given $z^0 \in \cH$
\begin{align*}\begin{cases}\numberthis \label{prop:relaxedADMM:eq:big}
y^k &= \argmin_{y \in \cH_2} g(y) - \dotp{z^k,By - b} + \frac{\gamma}{2} \|By - b\|^2 \\
w_{d_g}^{k} &= z^k - \gamma (By^{k} - b) \\
x^{k} &= \argmin_{x \in \cH_1} f(x) - \dotp{2w_{d_g}^{k} - z^k, Ax} + \frac{\gamma}{2} \|Ax\|^2 \\
w_{d_f}^{k} &= 2w_{d_g}^{k} - z^k  -  \gamma Ax^{k}\\
z^{k+1} &= z^k + 2\lambda_k(w_{d_f}^k - w_{d_g}^k)
\end{cases}
\end{align*}
We will use this form to get to the claimed iteration. First,
\begin{align*}
2w_{d_g}^k - z^k = w_{d_g}^k - \gamma(By^k - b) && \mathrm{and} && w_{d_f}^k = w_{d_g}^k - \gamma (Ax^k + By^k - b).\numberthis \label{prop:relaxedADMM:eq:2}
\end{align*}
Furthermore, we can simplify the definition of $x^k$:
\begin{align*}
x^k &= \argmin_{x \in \cH_1} f(x) - \dotp{2w_{d_g}^{k} - z^k, Ax} + \frac{\gamma}{2} \|Ax\|^2 \\
&\stackrel{\eqref{prop:relaxedADMM:eq:2}}{=} \argmin_{x \in \cH_1} f(x) - \dotp{w_{d_g}^k - \gamma(By^k - b), Ax} + \frac{\gamma}{2} \|Ax\|^2 \\
&= \argmin_{x \in \cH_1} f(x) - \dotp{w_{d_g}^k, Ax + By^k - b} + \frac{\gamma}{2} \|Ax + By^k - b\|^2. \numberthis \label{prop:relaxedADMM:eq:3}
\end{align*}
Note that the last two lines of Equation~(\ref{prop:relaxedADMM:eq:3}) differ by terms independent of $x$.

We now eliminate the $z^k$ variable from the $y^k$ subproblem: because $w_{d_f}^k + z^k = 2w_{d_g}^k - \gamma Ax^k$, we have
\begin{align*}
z^{k+1} &= z^k + 2\lambda_k(w_{d_f}^k - w_{d_g}^k)\\
&\stackrel{(\ref{prop:relaxedADMM:eq:2})}{=} z^k + w_{d_f}^k - w_{d_g}^k + \gamma(2\lambda_k - 1)(Ax^k + By^k - b) \\
&= w_{d_g}^k - \gamma Ax^k - \gamma(2\lambda_k - 1)(Ax^k+By^k - b).  \numberthis \label{prop:relaxedADMM:eq:4}
\end{align*}

We can simplify the definition of $y^{k+1}$ by with the identity in Equation~(\ref{prop:relaxedADMM:eq:4}):
\begin{align*}
y^{k+1} &= \argmin_{y \in \cH_2} g(y) - \dotp{z^{k+1},By - b} + \frac{\gamma}{2} \|By - b\|^2\\
&\stackrel{(\ref{prop:relaxedADMM:eq:4})}{=} \argmin_{y \in \cH_2} g(y) - \dotp{w_{d_g}^k - \gamma Ax^k - \gamma(2\lambda_k - 1)(Ax^k+By^k - b),By - b} + \frac{\gamma}{2} \|By - b\|^2 \\
&=\argmin_{y \in \cH_2} g(y) - \dotp{w_{d_g}^{k},Ax^{k} +  By - b} + \frac{\gamma}{2} \|Ax^k + By - b + (2\lambda_k - 1)(Ax^{k} + By^{k} -b) \|^2.\numberthis \label{prop:relaxedADMM:eq:5}
\end{align*}
The result then follows from Equations~(\ref{prop:relaxedADMM:eq:big}), (\ref{prop:relaxedADMM:eq:2}), (\ref{prop:relaxedADMM:eq:3}), and (\ref{prop:relaxedADMM:eq:5}), combined with the initial conditions listed in the statement of the proposition. In particular, note that the updates of $x, y, w_{d_f},$ and $w_{d_g}$ do not explicitly depend on $z$
\qed\end{proof}

The following proposition will help us derive primal fundamental inequalities akin to Proposition~\ref{prop:DRSupper} and~\ref{prop:DRSlower}.

\begin{proposition}\label{prop:ADMMdualtoprimalconversion}
Suppose that $(z^j)_{j \geq 0}$ is generated by Algorithm~\ref{alg:ADMM}. Let $z^\ast$ be a fixed point of $\TPRS$ and let $w^\ast = \prox_{\gamma d_f}(z^\ast)$. Then the following identity holds:
\begin{align*}
4\gamma \lambda_k (f(x^k) + g(y^k) - f(x^\ast) - g(y^\ast)) &= -4\gamma \lambda_k (d_f(w_{d_f}^k) + d_g(w_{d_g}^k) - d_f(w^\ast) - d_g(w^\ast))  \\
&+ \left(  2\left(1 - \frac{1}{2\lambda_k}\right) \|z^{k} - z^{k+1}\|^2 + 2\dotp{ z^k - z^{k+1}, z^{k+1}} \right).\numberthis \label{prop:ADMMdualtoprimalconversion:eq:main2}
\end{align*}
\end{proposition}
\begin{proof}
We have the following subgradient inclusions from Proposition~\ref{prop:relaxedADMM}: $A^\ast w_{d_f}^k \in \partial f(x^k)$ and $B^\ast w_{d_g}^k \in \partial g(y^k).$ From the Fenchel-Young inequality \cite[Proposition 16.9]{bauschke2011convex} we have the expression for $f$ and $g$:
\begin{align*}
d_f(w_{d_f}^k) = \dotp{ A^\ast w_{d_f}^k, x^k} - f(x^k) && \mathrm{and} && 
d_f(w_{d_g}^k) = \dotp{ B^\ast w_{d_g}^k, y^k} - g(y^k) - \dotp{w_{d_g}^k, b}.
\end{align*}
Therefore,
\begin{align*}
 -d_f(w_{d_f}^k) - d_g(w_{d_g}^k) &= f(x^k) + g(y^k) - \dotp{ Ax^k + By^k - b, w_{d_f}^k} - \dotp{w_{d_g}^k - w_{d_f}^k, By^k - b}.
\end{align*}
Let us simplify this bound with an identity from Proposition~\ref{prop:relaxedADMM}: from $w_{d_f}^k - w_{d_g}^k = -\gamma(Ax^k + By^k - b),$ it follows that
\begin{align}\label{prop:ADMMdualtoprimalconversion:eq:2}
 -d_f(w_{d_f}^k) - d_g(w_{d_g}^k) &= f(x^k) + g(y^k) +\frac{1}{\gamma} \dotp{w_{d_f}^k - w_{d_g}^k, w_{d_f}^k + \gamma (By^k - b)}.
\end{align}
Recall that $\gamma(By^k - b) = z^k - w_{d_g}^k$.  Therefore $$w_{d_f}^k + \gamma (By^k - b) = z^k + (w_{d_f}^k - w_{d_g}^k) = z^k + \frac{1}{2\lambda_k}(z^{k+1} - z^k) = \frac{1}{2\lambda_k} (2\lambda_k - 1) (z^{k} - z^{k+1}) + z^{k+1},$$ and the inner product term can be simplified as follows:
\begin{align*}
\frac{1}{\gamma} \dotp{w_{d_f}^k - w_{d_g}^k, w_{d_f}^k + \gamma (By^k - b)} &= \frac{1}{\gamma}\dotp{\frac{1}{2\lambda_k}(z^{k+1} - z^{k}), \frac{1}{2\lambda_k}(2\lambda_k - 1)(z^k - z^{k+1})} \\
& +  \frac{1}{\gamma}\dotp{\frac{1}{2\lambda_k}(z^{k+1} - z^{k}), z^{k+1}} \\
&=- \frac{1}{2\gamma\lambda_k} \left(1 - \frac{1}{2\lambda_k}\right) \|z^{k+1} - z^k\|^2 \\
&-  \frac{1}{2\gamma\lambda_k}\dotp{z^k - z^{k+1}, z^{k+1}}. \numberthis \label{prop:ADMMdualtoprimalconversion:eq:3}
\end{align*}

Now we derive an expression for the dual objective at a dual optimal $w^\ast$. First, if $z^\ast$ is a fixed point of $\TPRS$, then $0 = \TPRS(z^\ast) - z^\ast = 2(w_{d_g}^\ast - w_{d_f}^\ast)  = -2\gamma(Ax^\ast + By^\ast - b)$. Thus, from Equation~\eqref{prop:ADMMdualtoprimalconversion:eq:2} with $k$ replaced by $\ast$, we get 
\begin{align}\label{prop:ADMMdualtoprimalconversion:eq:4}
-d_f(w^\ast) - d_g(w^\ast) &= f(x^\ast) + g(y^\ast) + \dotp{Ax^\ast + Bx^\ast - b, w^\ast} = f(x^\ast) + g(y^\ast).
\end{align}

Therefore, Equation~\eqref{prop:ADMMdualtoprimalconversion:eq:main2} follows by subtracting \eqref{prop:ADMMdualtoprimalconversion:eq:4} from Equation~\eqref{prop:ADMMdualtoprimalconversion:eq:2}, rearranging and using the identity in Equation~\eqref{prop:ADMMdualtoprimalconversion:eq:3}.
\qed\end{proof}

\begin{proof}[Proposition~\ref{prop:ADMMupper}]
The fundamental lower inequality in Proposition~\ref{prop:DRSlower} applied to $d_f + d_g$ shows that
\begin{align*}
-4\gamma \lambda_k (d_f(w_{d_f}^k) + d_g(w_{d_g}^k) - d_f(w^\ast) - d_g(w^\ast)) &\leq 2\dotp{z^{k+1} - z^k, z^\ast - w^\ast}.
\end{align*}
The proof then follows from Proposition~\ref{prop:ADMMdualtoprimalconversion}, and the simplification:
\begin{align*}
2\dotp{z^{k} - z^{k+1}, z^{k+1} - (z^\ast - w^\ast)} &+  2\left(1- \frac{1}{2\lambda_k} \right) \|z^{k} - z^{k+1}\|^2\\
&= \|z^k - (z^\ast - w^\ast)\|^2 - \|z^{k+1} - (z^\ast - w^\ast)\|^2  + \left(1- \frac{1}{\lambda_k} \right) \|z^{k} - z^{k+1}\|^2. \qquad \qed
\end{align*}
\end{proof}

\end{document}